\documentclass[a4paper]{amsart}

\usepackage{amssymb, enumerate}
\usepackage[all]{xy}
\usepackage{hyperref,aliascnt}

\numberwithin{equation}{section}

\newtheorem{lma}{Lemma}[section]

\newaliascnt{thmCt}{lma}
\newtheorem{thm}[thmCt]{Theorem}
\aliascntresetthe{thmCt}

\newaliascnt{corCt}{lma}
\newtheorem{cor}[corCt]{Corollary}
\aliascntresetthe{corCt}

\newaliascnt{prpCt}{lma}
\newtheorem{prp}[prpCt]{Proposition}
\aliascntresetthe{prpCt}

\newcounter{theoremintro}

\newaliascnt{thmIntroCt}{theoremintro}
\newtheorem{thmIntro}[thmIntroCt]{Theorem}
\aliascntresetthe{thmIntroCt}

\newtheorem{corIntro}[theoremintro]{Corollary}

\theoremstyle{definition}

\newaliascnt{pgrCt}{lma}
\newtheorem{pgr}[pgrCt]{}
\aliascntresetthe{pgrCt}

\newaliascnt{dfnCt}{lma}
\newtheorem{dfn}[dfnCt]{Definition}
\aliascntresetthe{dfnCt}

\newaliascnt{rmkCt}{lma}
\newtheorem{rmk}[rmkCt]{Remark}
\aliascntresetthe{rmkCt}

\newaliascnt{pbmCt}{lma}
\newtheorem{pbm}[pbmCt]{Problem}
\aliascntresetthe{pbmCt}

\newaliascnt{exaCt}{lma}
\newtheorem{exa}[exaCt]{Example}
\aliascntresetthe{exaCt}

\newaliascnt{ntnCt}{lma}
\newtheorem{ntn}[ntnCt]{Notation}
\aliascntresetthe{ntnCt}

\newcommand{\vect}[1]{\textbf{#1}}
\newcommand{\NN}{\mathbb{N}}
\newcommand{\RR}{\mathbb{R}}
\newcommand{\CC}{\mathbb{C}}
\newcommand{\KK}{\mathcal{K}}
\newcommand{\freeVar}{\_\,}
\newcommand{\sa}{\mathrm{sa}}
\newcommand{\ca}{C*-al\-ge\-bra}
\newcommand{\axiomO}[1]{(O#1)}
\newcommand{\cc}{\mathrm{c}}
\newcommand{\CuSgp}{$\CatCu$-sem\-i\-group}
\newcommand{\CuMor}{$\CatCu$-mor\-phism}
\newcommand{\llpw}{\ll_{\mathrm{pw}}}
\newcommand{\filter}{\mathcal{U}}
\newcommand{\CatPomProd}{\mathrm{PoM}\text{-}\!\prod}
\newcommand{\CatCu}{\ensuremath{\mathrm{Cu}}}
\newcommand{\andSep}{\,\,\,\text{ and }\,\,\,}

\DeclareMathOperator{\Aff}{Aff}
\DeclareMathOperator{\Cu}{Cu}
\DeclareMathOperator{\Lsc}{Lsc}
\DeclareMathOperator{\LL}{L}
\DeclareMathOperator{\supp}{supp}
\DeclareMathOperator{\QT}{QT}
\DeclareMathOperator{\TT}{T}
\DeclareMathOperator{\FF}{F}
\DeclareMathOperator{\PP}{P}
\DeclareMathOperator{\set}{set}
\DeclareMathOperator{\rc}{rc}
\DeclareMathOperator{\CompAmpl}{ca}
\DeclareMathOperator{\LimF}{LimF}
\DeclareMathOperator{\LimT}{LimT}
\DeclareMathOperator{\LimQT}{LimQT}
\DeclareMathOperator{\LDF}{LDF}
\DeclareMathOperator{\DF}{DF}

\begin{document}

\title{Traces on ultrapowers of C*-algebras}

\author{Ramon Antoine}
\author{Francesc Perera}
\author{Leonel Robert}
\author{Hannes Thiel}

\date{\today}

\address{
R.~Antoine, 
Departament de Matem\`{a}tiques,
Universitat Aut\`{o}noma de Barcelona,
08193 Bellaterra, Barcelona, Spain, and
Centre de Recerca Matem\`atica, Edifici Cc, Campus de Bellaterra, 08193 Cerdanyola del Vall\`es, Barcelona, Spain}
\email[]{ramon.antoine@uab.cat}

\address{
F.~Perera, 
Departament de Matem\`{a}tiques,
Universitat Aut\`{o}noma de Barcelona,
\linebreak 08193 Bellaterra, Barcelona, Spain, and
Centre de Recerca Matem\`atica, Edifici Cc, Campus de Bellaterra,  08193 Cerdanyola del Vall\`es, Barcelona, Spain}
\email[]{francesc.perera@uab.cat}
\urladdr{https://mat.uab.cat/web/perera}

\address{
L.~Robert,
University of Louisiana at Lafayette,
Lafayette, 70504-3568, USA}
\email[]{lrobert@louisiana.edu}

\address{Hannes~Thiel, 
Department of Mathematical Sciences, Chalmers University of Technology and the University of
Gothenburg, Gothenburg 412 96, Sweden}
\email{hannes.thiel@chalmers.se}
\urladdr{www.hannesthiel.org}

\subjclass[2020]
{Primary
19K14, 
46L05, 
46B08, 
46M07. 
Secondary
06B35, 
06F05, 
46L30, 
46L35, 
46L51, 
46L80. 
}

\keywords{ultraproducts, traces, quasitraces, Cuntz semigroup, $C^*$-algebra}

\thanks{
The two first named authors were partially supported by MINECO (grants No.\  MTM2017-83487-P and PID2020-113047GB-I00), by the Comissionat per Universitats i Recerca de la Ge\-ne\-ralitat de Ca\-ta\-lu\-nya (grant No.\ 2017-SGR-1725) and by the Spanish State Research Agency through the Severo Ochoa and María de Maeztu Program for Centers and Units of Excellence in R\&D (CEX2020-001084-M).
The fourth named author was partially supported by the Deutsche Forschungsgemeinschaft (DFG, German Research Foundation) under Germany's Excellence Strategy EXC 2044-390685587 (Mathematics M\"{u}nster: Dynamics-Geometry-Structure), by the ERC Consolidator Grant No. 681207, and by the Knut and Alice Wallenberg Foundation (KAW 2021.0140).
}

\begin{abstract}
Using Cuntz semigroup techniques, we characterize when limit traces are dense in the space of all traces on a free ultrapower of a \ca{}. 
More generally, we consider density of limit quasitraces on ultraproducts of \ca{s}.

Quite unexpectedly, we obtain as an application that every simple \ca{} that is $(m,n)$-pure in the sense of Winter is already pure.
As another application, we provide a partial verification of the first Blackadar--Handelman conjecture on dimension functions.

Crucial ingredients in our proof are new Hahn--Banach type separation theorems for noncancellative cones, which in particular apply to the cone of extended-valued traces on a \ca.
\end{abstract}

\maketitle

\section{Introduction}
\label{sec:intro}

Ultraproducts have a well established presence in the field of operator algebras, starting with the groundbreaking work of McDuff \cite{McD69UncountablyManyII1} on tensorial absorption of the hyperfinite $\mathrm{II}_1$ factor and later the award-winning results of Connes \cite{Con76ClassifInjFactors} on the classification of injective factors.
Ultraproducts also play a crucial role in model theory of \ca{s}
\cite{GarLup18ApplModelThyCDynamics, GarKalLup19ModelThyRokhlinDim, FarHarLupRobTikVigWin21ModelThy}.

In recent years, work on the stably finite case of Elliott's classification program and on the Toms--Winter conjecture has drawn attention to the trace space of the free ultrapower of a \ca{} (see, for example, \cite{KirRor14CentralSeq}, \cite{TikWhiWin17QDNuclear}, and \cite{BBSTWW19}). 
A natural question arises in this context: To what extent is the trace space of a free ultrapower of a \ca{} determined by the trace space of the \ca? 
The same question can be asked more generally for trace spaces of products and ultraproducts of \ca{s}.

Since there are various notions of trace associated to a \ca{}, to make the above question more precise we must fix what is meant by trace space. 
Given a unital \ca{} $A$, let us consider first the set $\TT_1(A)$ of tracial states on $A$ regarded as a compact convex set (embedded in $A^*$ and endowed with the weak* topology). 
Let $\filter$ be a free ultrafilter on $\NN$, and let $A_\filter$ denote the free ultrapower of~$A$ with respect to $\mathcal U$. 
The category of compact convex sets admits ultracoproducts, and there is a naturally defined continuous affine map
\begin{equation}
\label{intro:coprod}
\coprod_\filter \TT_1(A) \to \TT_1(A_\filter).
\end{equation}

The question of calculating $\TT_1(A_\filter)$ can be made explicit by asking whether this map is an isomorphism. 
The map \eqref{intro:coprod} is always injective, but it may fail in general to be surjective. 
For instance, it is possible for  $\TT_1(A)$ to be  a singleton set while $\TT_1(A_\filter)$ is not. 
For the  ultraproduct of an arbitrary collection of unital \ca{s} $(A_j)_{j\in J}$, a similar map $\coprod_\filter \TT_1(A_j) \to \TT_1(\prod_\filter A_j)$ can be defined. 
Again, this map is guaranteed to be injective (\autoref{autoinj}), while its surjectivity may fail even more dramatically:
we can have $\TT_1(A_j)$ empty for all~$j$ while  $\TT(\prod_\filter A_j)$ is nonempty;
see Examples~\ref{exa:1} and~\ref{exa:2} below and \cite[Corollary~2.2]{BicFar15Trace}. 

The range of the map \eqref{intro:coprod} can be described as the closure of the set $\LimT_1(A_\filter)$ of limit tracial states in $\TT_1(A_\filter)$, where a limit tracial state is an element in $\TT_1(A_\filter)$ that is the limit along $\filter$ of a sequence of tracial states in $\TT_1(A)$.
Surjectivity in \eqref{intro:coprod} can thus be alternately stated as the density of $\LimT_1(A_\filter)$ in $\TT_1(A_\filter)$. 

Stated in this form, surjectivity of the map in \eqref{intro:coprod} has been obtained in the literature under different kinds of  ``regularity properties'' for the \ca{s}. 
For ultrapowers of an exact \ca{} that tensorially absorbs the Jiang--Su algebra, the density of $\LimT_1(A_\filter)$ was proved by Ozawa in \cite[Theorem~8]{Oza13DixmApproxSymmAmen}.  
This was extended by Ng and the third named author to ultraproducts of unital \ca{s} with the property of strict comparison of full positive elements by bounded traces (\cite[Theorem~1.2]{NgRob16CommutatorsPureCa}).
In the context of products of unital \ca{s}, Archbold, Tikuisis, and the third named author showed in \cite[Theorem~3.19]{ArcRobTik17Dixmier} that the map $\coprod_j \TT_1(A_j)\to \TT_1(\prod_j A_j)$ is surjective if the involved \ca{s} have uniformly bounded radius of comparison by traces. 
Very recently, analogous questions have also been considered in the context of ultraproducts of $W^*$-bundles by Vaccaro \cite{Vac23arX:UltraprodWBundles}.

In order to investigate exactly which regularity properties on $A$ guarantee an isomorphism in \eqref{intro:coprod}, we use the theory of the Cuntz semigroup. 
For this reason, our results are more naturally phrased in terms of spaces of $2$-quasitraces, since $2$-quasitracial states are in bijection with functionals on the Cuntz semigroup normalized at
the class of the unit (\cite[Theorem~4.4]{EllRobSan11Cone}). 
To pass from $2$-quasitraces to traces, one can add the assumption of exactness to the \ca{s}, and invoke Haagerup's theorem asserting that for exact \ca{s} ``$2$-quasitraces are traces'' (\cite{Haa14Quasitraces}), or simply assume that $2$-quasitraces are traces for the \ca{s} in the collection.  

Let us  introduce some notation.
Let $A$ be a \ca. 
Let $\Cu(A)$ denote its Cuntz semigroup, that is, the set of Cuntz classes of positive elements in $A\otimes \mathcal K$ (the stabilization of $A$) endowed with a suitable order and addition operation. 
We denote the Cuntz class of a positive element $a\in A\otimes \mathcal K$ by $[a]$;
see \autoref{pgr:axioms}.
 
Given $N\in \NN$, we define the relation $\leq_N$ on $\Cu(A)$ by setting $x\leq_N y$ if $nx\leq ny$ for all $n\geq N$;
see \autoref{ntn:leqN}.
Suppose now that $A$ is unital. 
Let $\QT_1(A)$ denote the set of $2$-quasitracial states of $A$. 
The \emph{rank} of a Cuntz class $[a]\in \Cu(A)$ is defined as the function $\widehat{[a]}\colon \QT_1(A)\to [0,\infty]$ given by $\widehat{[a]}(\tau)=\lim_n \tau(a^{\frac 1n})$, where~$\tau$ is canonically extended to $A\otimes \mathcal K$ to be evaluated at $a^{\frac 1n}$.
Finally, for an ultrapower $A_\filter$ of $A$, let $\LimQT_1(A_\filter)$ denote the set of limit $2$-quasitracial states on $A_\filter$, defined similarly as for traces.

We write $[x,y]$ for the additive commutator $xy-yx$ of $x,y \in A$, and we let $[A,A]$ denote the linear span of commutators in $A$.
The next theorem is our main result on the calculation of the tracial state space of an ultrapower.
The equivalence of~(i), (ii) and~(iii) follows from \autoref{prp:CharFuDenseUpower} in combination with \autoref{QTFbijection}, while the equivalence with~(iv) and~(v) follows from \autoref{nosillycomms}.

\begin{thmIntro}
\label{thmtraces}
Let $A$ be a unital \ca{} and let $\filter$ be a free ultrafilter on~$\NN$. 
The following are equivalent:
\begin{enumerate}[{\rm (i)}]
\item
The set $\LimQT_1(A_\filter)$ of limit $2$-quasitracial states is dense in $\QT_1(A_\filter)$.
\item
For every $\gamma\in (0,1)$ and $d\in\NN$ there exists $N\in\NN$ such that
\[
\widehat{x}\leq\gamma\widehat{y}
\quad \text{ implies } \quad 
x\leq_N y,
\]
for all $x,y\in \Cu(A)$ such that $x,y\leq d[1]$ and $[1]\leq dy$.

\item
There exists $M\in \NN$ such that for every $d\in\NN$ there exists $N\in\NN$  such that 
\[
\widehat{x} \leq \widehat{y}
\quad \text{ implies } \quad 
Nx\leq  NMy,
\]
for all $x,y\in \Cu(A)$ such that  $x,y\leq d[1]$ and $[1]\leq dy$.
\end{enumerate}
If we assume at the outset that $\QT_1(A)=\TT_1(A)$ (for example, if $A$ is exact), then the above statements are also equivalent to
\begin{enumerate}
\item[{\rm (iv)}]
There exists $N\in \NN$ such that if $a\in \overline{[A,A]}$, then
\[
\left\| a-\sum_{k=1}^N [b_k,c_k] \right\| \leq \frac{1}{2}\|a\|
\] 
for some $b_k,c_k\in A$ such that $\|b_k\|,\|c_k\|\leq \|a\|^{\frac12}$ for all $k = 1,\ldots,N$.
\item[{\rm (v)}]
The natural map $\coprod_\filter \TT_1(A) \to \TT_1(A_\filter)$ is an isomorphism.
\end{enumerate}
\end{thmIntro}

The above theorem subsumes all the existing results on the density of limit tracial states on an ultrapower.

Using that tracial states form a closed subset among $2$-quasitracial states, we obtain:

\begin{corIntro}
Let $A$ be a unital \ca{} such that every $2$-quasitracial state on~$A$ is a trace. 
If $\Cu(A)$ satisfies the equivalent conditions of \autoref{thmtraces}, then every $2$-quasitracial state on~$A_\filter$ is a trace.
\end{corIntro}

Let us discuss now a different trace space associated to a \ca{} $A$ (possibly nonunital).
Let $\TT(A)$ denote the set of $[0,\infty]$-valued, lower semicontinuous traces on~$A$;
see \autoref{pgr:quasitraces}.
We regard $\TT(A)$ as a cone endowed with pointwise addition and pointwise scalar multiplication by positive real numbers. 
The cone $\TT(A)$ is always noncancellative, so it does not embed into a vector space. 
There is, however, a compact Hausdorff topology on $\TT(A)$ compatible with the cone operations; see \cite{EllRobSan11Cone}.

Consider an ultrapower $A_\filter$ of $A$. 
A sequence of traces in $\TT(A)$ naturally defines a \emph{limit trace} in $\TT(A_\filter)$. 
We denote the set of limit traces on $A_\filter$ arising this way by $\LimT(A_\filter)$.  
The central question that we address in the context of the trace space $\TT(A_\filter)$ is that of characterizing, through comparibility properties on the Cuntz semigroup, the density of $\LimT(A_\filter)$ in $\TT(A_\filter)$.  
We also investigate this density question for the cones of traces of products and ultraproducts of arbitrary collections of \ca{s}. 
Although these density questions can be reformulated in terms of the surjectivity of maps with domain a coproduct or ultracoproduct of cones, we shall refrain from formulating them in this way. 
This point of view will be pursued in a separate work. 

As in the case of tracial states, our results are more naturally formulated in terms of the cone $\QT(A)$ of $[0,\infty]$-valued, lower semicontinuous $2$-quasitraces.
By the homeomorphism between $\QT(A)$ and the cone $\FF(\Cu(A))$ of functionals on $\Cu(A)$, the problem of density of $\LimQT(A_\filter)$ in $\QT(A_\filter)$ is translated into the density of a set of limit functionals in $\FF(\Cu(A_\filter))$.
Moreover, in this setting the problem admits a formulation working purely in the category $\mathrm{Cu}$ of abstract Cuntz semigroups, usually called \CuSgp{s};
see \autoref{pbm:denseLimitFctl}.

The category $\mathrm{Cu}$ was introduced in \cite{CowEllIva08CuInv} and was extensively studied in
\cite{AntPerThi18TensorProdCu, AntPerPet18PerfConditionsCu, AntPerThi20AbsBivariantCu, AntPerThi20AbsBivarII, AntPerThi20CuntzUltraproducts, AntPerRobThi22CuntzSR1} as well as \cite{ThiVil22DimCu, ThiVil21DimCu2, ThiVil21arX:NowhereScattered}. 
The cones of functionals on \CuSgp{s} have also been thoroughly studied; 
see, for example, \cite{EllRobSan11Cone, Rob13Cone, AntPerRobThi21Edwards}.
This allows us to use functional analytic techniques developed for the category $\mathrm{Cu}$ together with the computation of Cuntz semigroups of ultraproducts, as carried out in \cite{AntPerThi20CuntzUltraproducts}.
We thus show that the density of limit $2$-quasitraces on an ultrapower of a \ca{} is equivalent to a certain comparability property on the Cuntz semigroup of the algebra. 
These results are obtained as corollaries of their counterparts on functionals on abstract Cuntz semigroups, and solve the original problem under the not uncommon assumption that $2$-quasitraces are traces, and in particular if the algebra is exact.

We now explicitly state the results. Let $A$ be a \ca, and let $\tau\in \QT(A)$, which we regard as a $2$-quasitrace defined on $(A\otimes\mathcal K)_+$. 
For $a\in(A\otimes\KK)_+$, we define the \emph{rank} of~$a$ as the function $\widehat{[a]}\colon \QT(A)\to [0,\infty]$ given by $\widehat{[a]}(\tau)=\lim_n \tau(a^{\frac 1n})$, as we did in the case of quasitracial states. 
Given a free ultrafilter $\filter$ on $\NN$, as before we have that every sequence in $\QT(A)$ naturally induces a limit $2$-quasitrace in $\QT(A_\filter)$ and we denote the set of limit $2$-quasitraces by $\LimQT(A_\filter)$. 
Our main result characterizing the density of limit $2$-quasitraces of an ultrapower is the following:

\begin{thmIntro}[{\ref{prp:MainThm}}]
\label{thm-intro}
Let $A$ be a \ca{} and let $\mathcal{U}$ be a free ultrafilter on~$\NN$. 
The following are equivalent:
\begin{enumerate}[{\rm (i)}]
\item
The set of limit $2$-quasitraces $\LimQT(A_\filter)$ is dense in $\QT(A_\filter)$.
\item
For every $\gamma\in(0,1)$ and $d\in\NN$ there exists $N\in\NN$ such that 
\[
\widehat{[a]} \leq \gamma \widehat{[b]} 
\quad \text{ implies } \quad 
[a]\leq_N [b], \quad \text{ for all }a,b\in M_d(A)_+.\]
\item 
There exists $M\in\NN$ such that for every $d\in\NN$ there exists $N\in\NN$ such that 
\[
\widehat{[a]} \leq \widehat{[b]}
\quad \text{ implies } \quad 
N[a]\leq MN[b], \quad \text{ for all }a,b\in M_d(A)_+.
\]
\end{enumerate}
\end{thmIntro}

A rather unexpected corollary of our results is the equivalence of different kinds of comparability properties in the Cuntz semigroup of a \ca{} as is evidenced, for example, from (ii) and (iii) in \autoref{thm-intro}.
We do not know a direct proof of this equivalence that avoids the use of ultrapowers.

In \cite{Win12NuclDimZstable}, Winter defines a \ca{} to be $(m,n)$\emph{-pure} provided it satisfies certain comparability and divisibility properties, called $m$-comparison and $n$-divisibility;
see Paragraphs~\ref{pgr:comparison} and~\ref{pgr:divisibility}.
A \ca{} is said to be \emph{pure} if it is $(0,0)$-pure, which by definition means that its Cuntz semigroup is almost unperforated and almost divisible.
The relevance of purity resides in Winter's theorem \cite{Win12NuclDimZstable} showing that $(m,n)$-pure, unital, simple, separable \ca{s} with locally finite nuclear dimension are $\mathcal{Z}$-stable (an important regularity property) and thus pure by \cite{Ror04StableRealRankZ}.
We generalize this consequence of Winter's theorem to general simple \ca{s}: 

\begin{thmIntro}[{\ref{pureReduction}}]
\label{thm-introC}
A simple $(m,n)$-pure \ca{} is pure.
\end{thmIntro}

In the course of our investigations we obtain  a partial confirmation of a conjecture by Blackadar and Handelman (\cite{BlaHan82DimFct}), which we proceed to recall.
The classical Cuntz semigroup $W(A)$ of a \ca{} $A$ is the subsemigroup of $\Cu(A)$ consisting of the Cuntz classes of positive elements in $A \otimes M_n(\CC) \subseteq A\otimes\KK$, for $n\in\NN$.
A dimension function on a unital \ca{}~$A$ is a normalized state on~$W(A)$, and the set of dimension functions is denoted by $\DF(A)$. 
The subset $\LDF(A)$ of  lower semicontinuous dimension functions is of special importance, as by results from \cite{BlaHan82DimFct} these are in natural bijection with the set $\QT_1(A)$ of normalized quasitraces. 
Blackadar and Handelman conjectured that $\LDF(A)$ is always dense in $\DF(A)$, and confirmed this in the commutative case. 
The simple, exact $\mathcal{Z}$-stable case was established in \cite{BroPerTom08CuElliottConj}, and the case of \ca{s}   with finite radius of comparison  in \cite{Sil16Thesis}. 
We show here  that $\LDF(A)$ is dense in $\DF(A)$ whenever $A$ is a unital  \ca{} such that $\LimQT_1(A_\filter)$ is dense in $\QT_1(A_\filter)$ (see \autoref{BHdensitycor}). This result suggests that the Blackadar--Handelman conjecture might be false in general, but an example seems difficult to come by.

%
%
%
%

The central results obtained on cones of functionals, which may well be of independent interest, are separation results \`a la Hahn--Banach that allow us to characterize when a subcone of functionals is dense; 
see \autoref{sec:separateFS}. 
Moreover, in this context, and under mild additional assumptions (that are satisfied by the Cuntz semigroups of any \ca{}, \cite{AntPerRobThi21Edwards}), we are able to obtain even stronger separation results; see \autoref{sec:strongerSeparateFS}.
We discuss these results in the appendix in order not to disturb the flow of the presentation.

Here is a brief outline of the paper. 
Throughout, we will largely focus on ultrapowers and ultraproducts. 
When similar results hold for products we make some brief remarks on how the methods can be adapted to that case. 
In \autoref{sec:commutators} we review the main ideas relating limit tracial states and commutators, present some motivating examples, and prove the second part of \autoref{thmtraces}. 
\autoref{sec:prelim} contains the necessary preliminaries on the Cuntz semigroup, its functionals, and  (quasi)traces on \ca{s}.
In \autoref{sec:ultrapowers} we review the construction of ultraproducts of abstract Cuntz semigroups as well as their relation to ultraproducts of \ca{s}.
In \autoref{sec:main}, we use the new Hahn--Banach theorems from the appendices to give a characterization of density of limit functionals in an ultrapower of an abstract Cuntz semigroup in terms of comparability conditions. 
In \autoref{sec:LCBA} we introduce the notion of Locally Bounded Comparison Amplitude and view its importance in connection with the density results of the previous section.
By translating the density characterization to the setting of \ca{s} we prove \autoref{thm-intro} in \autoref{sec:StrongerDensity}.
In \autoref{sec:BH} we prove the first part of \autoref{thmtraces} and study the Blackadar--Handelman conjecture.
In \autoref{sec:pure} we analyze the pureness of simple \ca{s} and prove \autoref{thm-introC}.

\section{Tracial states on products and ultraproducts}
\label{sec:commutators}

The main result of this section, \autoref{nosillycomms}, characterizes in multiple ways the density of limit tracial states on an ultraproduct of unital \ca{s}. 
This includes the equivalence of (i), (iv) and~(v) in \autoref{thmtraces} from the introduction under the assumption that $\QT_1(A)=\TT_1(A)$.
Parts of this result are well known to experts in the area, although it has not been previously stated in the form given below  (see \cite[Section~2]{BicFar15Trace}, \cite[Theorem~8]{Oza13DixmApproxSymmAmen},  \cite[Proposition~2.3]{NgRob16CommutatorsPureCa}, \cite[Section~3.5]{FarHarLupRobTikVigWin21ModelThy}). 
At the end of the section we give two examples of ultraproducts of \ca{s} where the density of limit tracial states fails to hold. 

\medskip

Throughout this section we assume that~$A$ is a unital \ca. 
We denote by~$A_\sa$ the set of selfadjoint elements of~$A$.
Let $\TT_1(A)$ denote the set of tracial states of~$A$ endowed with the weak* topology. 
Given $a\in A_\sa$, define $\widehat a\colon \TT_1(A)\to \RR$ by $\widehat a(\tau)=\tau(a)$ for all $\tau\in \TT_1(A)$. 
Let $\Aff(\TT_1(A))$ denote Banach space of $\RR$-valued, continuous, affine functions on $\TT_1(A)$, equipped with the supremum norm.
Observe that $\widehat a\in \Aff(\TT_1(A))$,
for $a\in A_\sa$.

Let $[A,A]$ denote the linear span of the set of commutators $\{[x,y]:x,y\in A\}$, where $[x,y]:=xy-yx$) in $A$. 
We form the quotient $A_\sa / (A_\sa \cap \overline{[A,A]})$, which we regard as a real Banach space under the quotient norm.

The following lemma is well known.

\begin{lma}
\label{distance2comms}
The real Banach spaces $A_\sa / (A_\sa\cap \overline{[A,A]})$ and $\Aff(\TT_1(A))$  are isomorphic via the map $a+\overline{[A,A]}\mapsto \widehat a$.
\end{lma}
\begin{proof}
Surjectivity is a well-known consequence of Kadison's function representation theorem;
see \cite[Theorem~II.1.8]{Alf71CpctCvxSets} and \cite[Section~3.10]{Ped79CAlgsAutGp}.
That the map $a+\overline{[A,A]}\mapsto \widehat a$ is isometric is proven in the proof of \cite[Lemma~3.1]{Tho95TracesCrProdZ}, and also in \cite[Theorem~5]{Oza13DixmApproxSymmAmen}.
The case of a positive $a$ is also obtained in \cite[Theorem~2.9]{CunPed79EquivTraces}. 
(Note that the subspace~$A_0$ considered in \cite{CunPed79EquivTraces} and in \cite[Lemma~3.1]{Tho95TracesCrProdZ} agrees with $A_\sa \cap \overline{[A,A]}$.) 
\end{proof}

We will use below the following lemma.

\begin{lma}
\label{easyinduction}
Let $N\in \NN$ and let $A$ be a \ca{} with the property that for all $a\in  \overline{[A,A]}$ there exist $b_k,c_k\in A$ with $\|b_k\|,\|c_k\|\leq \|a\|^{\frac12}$ for $k=1,\ldots,N$ such that 
\[
\left\| a-\sum_{k=1}^N [b_k,c_k] \right\| \leq \frac12\|a\|.
\]

Then for all $m\in \NN$ and all $a\in \overline{[A,A]}$ there exist $b_k,c_k\in A$ with $\|b_k\|,\|c_k\|\leq \|a\|^{\frac12}$ for $k=1,\ldots,mN$ such that
\[
\left\| a-\sum_{k=1}^{mN} [b_k,c_k] \right\| \leq \frac1{2^m}\|a\|.
\]
\end{lma}
\begin{proof}
Given $a\in \overline{[A,A]}$, the element $a_1=a-\sum_{k=1}^N [b_k,c_k]$ is again in $\overline{[A,A]}$, and has norm $\|a_1\| \leq \frac 12\|a\|$.
A straightforward induction yields the desired result.
\end{proof}

\begin{pgr}[Compact convex sets and complete order unit vector spaces]
Let us recall briefly the duality between compact convex sets and complete order unit vector spaces. 
We refer the reader to \cite{Alf71CpctCvxSets} for further details.
Given a compact, convex set $K$, let $\Aff(K)$ denote the vectors space of $\RR$-valued, continuous, affine functions on $K$. 
We regard $\Aff(K)$ as an ordered vector space endowed with the pointwise order and with order unit the constant function $1$.
The norm induced by the order unit is the supremum norm, and thus $\Aff(K)$ is a complete order unit vector space. 

Given a complete order unit vector space $(V,V_+,e)$, let $S_1(V)\subseteq V^*$ denote the set of states on $V$, that is, $\lambda\in V^*$ such that $\lambda(e)=\|\lambda\|=1$.  
Then $S_1(V)$ is convex and compact when endowed with the weak* topology. 
The constructions of $\Aff(\cdot)$ and $S_1(\cdot)$ extend to morphisms thus yielding functors $\Aff$ and $S_1$ between the categories of complete order unit vector spaces and of compact convex sets. 
The natural isomorphisms $K\to S_1(\Aff(K))$ and $V\to \Aff(S_1(V))$ establish a contravariant duality between these two categories.
\end{pgr}

\begin{pgr}[Ultraproducts of \ca{s}]
\label{pgr:ultraproducts}
Let $(A_j)_{j\in J}$ be a family of \ca{s} and let~$\filter$ be a free ultrafilter on the set $J$. 
Consider the product \ca{} $\prod_j A_j$.
Set
\[
c_\filter\big( (A_j)_j \big) := \Big\{ (a_j)_j \in \prod_j A_j : \lim_{j\to\filter}\|a_j\|=0 \Big\},
\]
which is a (closed, two-sided) ideal in $\prod_j A_j$.
The \emph{ultraproduct} of the family $(A_j)_{j\in J}$ (along $\filter$) is $\prod_\filter A_j := \prod_j A_j / c_\filter((A_j)_j)$. 
In case $A_j=A$ for all $j$, we speak of the \emph{ultrapower} $\prod_\filter A$.  
We denote by $\pi_\filter\colon \prod_j A_j\to\prod_\filter A_j$ the quotient map.
\end{pgr}

\begin{pgr}
The category of complete order unit vector spaces admits products and ultraproducts: 
given $(V_j,(V_j)_+,e_j)$ for $j\in J$, we form $V=\prod_j V_j$, composed of norm bounded collections $(v_j)_{j\in J}$, and endow it with the coordinatewise order and with order unit $(e_j)_j$. If $\filter$ is an ultrafilter on the set $J$, then passing to the quotient by the subspace
$c_\filter((V_j)_j)= \{(v_j)_j:\lim_\filter \|v_j\|=0\}$ we obtain the ultraproduct~$\prod_\filter V_j$. 

Since the category of complete order unit vector spaces admits products and ultraproducts, the category of compact convex sets admits coproducts and ultra-coproducts.
Given compact convex sets $(K_j)_{j\in J}$ and an ultrafilter $\filter$ on the index set $J$, we denote by $\coprod_j K_j$ and $\coprod_\filter K_j$  their coproduct and ultracoproduct, respectively. 
We can concretely think of these compact convex sets as follows:
\begin{align*}
\coprod_j K_j &=S_1\big( \prod_j \Aff(K_j) \big), \\
\coprod_\filter K_j &= S_1\big( \prod_\filter \Aff(K_j) \big).
\end{align*}
\end{pgr}

Now consider a family of unital \ca{s} $(A_j)_{j\in J}$. 
For each $k\in J$, the projection map $\pi_k\colon \prod_j A_j\to A_k$ induces $\Aff \TT_1(\pi_k)\colon \Aff \TT_1(\prod_j A_j)\to \Aff \TT_1(A_k)$. 
By the universal property of the product, we get a map
\begin{equation}
\label{affprod}
\Aff \TT_1\big( \prod_j A_j \big) \to \prod_j \Aff \TT_1(A_j).
\end{equation}
It is easy to calculate that given a selfadjoint $a=(a_j)_j\in \prod A_j$, the function $\widehat a\in \Aff \TT_1(\prod_j A_j)$ is mapped by the above map to $(\widehat a_j)_j\in \prod_j \Aff \TT_1(A_j)$. 

Let $\filter$ be an ultrafilter on $J$. 
If $a\in c_\filter((A_j)_J)$, then $\lim_\filter \|\widehat a_j\|=0$.
Thus, again we have a map
\begin{equation}
\label{affultra}
\Aff \TT_1\big( \prod_\filter A_j \big) \to \prod_\filter \Aff \TT_1(A_j).
\end{equation}
Applying the functor $S_1(\cdot)$ in \eqref{affprod} and $\eqref{affultra}$ we obtain continuous affine maps
\begin{align}\label{coprod}
\coprod_j \TT_1(A_j)\to \TT_1\big( \prod_j A_j \big), \\
\label{coultra}
\coprod_\filter \TT_1(A_j)\to \TT_1\big( \prod_\filter A_j \big).
\end{align} 

\begin{lma}
\label{autoinj}
Let $(A_j)_{j\in J}$ be a family of unital \ca{s} and let $\filter$ be a free ultrafilter on the index set $J$. 
The following are true:
\begin{enumerate}[\rm (i)]
\item
The maps in \eqref{affprod} and \eqref{affultra} are surjective. 
\item
The maps in \eqref{coprod} and \eqref{coultra} are injective. 
\end{enumerate}
\end{lma}
\begin{proof}
(i) Let $(f_j)_j$ be an element of $\prod_j \Aff \TT_1(A_j)$. 
By the isometric isomorphism of $\Aff \TT_1(A_j)$ with $(A_j)_\sa / ((A_j)_\sa \cap \overline{[A_j,A_j]})$ 
(\autoref{distance2comms}), we can choose for each $j$ an element 
$a_j\in (A_{j})_\sa$ such that $f_j=\widehat a_j$ and $\|a_j\|\leq 3/2\|f_j\|$. Let  $a=(a_j)_j\in \prod_j A_j$.
Then $\widehat a\in \Aff \TT_1(\prod_j A_j)$  is mapped to $(f_j)_j$ by \eqref{affprod}.
This proves surjectivity of this mapping.

Let $f\in \prod_\filter \Aff \TT_1(A_j)$. 
Let $(f_j)_{j}$ be a lift of $f$ in $\prod_j \Aff \TT_1(A_j)$ having norm at most $3/2\|f\|$. 
By the arguments from the previous paragraph, we can choose a selfadjoint 
$a\in \prod_j A_j$
such that $\widehat a$ is mapped to $(f_j)_j$ by \eqref{affprod} and $\|a\|\leq 9/4\|f\|$.  
Let $b\in \prod_\filter A_j$ be the image of $a$ in the ultraproduct. 
Then $\widehat{b}$ is mapped to $f$ by \eqref{affultra}. This proves surjectivity of \eqref{affultra}.

\smallskip

(ii) Injectivity of \eqref{coprod} and \eqref{coultra}  follows at once from the surjectivity of \eqref{affprod} and \eqref{affultra}  
and the definition of the former maps as the functor $S_1(\cdot)$ applied to the latter.
\end{proof}

\begin{pgr}[Limit tracial states]
\label{limittracialstates}
Let us recall the construction of limit tracial states on an ultraproduct of unital \ca{s}. 
Let $(A_j)_{j\in J}$ be a family of unital \ca{s} and let~$\filter$ be a free ultrafilter on the index  set~$J$. 
Let $(\tau_j)_{j\in J}$ be such that $\tau_j\in \TT_1(A_j)$ for all~$j$. 
For each $k\in J$, let $\bar \tau_k$ denote the tracial state on $\prod_{j\in J} A_j$ induced by $\tau_k$ via the projection $\pi_k \colon \prod_{j\in J}A_j\to A_k$. 
The limit 
\[
\bar\tau_\filter=\lim_{\filter}\bar \tau_k
\] 
exists by the compactness of $\TT_1(\prod_{j\in J} A_j)$. 
Moreover, $\bar\tau_\filter$ is easily seen to vanish on the ideal $c_\filter((A_j)_j)$. 
It thus induces a tracial state $\tau_\filter\in \TT_1(\prod_\filter A_j)$. 
The tracial states on $\prod_\filter A_j$ obtained in this way are called \emph{limit tracial states}. 
We denote by $\LimT_1(\prod_\filter A_j)$ the subset of $\TT_1(\prod_\filter A_j)$ of limit tracial states.
\end{pgr}

We use $\mathrm{co}(M)$ to denote the convex hull of a subset $M$ of a convex space.

\begin{lma}
\label{comaprange}
Let $(A_j)_{j\in J}$ be a family of unital \ca{s} and let $\filter$ be a free ultrafilter on the index set $J$. 
The following are true:
\begin{enumerate}[{\rm (i)}]
\item
The range of the map \eqref{coprod} is equal to  $\overline{\mathrm{co}(\bigcup_j \TT_1(A_j))}$ (closure in the weak* topology).
\item
The range of the  map \eqref{coultra} is equal to $\overline{\LimT_1(\prod_\filter A_j)}$ (closure in the weak* topology). 
\end{enumerate}
\end{lma}
\begin{proof}
(i) For each $k\in J$, we have a commutative diagram
\[
\xymatrix{
\TT_1(A_k)\ar[d]\ar[rd] &\\
\coprod_j \TT_1(A_j)\ar[r] & \TT_1(\prod_j A_j)\\
}
\]
where the horizontal arrow is the map from \eqref{coprod}, the vertical arrow maps a trace $\tau\in \TT_1(A_k)$ to the functional $\lambda_\tau\in \coprod_j \TT_1(A_j)$ given by $\lambda_\tau((f_j)_j)=f_k(\tau)$, and the diagonal one maps~$\tau$ to $\bar\tau\in \TT_1(\prod A_j)$ given by $\bar\tau((a_j)_j)=\tau(a_k)$.
Since the range of \eqref{coprod} is closed, as it is the image of a compact set under a continuous map, to complete the proof it will suffice to show that the convex hull of the images of $\{\TT_1(A_k) : k\in J\}$ in $\coprod_j \TT_1(A_j)$ is a dense set in $\coprod_j \TT_1(A_j)$.
 
Suppose that this is not the case. 
Then by Hahn--Banach's separation theorem, there exists $f=(f_j)_{j}\in \prod_j \Aff\TT_1(A_j)$
such that $\lambda_\tau(f)\leq 1$ for all $\tau\in \TT_1(A_k)$ and all $k$, but $\mu(f)>1$ for some $\mu\in \coprod_j \TT_1(A_j)$.
Shifting $f$ by a scalar multiple of the unit and renormalizing (that is, replacing $f$ by $\tfrac{f+t}{1+t}$ for sufficiently large $t\in\RR$), we may assume that $f\geq 0$. 
Then $0\leq \lambda_\tau(f)\leq 1$ for all $\tau\in \TT_1(A_k)$ readily implies that $\|f_j\|\leq 1$ for all~$j$. 
Hence $\|f\|\leq 1$, which contradicts that $\mu(f)\geq 1$.

\smallskip

(ii) A collection of tracial states  $(\tau_j)_j$, with $\tau_j\in \TT_1(A_j)$, induces an element~$\lambda_\filter$ of the coproduct as follows: 
Given $f\in \prod_\filter \Aff\TT_1(A_j)$, choose a lift $(f_j)_j\in  \prod_j \Aff\TT_1(A_j)$, and define
\[
\lambda_\filter(f)=\lim_\filter f_j(\tau_j).
\]
Let us call such a $\lambda_\filter$ a limit state on $\prod_\filter \Aff \TT_1(A_j)$. Limit states are mapped to limit tracial states in $\TT_1(\prod_\filter A_j)$ by the map \eqref{coultra}, with $\lambda_\filter$ as defined above being mapped to the limit tracial state $\tau_\filter$ associated to $(\tau_j)_j$. 
Since  the range of \eqref{coultra} is closed, to complete the proof it will suffice to show that the limit states are dense in $S_1(\prod_\filter \Aff \TT_1(A_j))$.

Supposing that this is not case, we use Hahn--Banach as in (i) to obtain an element $f\in \prod_\filter \Aff \TT_1(A_j)$, with lift $(f_j)_j$, such that $\lambda_\filter(f)\leq 1$ for every limit state, while $\mu(f)>1$ for some $\mu\in \coprod_\filter \TT_1(A_j)$. 
As before, we may assume that $f\geq 0$. 
Let $(f_j)_j$ be a positive lift of $f$ in $\prod_j \Aff \TT_1(A_j)$. 
For each $j$, let $\tau_j$ be a tracial state such that $f_j(\tau_j)=\|f_j\|$ (which exists by the compactness of $\TT_1(A_j)$). 
Let $\lambda_\filter$ be the associated limit state.
Then 
\[
\lim_\filter \|f_j\| = \lim_\filter f_j(\tau_j) =\lambda_\filter(f)\leq 1.
\]
Thus, $\|f\|\leq 1$, in contradiction with $\mu(f)>1$. 
\end{proof}

An ultrafilter $\filter$ is said to be \emph{countably incomplete} if there exists a sequence~$(E_n)_n$ in~$\filter$ with $\bigcap_n E_n=\varnothing$.
We note that every free ultrafilter on a countable set is countably incomplete.

The next result shows the equivalence of (i), (iv) and~(v) in \autoref{thmtraces}.

\begin{thm}
\label{nosillycomms}
Let $(A_j)_{j\in J}$ be a family of unital \ca{s} and let $\filter$ be a countably incomplete ultrafilter on the index set~$J$.
The following are equivalent:
\begin{enumerate}[{\rm (i)}]
\item
The natural map $\coprod_\filter \TT_1(A_j)\to \TT_1(\prod_\filter A_j)$ from \eqref{coultra} is an isomorphism.
\item
The set $\LimT_1 (\prod_\filter A_j)$ is dense in $\TT_1(\prod_\filter A_j)$, in the weak* topology.
\item
There exist $N\in \NN$ and $E\in \filter$ such that for all $j\in E$ and all $a\in \overline{[A_j,A_j]}$ there exist $b_k,c_k\in A_j$ with $\|b_k\|,\|c_k\|\leq \|a\|^{\frac12}$ for $k=1,\ldots,N$ such that
\[
\left\| a-\sum_{k=1}^N [b_k,c_k] \right\| \leq  \frac{1}{2}\|a\|.
\]
\item 
We have
\[
\overline{\Big[\prod_\filter A_j,\prod_\filter A_j \Big]}=\prod_\filter \overline{[A_j,A_j]}.
\]
(Here $\prod_\filter \overline{[A_j,A_j]}$ denotes the image of  $\prod_{j\in J} \overline{[A_j,A_j]}$, regarded as a subset of 
$\prod_{j\in J} A_j$, under the quotient map $\prod_{j\in J} A_j\to \prod_{\filter} A_j$.)
\end{enumerate}
\end{thm}
\begin{proof}
The equivalence of~(i) and~(ii) follows from previous lemmas. 
Indeed, the map \eqref{coultra} is always injective, by \autoref{autoinj}. 
Thus, it is an isomorphism if and only if it is surjective. 
Since its range is the closure of the set of limit tracial states, by \autoref{comaprange} its surjectivity amounts to the density of the limit tracial states.

\smallskip

\emph{We show that~(ii) implies~(iii).}
For $n\in \NN$ and $j\in J$, set
\[
\Gamma_{n,j} = \left\{ \sum_{k=1}^n[x_k,y_k]: x_k,y_k\in A_j,\, \|x_k\|,\|y_k\|\leq 1\hbox{ for all }k \right\}.
\]
It will suffice to show that there exists $N$ such that the set of indices $j\in J$ for which $\mathrm{dist}(a,\Gamma_{N,j})<\frac{1}{4}$ for all $a \in (A_j)_\sa \cap \overline{[A_j,A_j]}$, with $\|a\| = 1$, belongs to $\filter$. 
The result for a general nonselfadjoint element $a$ is then easily obtained decomposing it as $a=a_1+ia_2$, with $a_1,a_2$ selfadjoint, and normalizing $a_1$ and $a_2$.  

Suppose for the sake of contradiction that for every $n=1,2,\ldots$, the sets
\[
E_n := \big\{ j\in J : \mathrm{dist}(a,\Gamma_{n,j})\geq \frac14\hbox{ for some }a\in \overline{[A_j,A_j]}\hbox{ with }\|a\| \leq 1 \big\}
\] 
belong to $\filter$. 
Using that $\filter$ is countably incomplete, let us choose a decreasing sequence $(E_n')_{n\in\NN}$ in $\filter$ such that $\bigcap_n E_n'=\varnothing$ and  $E_n'\subseteq E_n$ for all $n$.

Let us choose $(a_j)_{j\in J}$ as follows: 
If $j\in J\backslash E_1'$, set $a_j=0$. 
If $j\in E_n'\backslash E_{n+1}'$, choose $a_j\in \overline{[A_j,A_j]}\cap (A_j)_\sa$ of norm $1$ whose distance to the set $\Gamma_{n,j}$ is $\geq 1/4$. 
Since the family~$(\Gamma_{n,j})_n$ is increasing, this construction has the property that if~$j\in E_n'$,~then the distance from~$a_j$ to~$\Gamma_{n,j}$ is $\geq 1/4$.

Let $a=(a_j)_{j\in J}$.
Observe that the tracial states in $\TT_1(A_j)$, regarded as a subset of $\TT_1(\prod_j A_j)$, vanish on $a$ for all $j$. The same is thus true for the limits $\lim_\filter  \tau_j$, with $\tau_j\in \TT_1(A_j)$ for all $j$. Thus, $\pi_\filter(a)$ is in the kernel of every limit tracial state.  It follows by hypothesis that $\pi_\filter(a)$ is in the kernel of every tracial state of  $\prod_\filter A_j$. 
By \autoref{distance2comms}, we have
\[
\left\| \pi_\filter(a)-\sum_{k=1}^N[b_k,c_k] \right\| < \frac{1}{4}
\] 
for some $N\in \NN$ and  $b_k,c_k\in \prod_\filter A_j$. 
Enlarging $N$ if necessary, let us assume that $\|b_k\|,\|c_k\|\leq 1$ for all $k$.
Choose lifts  $\tilde b_k,\tilde c_k\in \prod_j A_j$ of  $b_k,c_k$ such that  $\|\tilde b_k\|,\|\tilde c_k\|\leq 1$ for all $k$.  
Then the set
\[
E := \left\{ j\in J : \|a_j-\sum_{k=1}^N[(\tilde b_k)_j,(\tilde c_k)_j]\| < \frac{1}{4} \right\}
\]
belongs to $\filter$. 
Now choose $j\in E\cap E_N'$. 
On one hand, $\|a_j-\sum_{k=1}^N[(\tilde b_k)_j,(\tilde c_k)_j]\|<1/4$. 
On the other hand, since $j\in E_N'$, the distance from $a_j$ to $\Gamma_{N,j}$ is $\geq 1/4$.
This is the desired contradiction. 

\smallskip

\emph{We show that~(iii) implies~(iv).}
Let us first prove the inclusion of the left-hand side in the right-hand side. 
Let $a\in \overline{[\prod_\filter A_j,\prod_\filter A_j]}$. 
By \cite[Theorem 1.6]{ArcRobTik17Dixmier}, $a$ can be lifted to $\tilde a\in \overline{[\prod_j A_j,\prod_j A_j]}$. 
The latter element clearly belongs to $\prod_j \overline{[A_j,A_j]}$. Thus, $a$ belongs to the image of $\prod_j \overline{[A_j,A_j]}$ under $\pi_\filter$.

Suppose now that $a\in \prod_\filter \overline{[A_j,A_j]}$. 
Let $\tilde a\in \prod_j \overline{[A_j,A_j]}$ be a lift of $a$.  
Let $\varepsilon>0$. 
Using (ii) and \autoref{easyinduction}, choose $N\in \NN$ and $E\in \filter$ such that
\[
\left\| \tilde a_j-\sum_{k=1}^N [b_{k,j},c_{k,j}] \right\| 
\leq \varepsilon\|\tilde a_j\|,
\]
for all $j\in E$, where $b_{k,j},c_{k,j}\in A_j$ are such that $\|b_{k,j}\|,\|c_{k,j}\|\leq \|\tilde a_j\|^{\frac 12}$. 
Set $b_{k,j}=c_{k,j}=0$ for all $j\notin E$ and all $k$. 
Define $b_k=\pi_{\filter}((b_{k,j})_j)$ and $c_k=\pi_\filter((c_{k,j})_j)$. 
Then
\[
\left\| a-\sum_{k=1}^N [b_k,c_k] \right\|
\leq \varepsilon\|\tilde a\|.
\] 
Since this argument can be applied to every $\varepsilon>0$, we get $a\in \overline{[\prod_\filter A_j,\prod_\filter A_j]}$.

\smallskip

\emph{We show that~(iv) implies~(ii).}
Assume (iv).  
Suppose for the sake of contradiction that there exists $\mu\in \TT_1(\prod_\filter A_j)$ that is not in the weak* closure of $\LimT_1(\prod_\filter A_j)$.  
Observe that the set of limit tracial states is convex. 
Thus, by Hahn--Banach, there exists $b\in \prod_\filter A_j$ separating $\mu$ from $\LimT_1(\prod_\filter A_j)$, that is, such that $\mathrm{Re}(\tau(b))\leq 1$ for all $\tau\in \LimT_1(\prod_\filter A_j)$ and $\mathrm{Re}(\mu(b))\geq 1+\delta$, for some $\delta>0$.
Replacing $b$ by its selfadjoint part, we may assume that it is selfadjoint. 
Translating $b$ by a scalar multiple of $1$ and renormalizing (as in the proof of \autoref{comaprange}), we may further assume that $b$ is positive (this step may change $\delta$). 

Let $(b_j)_j\in \prod_j A_j$ be a positive lift of $b$.
For each $j\in J$, let $\bar \tau_j\in \TT_1(A_j)$ be a tracial state at which the mapping $\TT_1(A_j)\ni \tau\mapsto \tau(b_j)$ attains its maximum.
Since $\lim_\filter \bar\tau_j(b_j)\leq 1$, we have that
\[
E := \big\{ j\in J : \bar\tau_j(b_j)< 1+\tfrac{\delta}{4} \big\} \in \filter.
\]
Thus, $\tau(b_j)<1+\tfrac{\delta}{4}$ for all $\tau\in \TT_1(A_j)$ and $j\in E$.
By \autoref{distance2comms}, the distance from $b_j$ to $\overline{[A_j,A_j]}$ is at most $1+\tfrac{\delta}{4}$.
Hence, for each $j\in E$ there exist $c_j\in A_j$ and $d_j\in \overline{[A_j,A_j]}$ such that 
\[
b_j = c_j + d_j,
\]
and $\|c_j\|\leq 1+\tfrac{\delta}{2}$. 
Set $c_j=d_j=0$ for all $j\notin E$. Observe that $(d_j)_j$ is bounded, since $(b_j)_j$
and $(a_j)_j$ are bounded. By hypothesis, $\pi_\filter((d_j)_j)\in \overline{[\prod_\filter A_j,\prod_\filter A_j]}$. Then
\[
b = \pi_\filter((b_j)_j) = \pi_\filter((c_j)_j) + \pi_\filter((d_j)_j).
\]
The term $\pi_\filter((c_j)_j)$ has norm at most $1+\tfrac{\delta}{2}$,  while the term $\pi_\filter((d_j)_j)$ vanishes on every tracial state of  
$\prod_\filter A_j$. Evaluating both sides on $\mu$ we get a contradiction.
\end{proof}

We state below a similar density theorem for tracial states on the product $\prod_{j=1}^\infty A_j$. 
We omit the proof as the arguments run along the same lines (with some simplifications).

\begin{thm}
\label{prodcomms}
Let $(A_j)_{j\in J}$ be a collection of unital \ca{s} indexed by an infinite set~$J$.
The following are equivalent:
\begin{enumerate}[{\rm (i)}]
\item
The natural map $\coprod_j \TT_1(A_j)\to \TT_1(\prod_j A_j)$ from \eqref{coprod} is an isomorphism.
\item
The set $\mathrm{co}(\bigcup_{j\in J} \TT_1(A_j))$ is dense in $\TT_1(\prod_{j\in J} A_j)$, in the weak* topology.
\item
There exists $N\in \NN$ such that for all $j\in J$ and $a\in \overline{[A_j,A_j]}$ there exist $b_k,c_k\in A_j$ with $\|b_k\|,\|b_k\|\leq \|a\|^{\frac12}$ for $k=1,\ldots,N$ such that
\[
\left\| a-\sum_{k=1}^N [b_k,c_k] \right\| \leq \frac{1}{2}\|a\|.
\]
\item 
We have $\overline{[\prod_j A_j,\prod_j A_j]}=\prod_{j} \overline{[A_j,A_j]}$.
\end{enumerate}
\end{thm}

\begin{exa}
\label{exa:1}
In \cite[Theorem~1.4]{Rob15NucDimSumsCommutators} (see also \cite[Example~4.7]{GarThi23arX:GenByCommutators}) an example is given of a simple, unital \ca{} $A$ with a unique tracial state such that for each $m \in \NN$ there exists a contraction $a_m \in \overline{[A,A]}$ whose distance to the set 
\[
\Big\{\sum_{i=1}^m [x_i,y_i]: x_i,y_i\in A\Big\}
\] 
is 1.
Let $\filter$ be a free ultrafilter on $\NN$. 
Observe that, since $\TT_1(A)$ is a singleton set, so is $\LimT_1(\prod_\filter A_j)$ (and it is thus closed).  
On the other hand, since the property in Theorem \ref{nosillycomms}(iii) does not hold, $\TT_1(A_\filter)$ is not a singleton in this case. 
\end{exa}

\begin{exa}
\label{exa:2}
Consider the nc-polynomial in four variables
\[
g = [x_1,x_2][x_3,x_4].
\]
Given a \ca{} $A$, denote by $g(A)$ the range of $g$ on $A$.
Given $n\in \NN$, denote by $\sum^n g(A)$ the set of sums $\sum_{j=1}^n a_j$, with $a_j\in g(A)$ for all $j$. 

Let $n\in \NN$. 
By \cite[Example~3.11]{Rob16LieIdeals}, there exists a unital \ca{} $B_n$ without bounded traces and a projection $b_n \in B_n$, such that the distance from~$b_n$ to the set $\sum^n g(B_n)$ is $1$. 
Fix a free ultrafilter $\filter$ on $\NN$, and set $B=\prod_\filter B_n$. 
Observe that there are no limit traces in $\TT_1(B)$, since $\TT_1(B_n)=\varnothing$ for all $n\in \NN$. Let us argue that~$B$ has a non-zero one-dimensional representation (and in particular $\TT_1(B)$ is non-empty). 

Suppose for a contradiction that $B$  has no one-dimensional representations. 
Then, by \cite[Theorem~A]{GarThi23arX:GenByCommutators}, there exists $N\in \NN$ such that  
$B=\sum^N g(B)$. 
In particular, $\pi_\filter((b_n)_n)$ belongs to $\sum^N g(B)$. We thus get a set of indices
$E\in \filter$ such that the distance from $b_n$ to the set $\sum^N g(B_n)$ is $<1/2$ for all $n\in E$. This, however, contradicts our choice
of $b_n$ for any $n\in E$ such that  $n\geq N$. (Note that $\sum^N g(B_n)$ is contained in $\sum^n g(B_n)$ for $n\geq N$, as $0$ belongs to the range of $g$.)
\end{exa}

\section{Quasitraces and the Cuntz semigroup}
\label{sec:prelim}

In this section we describe the main objects that appear in coming sections of the paper: 
quasitraces on \ca{s}, abstract Cuntz semigroups, and functionals on Cuntz semigroups; see, among others, \cite{CowEllIva08CuInv,AntPerThi18TensorProdCu,EllRobSan11Cone,Rob13Cone}. 

\begin{pgr}[Traces and quasitraces]
\label{pgr:quasitraces}
Let $A$ be a \ca{}. 
We call a map $\tau\colon A_+\to [0,\infty]$ a \emph{trace} (on $A$) if it is additive, linear, and maps $0$ to $0$. 
We denote the set of all lower semicontinuous traces on $A$ by $\TT(A)$. 
This is a cone when endowed with the operations of pointwise  addition and pointwise multiplication by positive scalars.
(In this paper, by a cone we understand a commutative monoid endowed with a scalar multiplication by $(0,\infty)$. Note that we do not define multiplication by $0$. We call the zero element of a cone its origin. We refer to \cite[Section~3.1]{AntPerRobThi21Edwards} for details.)

By a \emph{quasitrace} on $A$ we understand a map $\tau\colon A_+\to [0,\infty]$ whose restriction to the positive part of any commutative sub-\ca{} of $A$ is a trace. 
A \emph{$2$-quasitrace} is a quasitrace that admits an extension to a quasitrace on $M_2(A)_+$. 
We denote by $\QT(A)$ the cone of $[0,\infty]$-valued, lower semicontinuous $2$-quasitraces on~$A$. 
Every lower semicontinuous $2$-quasitrace admits a unique extension to a lower semicontinuous $2$-quasitrace on $A\otimes\mathcal{K}$, where $\mathcal{K}$ denotes the compact operators on $\ell^2(\NN)$.  
We thus regard $(A\otimes\mathcal{K})_+$ as the common domain of the elements of $\QT(A)$. 

The cone $\QT(A)$ can be endowed with a compact Hausdorff topology in which a net~$(\tau_j)_j$ converges to $\tau$ in $\QT(A)$ if and only if for all $a\in A_+$ and $\varepsilon>0$ we have 
\[
\limsup_j \tau_j\big( (a-\varepsilon)_+ \big) \leq \tau(a)\leq \liminf_j \tau_j(a),
\]
where $(a-\varepsilon)_+$ is the \emph{$\varepsilon$-cut-down} of $a$, which is defined by applying continuous functional calculus to $a$ with the function $\RR\to\RR$, $t\mapsto\max\{0,t-\varepsilon\}$;
see \cite[Section~4]{EllRobSan11Cone}. 
\end{pgr}

A very convenient technical tool to deal with quasitraces on a \ca{} is the Cuntz semigroup. 
We give below the axioms used to define the objects of the category $\Cu$ that they belong to.
For further details, we refer to the recent survey \cite{GarPer23arX:ModernCu}.

\begin{pgr}[Cuntz semigroups]
\label{pgr:axioms}
A partially ordered monoid $S$ is \emph{positively ordered} provided that $x\geq 0$ for every element $x\in S$. 
A commutative, positively ordered monoid $S$ is called a \emph{\CuSgp{}} if it satisfies the following axioms: 
\begin{enumerate}
\item[\axiomO{1}]
If $(x_n)_n$ is an increasing sequence in $S$, then $\sup_n x_n$ exists.
\item[\axiomO{2}]
For any $x\in S$ there exists a sequence $(x_n)_n$ such that $x_n\ll x_{n+1}$ for all
$n$ and $x=\sup_n x_n$.
(We say that $(x_n)_n$ is a $\ll$-increasing sequence.)
\item[\axiomO{3}]
If $x_1\ll x_2$ and $y_1\ll y_2$, then $x_1+y_1\ll x_2+y_2$.
\item[\axiomO{4}]
If $(x_n)_n$ and $(y_n)_n$ are increasing sequences in $S$, then $\sup_n (x_n+y_n) = \sup_n x_n + \sup_n y_n$.
\end{enumerate}

The relation $\ll$ in these axioms is defined as follows:
$x\ll y$ if for every increasing sequence $(y_n)_n$ satisfying $y\leq \sup_n y_n$ there exists $n_0\in\NN$ such that $x\leq y_{n_0}$.
The relation $\ll$ is called the \emph{way-below relation}, or \emph{compact containment relation}, and one says that `$x$ is way-below $y$' if $x \ll y$.
An element $u\in S$ such that $u\ll u$ is termed \emph{compact}.

There are additional axioms that we often impose on a \CuSgp:
\begin{enumerate}
\item[\axiomO{5}]
For all $x',x,y$ with $x'\ll x\leq y$ there exists $z$ such that $x'+z\leq y\leq x+z$. 
Moreover, if $x+w\leq y$ for some $w$, and $w'\ll w$, then $z$ may be chosen such that $w'\ll z$. 
\item[\axiomO{6}]
For all $x',x,y,z\in S$ such that $x\leq y+z$ and $x'\ll x$ there exist $y',z'$ such that $x'\leq y'+z'$, such that $y'\leq y,x$, and $z'\leq z,x$. 
\end{enumerate}

Given positive elements $a,b$ in a \ca{} $A$, one says that $a$ is \emph{Cuntz subequivalent} to $b$, denoted $a \precsim b$, if there is a sequence $(r_n)_n$ in $A$ such that $\lim_{n\to\infty}\|a-r_nbr_n^*\|=0$.
Further, $a$ and $b$ are \emph{Cuntz equivalent}, denoted $a \sim b$, if $a \precsim b$ and $b \precsim a$.
These relations were introduced by Cuntz in \cite{Cun78DimFct}. 

The \emph{Cuntz semigroup} of $A$ is defined as $\Cu(A)=(A\otimes \mathcal{K})_+/{\sim}$, equipped with the partial order induced by $\precsim$, and equipped with addition induced by addition of orthogonal positive elements.
It is known that $\Cu(A)$ satisfies \axiomO{1}--\axiomO{6};
see \cite{CowEllIva08CuInv}, \cite{RorWin10ZRevisited}, \cite[Section~4]{AntPerThi18TensorProdCu}, \cite[Proposition 5.1.1]{Rob13Cone}.
Further properties \axiomO{7} and \axiomO{8} for $\Cu(A)$ have been obtained in \cite[Section~2.2]{AntPerRobThi21Edwards} and \cite[Section~7]{ThiVil21arX:NowhereScattered}.

Classes of projections in $A$ are natural examples of compact elements in $\Cu(A)$, and often the only ones;
see \cite{BroCiu09IsoHilbModSF}.

As defined above, \CuSgp{s} are the objects of a category, termed $\Cu$. 
The morphisms in this category are called \emph{\CuMor{s}}. 
By definition, a \CuMor{} between \CuSgp{s} is an order-preserving monoid homomorphism that preserves the relation $\ll$ and suprema of increasing sequences. 
The assignment $A\mapsto\Cu(A)$ is functorial; see \cite{CowEllIva08CuInv}.
\end{pgr}

\begin{pgr}[Functionals on Cuntz semigroups]
\label{pgr:functionals}
Let $S$ be a \CuSgp. We call a map $\lambda\colon S\to [0,\infty]$  a \emph{functional} on $S$ if $\lambda$ is an order-preserving monoid homomorphism that preserves the suprema of increasing sequences. 
The set of functionals on~$S$ is denoted by~$\FF(S)$. 
This set is a cone under pointwise addition of functionals and pointwise scalar multiplication by positive real numbers. Its origin is the zero functional. 
The properties of $\FF(S)$ have been studied in \cite{Rob13Cone} under the additional assumption that $S$ satisfies \axiomO{5}. 
The question of whether \axiomO{5} is necessary for a proper theory of $\FF(S)$ is an interesting one, but we do not take it up here.

The cone $\FF(S)$ has a natural compact Hausdorff topology such that a net $(\lambda_j)_j$ converges to $\lambda$ in $\FF(S)$ if and only if 
\[
\limsup_j \lambda_j(x') \leq \lambda(x) \leq \liminf_j \lambda_j(x),
\]
for all $x'\ll x$ in $S$; 
see \cite[Theorem~4.8]{EllRobSan11Cone}, \cite{Rob13Cone}, and \cite[Theorem~3.17]{Kei17CuSgpDomainThy}. 

Given $x\in S$, we denote by $\widehat{x}\colon \FF(S)\to [0,\infty]$ the function such that $\widehat{x}(\lambda)=\lambda(x)$ for all $\lambda\in \FF(S)$, which is lower semicontinuous, zero-preserving, additive and $(0,\infty)$-homogeneous (see \autoref{pgr:LC} for further details). 
Given $u\in S$, we denote by $\FF_u(S)$ the set of functionals $\lambda\in \FF(S)$ that are normalized at $u$, that is, $\lambda(u)=1$. 
If $\widehat{u}$ is continuous (for example, if $u$ is a compact element of $S$), then $\FF_u(S)$ is a closed, convex subset of $\FF(S)$, and hence a compact convex set. 
\end{pgr}

Below, we will work with limits along ultrafilters. 
We will thus find it convenient to formulate convergence of functionals in those terms:

\begin{lma}
\label{filterconvergence}
Let $S$ be a \CuSgp{} satisfying \axiomO{5}, let $(\lambda_j)_{j\in J}$ be a collection of functionals in $\FF(S)$, and let $\filter$ be an ultrafilter on the set $J$.
Then there is a unique $\lambda\in \FF(S)$ such that $(\lambda_j)_j$ converges to $\lambda$ along $\filter$ in the compact Hausdorff topology of~$\FF(S)$.
This $\lambda$ is given by
\[
\lambda(x)=\sup_{x'\ll x}\lim_{j\to\filter} \lambda_j(x'), \quad \hbox{ for all }x\in S. 
\]
\end{lma}
\begin{proof}
Since $\FF(S)$ is compact and Hausdorff, the limit $\lambda$ exists and is unique.  
Let $\filter \ni E\mapsto j_E\in E$ be an arbitrary selection. 
Let $x'\ll x$ in $S$. 
Since the net $(\lambda_{j_E})_{E\in \filter}$ converges to~$\lambda$, we have
\[
\lim_{j\to\filter} \lambda_j(x')
= \limsup_E \lambda_{j_E}(x')
\leq \lambda(x),
\]
and
\[
\lambda(x')
\leq \liminf_E \lambda_{j_E}(x')
= \lim_{j\to\filter} \lambda_j(x').
\]
Thus, $\lambda(x')\leq \lim_{j\to\filter} \lambda_j(x')\leq \lambda(x)$.
This, combined with the fact that $\lambda(x)=\sup_{x'\ll x}\lambda(x')$, yields the desired result.
\end{proof}

\begin{pgr}
\label{linkqtracesfunc}
The link between quasitraces and functionals on Cuntz semigroups is as follows: 
For every $\tau\in\QT(A)$, define $d_\tau\colon \Cu(A)\to [0,\infty]$ by 
\[
d_\tau([a])=\lim_n \tau(a^{\frac1n})
\] 
for all positive elements $a\in A\otimes \mathcal{K}$. Then $d_\tau$ is a functional on $\Cu(A)$. Moreover, 
the assignment 
\[
\tau\mapsto d_\tau,
\]
from $\QT(A)$  to $\FF(\Cu(A))$, is an isomorphism of topological cones;
see \cite[Theorem~4.4]{EllRobSan11Cone}.
\end{pgr}

\section{Ultraproducts, limit quasitraces, and limit functionals}
\label{sec:ultrapowers}

In this section we define limit quasitraces and state the density of limit quasitraces problem;
see \autoref{pbm:3.1}. 
We then rephrase this problem in the language of abstract Cuntz semigroups; 
see \autoref{pbm:denseLimitFctl}.

\begin{pgr}[Limit quasitraces]
\label{pgr:limitquasitraces}
Let $(A_j)_{j\in J}$ be a family of \ca{s}. 
Let $\filter$ be a free ultrafilter on~$J$, and let $\prod_\filter A_j$ denote the ultraproduct of the family $(A_j)_j$ along $\filter$. 
Given a selection of $2$-quasitraces $\tau_j\in \QT(A_j)$ for all $j\in J$, let $\bar \tau_j=\tau_j\pi_j$, where $\pi_j\colon\prod_j A_j\to A_j$ is the quotient map. 
Observe that $\bar\tau_j\in \QT(\prod_j A_j)$ for all $j$. 
Define $\bar \tau_\filter\in \QT(\prod_j A_j)$ as the limit of $(\bar{\tau}_j)_j$ along $\filter$, which exists by compactness of $\QT(\prod_j A_j)$. 
More explicitly, it is not difficult to calculate that  $\bar{\tau}_\filter$ is given by
\[
\bar{\tau}_\filter(a)
= \sup_{t>0}\lim_{j\to\filter} \tau_j\big( (a_j-t)_+ \big),
\]
for $a=(a_j)_j$ in $(\prod_j A_j)_+$;
see \autoref{filterconvergence}.
Observe that $\bar{\tau}_\filter$ vanishes on the ideal $c_\filter((A_j)_j)$, and thus induces a lower semicontinuous $2$-quasitrace $\tau_\filter$ on the ultraproduct $\prod_\filter A_j$ such that $\bar{\tau}_\filter=\tau_\filter\pi_\filter$. 
We call $\tau_\filter$ a \emph{limit $2$-quasitrace} on $\prod_\filter A_j$. We denote by  $\LimQT(\prod_\filter A_j)$  the set of all limit $2$-quasitraces.
 
If each $\tau_j$ is a trace, then so is  $\tau_\filter$ and we call it a \emph{limit trace}.
We denote by $\LimT(\prod_\filter  A_j)$ the set of all limit traces on $\prod_\filter A_j$.

Finally, if each  $A_j$ is unital, and each $\tau_j$ is a tracial state, then $\tau_\filter$ is again a tracial state.
In this case the set of limit tracial states agrees with the set $\LimT_1(\prod_\filter A_j)$  that we have already 
introduced in \autoref{limittracialstates}.
\end{pgr}

As mentioned in the introduction, one of the main problems that we address in this paper  is the following:

\begin{pbm}
\label{pbm:3.1}
Retaining the notation from the previous paragraph, under what conditions is the set of limit $2$-quasitraces $\LimQT(\prod_\filter A_j)$ dense in $\QT(\prod_\filter A_j)$?
\end{pbm}

To tackle this problem, we use the correspondence between $2$-quasitraces and functionals on the Cuntz semigroup described in \autoref{linkqtracesfunc}.
This translates the above problem into a question on the density of limit functionals in the cone of functionals of an ultraproduct of Cuntz semigroups. 
In the coming paragraphs we formulate a version of said problem in this setting and, as we shall see, the techniques developed in \cite{AntPerThi20CuntzUltraproducts} play a key role in the solution of \autoref{pbm:3.1}.

\begin{pgr}
For completeness, we give a brief account of the construction of the quotient semigroup by an ideal, which will be used in the sequel. 
For more details see, for example, \cite[5.1.1]{AntPerThi18TensorProdCu}. 
Given a \CuSgp{} $S$, an ideal $I$ of $S$ is a downward-hereditary subset that is closed under addition and under suprema of increasing sequences. Given elements $x,y\in S$, we define $x\leq_I y$ if there is $z\in I$ such that $x\leq y+z$. 
We also set $x\sim_I y$ if both $x\leq_I y$ and $y\leq_I x$ occur. 
Define $S/I=S/{\sim_I}$, which is a \CuSgp{} with the naturally induced addition and order.
The quotient map $\pi_I\colon S\to S/I$ is a surjective \CuMor. 
In the case of a \ca{} $A$ and a closed, two-sided ideal $J$ of $A$, the inclusion of $J$ in $A$ induces an order embedding of $\Cu(J)$ as an ideal of $\Cu(A)$, and the quotient map $A\to A/J$ induces a $\Cu$-isomorphism $\Cu(A)/\Cu(J)\cong \Cu(A/J)$; 
see \cite[Proposition 1]{CiuRobSan10CuIdealsQuot}.
\end{pgr}

\begin{pgr}[Products and ultraproducts of \CuSgp{s}]
\label{pgr:Cuultraproducts} 
Let us review the construction of products and ultraproducts of \CuSgp{s} developed in
\cite{AntPerThi20CuntzUltraproducts}. 
Let $(S_j)_{i\in J}$ be a collection of \CuSgp{s}. 
We denote by $\CatPomProd_j S_j$ their product in the category of positively ordered monoids. 
This is simply the cartesian product endowed with the entrywise order and entrywise addition. 
We denote by $\llpw$ the relation in $\CatPomProd_j S_j$ of entrywise $\ll$-comparison.

By a \emph{path} in $\CatPomProd_j S_j$ we understand a map $\vect{v}\colon (-\infty,0]\to \CatPomProd_j S_j$, $t\mapsto\vect{v}_t$, that satisfies:
\begin{enumerate}
\item
$\vect{v}_s \llpw \vect{v}_t$ for all $s,t\leq 0$ with $s<t$,
\item
$\vect{v}_t = \sup_{t'<t} \vect{v}_{t'}$ for all $t\in (-\infty,0]$. 
\end{enumerate}

In the sequel, given a path $\vect{v}=(\vect{v}_t)_{t\leq 0}$ we shall write $\vect{v}_t=(v_{t,j})_j$ with $v_{t,j}\in S_j$ for each $t\leq 0$ and $j\in J$. 
We define on the set of paths in $\CatPomProd_j S_j$ a preorder relation as follows: 
$(\vect{v}_t)_{t\leq 0}\precsim (\vect{w}_t)_{t\leq 0}$ if for every $s<0$ there exists $t<0$ such that $\vect{v}_s\llpw \vect{w}_t$, that is, $v_{s,j}\ll w_{t,j}$ for all $j\in J$.
We define $(\vect{v}_t)_{t\leq 0}\sim (\vect{w}_t)_{t\leq 0}$ if $(\vect{v}_t)_{t\leq 0}\precsim (\vect{w}_t)_{t\leq 0}$ and $(\vect{w}_t)_{t\leq 0}\precsim (\vect{v}_t)_{t\leq 0}$. 
We denote by $[(\vect{v}_t)_{t\leq 0}]$ the equivalence class of the path~$(\vect{v}_t)_{t\leq 0}$. 

The product $\prod_j S_j$ in the category of \CuSgp{s} is defined as
the set of equivalence classes $[(\vect{v}_t)_{t\leq 0}]$, where $v\colon (-\infty,0]\to \CatPomProd_j S_j$ is a path.
Addition and order on $\prod_j S_j$ are defined by $[(\vect{v}_t)_{t\leq 0}]+[(\vect{w}_t)_{t\leq 0}]=[(\vect{v}_t+\vect{w}_t)_{t\leq 0}]$ and $[(\vect{v}_t)_{t\leq 0}]\leq [(\vect{w}_t)_{t\leq 0}]$ if $(\vect{v}_t)_{t\leq 0}\precsim (\vect{w}_t)_{t\leq 0}$, respectively. Here $\vect{v}_t+\vect{w}_t=(v_{t,j}+w_{t,j})_j$. 
The projection maps $\pi_j\colon \prod_j S_j\to S_j$ are defined as 
\[
\pi_j\big( [(\vect{v}_t)_{t\leq 0}] \big) = v_{0,j}\in S_j,\hbox{ with }\vect{v}_t=(v_{t,j})_j.
\]
It is shown in \cite[Corollary~3.9]{AntPerThi20CuntzUltraproducts} that $\prod_j S_j$ is a \CuSgp{} satisfying the universal property for products in the category of \CuSgp{s}.
(See also \cite[Section 3]{AntPerThi20AbsBivariantCu} for a full account of this construction.)

Let $\filter$ be an ultrafilter on the set $J$. Define $\cc_\filter((S_j)_j)$ as the subset of $\prod_j S_j$ of $[(\vect{v}_t)_{t\leq 0}]$, with $\vect{v}_t=(v_{t,j})_j$ for all $t\leq 0$, such that
\[
\big\{ j\in J : v_{t,j}=0 \big\} \in \filter \quad \hbox{ for all }t<0. 
\]
Then $\cc_\filter((S_j)_j)$ is an ideal of  $\prod_j S_j$. 
Following \cite{AntPerThi20CuntzUltraproducts}, we define the ultraproduct of~$(S_j)_j$ along $\filter$ as follows:
\[
\prod_\filter S_j = \Big(\prod_j S_j\Big)/\cc_\filter\big( (S_j)_j \big).
\]
The natural quotient map $\prod_j S_j\to \prod_\filter S_j$ will be denoted by $\pi_\filter$. 

By \cite[Lemma~7.8]{AntPerThi20CuntzUltraproducts}, the order in the ultraproduct is characterized as follows:
For $[(\vect{v}_t)_{t\leq 0}], [(\vect{w}_t)_{t\leq 0}]\in\prod_j S_j$ with $\vect{v}_t=(v_{t,j})$ and $\vect{w}_t=(w_{t,j})$, we have $\pi_\filter ([(\vect{v}_t)_{t\leq 0}])\leq \pi_\filter([(\vect{w}_t)_{t\leq 0}])$ if, and only if, for every $s<0$, there are $t<0$ and $E\in \filter$ such that $v_{s,j}\ll w_{t,j}$ for each $j\in E$.
\end{pgr}

We are also interested in products and ultraproducts of scaled \CuSgp{s}, as these arise naturally from products and ultraproducts of \ca{s}.

\begin{pgr}[Scales]
\label{pgr:scales}
A \emph{scale} on a \CuSgp{} $S$ is a downward hereditary subset $\Sigma\subseteq S$ that is closed under suprema of increasing sequences and that generates $S$ as an ideal, that is, for every $x', x\in S$ with $x'\ll x$, there are elements $x_1,\dots,x_n\in \Sigma$ such that $x'\leq \sum_{i=1}^nx_i$;
see \cite[Definition~4.1]{AntPerThi20CuntzUltraproducts}.
The pair $(S,\Sigma)$ is referred to as a \emph{scaled \CuSgp}. 
Given scaled \CuSgp{s} $(S,\Sigma)$ and $(T,\Theta)$, a \emph{scaled \CuMor} is a \CuMor{} $\varphi\colon S\to T$ such that $\varphi(\Sigma)\subseteq \Theta$. 
We denote by $\CatCu_\mathrm{sc}$ the category of scaled \CuSgp{s} with scaled \CuMor{s}.

We shall also consider pairs $(S,u)$ of a \CuSgp{} together with a compact full element $u\in S$, that is, $u$ is such that $u\ll u$ and $\infty\cdot u$ is the largest element of~$S$. 
The element $u$ gives rise to a scale on $S$,  namely, $\Sigma_u=\{x\in S:x\leq u\}$. In the sequel, we regard a pair $(S,u)$ as a scaled \CuSgp{} precisely in this fashion. 

For a \ca{} $A$, the set 
\[
\Sigma_A 
:= \left\{ x\in\Cu(A) : \, \parbox{7cm}{for every $x'\in\Cu(A)$ with $x' \ll x$ there exists $a\in A_+$ with $x \leq [a]$} \right\}
\]
is a scale for $\Cu(A)$. 
The \emph{scaled Cuntz semigroup} of $A$ is $\Cu_\mathrm{sc}(A)=(\Cu(A),\Sigma_A)$;
see \cite[4.2]{AntPerThi20CuntzUltraproducts}.
By parts~(1) and~(2) of \cite[Lemma~3.3]{ThiVil22arX:Glimm}, the scale $\Sigma_A$ can also be described as
\begin{align*}
\Sigma_A 
&= \left\{ x\in\Cu(A) :  \, \parbox{7cm}{there exists a $\precsim$-increasing sequence $(a_n)_n$ in~$A_+$ such that $x = \sup_n [a_n]$} \right\} \\
&= \big\{ x\in\Cu(A) : \text{there exists } a\in A_+ \text{ with }x\leq [a] \big\}.
\end{align*}

If $\varphi\colon A\to B$ is a $\ast$-homomorphism, then $\Cu(\varphi)$ maps $\Sigma_A$ into $\Sigma_B$, and thus is a scaled \CuMor. 
One has therefore a functor from the category of \ca{s} to the category $\CatCu_\mathrm{sc}$.
For a unital \ca{} $A$, we obtain a pair $(\Cu(A),[1])$ of a \CuSgp{} with a compact full element $[1]\in \Cu(A)$. 
\end{pgr}

\begin{pgr}[Scaled products and ultraproducts]
\label{scaledultra}
Let $((S_j,\Sigma_j))_{j\in J}$ be a collection of scaled \CuSgp{s}.
Define $\Sigma\subseteq \prod_j S_j$ as  
\[
\Sigma = \big\{ [(\vect{v}_t)_{t\leq 0}]\in \prod_j S_j : v_{t,j}\in \Sigma_j\hbox{ for all $j\in J$ and all $t<0$} \big\}.
\]
The set $\Sigma$ is downward hereditary and closed under passing to suprema of increasing sequences, though possibly not full in $\prod_j S_j$. 
The \emph{scaled product} of $\prod_j (S_j,\Sigma_j)$ is defined as the pair $(S,\Sigma)$, where $S$ is the ideal generated by $\Sigma$ in $\prod_j S_j$. 

Let $\filter$ be an ultrafilter on $J$. 
The \emph{scaled ultraproduct} $(T,\Theta)=\prod_\filter (S_j,\Sigma_j)$ is defined as the images of $S$ and $\Sigma$ under the quotient by $\cc_\filter((S_j)_j)$. 
In the case $S_j=S$ for all $j$, we shall denote the ultrapower $\prod_\filter (S,\Sigma)$ by $(S,\Sigma)_\filter$. 
The reader is referred to \cite[Paragraph~4.5]{AntPerThi20CuntzUltraproducts} for further details on this construction.

Consider now a collection $(S_j,u_j)_{j\in J}$ of \CuSgp{s} together with a full compact element $u_j\in S_j$ for each $j$.
Let $\vect{v}_t = (u_j)_j$, for $t\leq 0$, denote the constant path equal to $(u_j)_j$ in the cartesian product $\CatPomProd_{j\in J} S_j$. 
Let $\bar{v}=[(\vect{v}_t)_{t\leq 0}]$ be the corresponding equivalence class in $\prod_{j\in J} S$. 
Then it is readily verified that $\bar{v}$ is a compact full element of the scaled product $(S,\Sigma)=\prod_{j\in J}  (S_j,\Sigma_{u_j})$. 
We define $(S,\bar{v})=\prod_{j\in J}(S_j,u_j)$.
If $\filter$ is an ultrafilter on $J$, then passing to the quotient by $\cc_\filter\big( (S_j)_j \big)$ we obtain $v=\pi_\filter(\bar{v})$, a compact full element in the ultraproduct $(T,\Theta)=\prod_\filter (S_j,\Sigma_{u_j})$. 
Again, in this case we write $(T,v)=\prod_\filter (S_j,u_j)$. 
For ultrapowers, we denote by  $(S,u)_\filter$ the ultraproduct $\prod_\filter (S,u)$. 

It is proved in \cite[Theorem~5.13]{AntPerThi20CuntzUltraproducts} that the scaled Cuntz semigroup functor preserves products. 
More concretely, given a family $(A_j)_{j\in J}$ of \ca{s}, let $(S,\Sigma)$ be the scaled product of $(\Cu(A_j),\Sigma_{A_j})$ as described in the paragraph above. 
Then $(S,\Sigma)\cong \Cu_\mathrm{sc}(\prod_j A_j)$ as scaled \CuSgp{s}. 
It is also shown in \cite[Theorem~7.5]{AntPerThi20CuntzUltraproducts} that the scaled Cuntz semigroup preserves ultraproducts. 
In other words, given an ultrafilter $\filter$ on a set $J$ and a family of \ca{s} $(A_j)_{j\in J}$, there is an isomorphism $\Cu_\mathrm{sc}(\prod_\filter A_j)\cong \prod_\filter (\Cu(A_j),\Sigma_{A_j})$. In fact, we have the following commutative diagram:
\[
\xymatrix{
\Cu_\mathrm{sc}(\prod_j A_j)\ar[d]_{\Cu_\mathrm{sc}(\pi_\filter)} \ar[r]^{\cong} 
& \prod_j \Cu_\mathrm{sc}(A_j)\ar[d]^{\pi_\filter} \\
\Cu_\mathrm{sc}(\prod_\filter A_j)\ar[r]^{\cong} &  \prod_\filter \Cu_\mathrm{sc}(A_j).
}
\]
In the case $A_j=A$ for all $j$, we shall use $(\Cu(A),\Sigma_A)_\filter$ to denote the scaled ultrapower which, as observed, is isomorphic to $\Cu_{\mathrm{sc}}(A_\filter)$.
\end{pgr}

We now introduce the limit functionals on an ultraproduct of \CuSgp{s}. 

\begin{pgr}
\label{pgr:limit_functionals}
Let us continue to denote by $(S_j)_{j\in J}$ a collection of \CuSgp{s} and by $\filter$ an ultrafilter on $J$. 
Observe that for each $k\in J$ the projection map $\pi_k\colon \prod_j S_j\to S_k$ induces a cone morphism $\FF(\pi_k)\colon \FF(S_k)\to \FF(\prod S_j)$. 
Consider now a selection of functionals $\lambda_j\in \FF(S_j)$ for all~$j$, and set $\bar{\lambda}_j=F(\pi_j)(\lambda_j)$ for all~$j$. 

Let $\bar{\lambda}_\filter$ be the limit of $(\lambda_j)_j$ in $\FF(\prod_j S_j)$ along $\filter$, which exists and is unique, since $\FF(\prod_j S_j)$ is compact and Hausdorff. 
Using \autoref{filterconvergence}, it is readily established that 
\begin{equation}\label{limit_func_formula}
\bar{\lambda}_\filter\big( [(\vect{v}_t)_{t\leq 0}] \big)
= \sup_{t<0} \lim_{j\to\filter} \lambda_j( v_{t,j}),
\end{equation}
for any path $(\vect{v}_t)_{t\leq 0}$ in $\CatPomProd_j S_j$.
\end{pgr}

\begin{lma}
The functional $\bar{\lambda}_\filter$ vanishes on $\cc_\filter((S_j)_j)$.
\end{lma}
\begin{proof}
Let $\vect{v}=(\vect{v}_t)_{t\leq 0}$  be a path in $\CatPomProd_j S_j$ with $\vect{v}_t=(v_{t,j})_j$, and assume that $[\vect{v}]\in \cc_\filter((S_j)_j)$. 
Then $\{j\in J\colon v_{t,j}=0\}\in\filter$ for every $t<0$. 
It follows that $\lim_{j\to\filter} \lambda_j(v_{t,j})=0$ for every $t<0$, and therefore $\bar{\lambda}_\filter([\vect{v}])=0$ by \eqref{limit_func_formula}.
\end{proof}

Since $\bar{\lambda}_\filter$ vanishes on $\cc_\filter((S_j)_j)$, it induces a functional $\lambda_\filter$ on the ultraproduct~$\prod_\filter S_j$, which is simply given by
\[
\lambda_\filter(\pi_\filter\big( [\vect{v}]) \big) 
= \bar{\lambda}_\filter([\vect{v}]) \quad \hbox{ for all }[\vect{v}]\in \prod_j S_j.
\]

\begin{pgr}[Limit functionals]
\label{ntn:limitFctl}
Retain the notation of \autoref{pgr:limit_functionals}.
We call the functional~$\lambda_\filter$ on $\prod_\filter S_j$ defined above the \emph{limit functional} associated to the family~$(\lambda_j)_j$. 
We use $\LimF(\prod_\filter S_j)$ to denote the subset of limit functionals in $\FF(\prod_\filter S_j)$. 

We also call the functional $\bar{\lambda}_\filter$ on $\prod_j S_j $ defined in \autoref{pgr:limit_functionals} a \emph{limit functional}. 
We denote by $\LimF(\prod_j S_j)$ the subset of $\FF(\prod_j S_j)$ consisting of such limit functionals.

Given a scaled \CuSgp{} $(S,\Sigma)$, we set $\FF((S,\Sigma))=F(S)$. That is, when we speak of functionals on a scaled \CuSgp{} $(S,\Sigma)$ we simply mean functionals on $S$. 

Let $\Sigma_j$ be a scale on $S_j$ for each $j$, and let $(S,\Sigma)=\prod_\filter (S_j,\Sigma_j)$ be the scaled ultraproduct.
Recall that $S$ is an ideal in  $\prod_\filter S_j$. 
Thus, functionals on $\prod_\filter S_j$ induce functionals on the scaled ultraproduct by restriction.  
We use $\LimF(\prod_\filter (S_j,\Sigma_j))$ to denote the subset of $\FF(\prod_\filter (S_j,\Sigma_j))$ induced by the limit functionals.  

Suppose now that $(S_j,u_j)_{j\in J}$ is a collection of \CuSgp{s} endowed with full compact elements $u_j\in S_j$ for each $j\in J$. 
We let $\LimF_u(\prod_\filter (S_j,u_j))$ denote the set of limit functionals associated to families $(\lambda_j)_j$ with $\lambda_j \in \FF_{u_j}(S_j)$, that is, normalized at~$u_j$, for each~$j$. 
Notice that every limit functional in $\LimF_u(\prod_\filter (S_j,u_j))$ is normalized at~$v$. 
In fact, it is easily established that
\begin{equation}
\label{normalizedlimits}
\LimF_u\big( \prod_\filter (S_j,u_j) \big)
= \left\{ \lambda\in \LimF(\prod_\filter S_j): \lambda(v)=1 \right\}.
\end{equation}

That is, a limit functional normalized at $v$ is a limit of normalized functionals (and conversely). To see that the right-hand side is contained in the left-hand side,
let $\bar\lambda_\filter = \lim_\filter \lambda_j$, with $\lambda_j\in \FF(S_j)$ be such that $\bar\lambda_\filter(v)=1$. 
We get at once that 
$\lim_{\filter}\lambda_j(u_j)=1$, and after normalizing each $\lambda_j$ (along an index set where $\lambda_j(u_j)<\infty$), we obtain that 
$\bar\lambda_\filter = \lim_\filter \tilde \lambda_j$ where $\tilde\lambda_j\in \FF_{u_j}(S_j)$.
\end{pgr}

\begin{pgr}
Let $(A_j)_{j\in J}$ be a family of \ca{s}. 
Let $\filter$ be a free ultrafilter on the set~$J$.
As mentioned at the end of \autoref{pgr:scales}, $\Cu_\mathrm{sc}(\prod_\filter A_j)$ is isomorphic to the scaled ultraproduct $\prod_\filter\Cu_\mathrm{sc}(A_j)=\prod_\filter (\Cu(A_j),\Sigma_{A_j})$. 
We thus obtain an isomorphism between $\FF(\Cu_\mathrm{sc}(\prod_\filter A_j))$ and $\FF(\prod_\filter \Cu_\mathrm{sc}(A_j))$.
Recall that, for a scaled \CuSgp{} $(S,\Sigma)$, we have defined $\FF(S,\Sigma)=F(S)$. 
Therefore we may identify $\FF(\Cu(\prod_\filter A_j))$ with $\FF(\prod_\filter \Cu_\mathrm{sc}(A_j))$. 

Fix $k\in J$. A $2$-quasitrace $\tau\in \QT(A_k)$ induces a functional $\FF(\Cu(A_k))$ under the correspondence $\tau\mapsto d_\tau$ described in \autoref{linkqtracesfunc}. 
On the other hand, $\tau$ gives rise to  $\bar \tau =\tau\pi_k$ in $\QT(\prod_j A_j)$ via the projection map.
We have the commutative diagram
\[
\xymatrix@C+=1cm{
\QT(A_k)\ar[d]\ar[r]^{\tau\mapsto d_\tau} &F(\Cu(A_k))\ar[d]\\
\QT(\prod_j A_j)\ar[r]^-{\tau\mapsto d_\tau} & \FF(\prod_j \Cu_\mathrm{sc}(A_j)),
}
\]
where the vertical arrows are induced by the projection maps $\pi_k\colon \prod A_j\to A_k$ and $\Cu(\pi_k)\colon \prod_j \Cu(A_j)\to \Cu(A_k)$.
Since $\tau\mapsto d_\tau$ is a homeomorphism, the limit $\bar \tau_\filter=\lim_\filter \bar \tau_j$ associated to a collection $(\tau_j)_j$ is mapped to the limit $\bar \lambda_\filter=\lim_\filter \overline{d_{\tau_j}} $ associated to the functionals $(d_{\tau_j})_j$. 
After factoring both $\bar\tau_\filter$ and $\bar\lambda_\filter$ by suitable ideals, the limit $2$-quasitrace $\tau_\filter$ associated to $(\tau_j)_j$ is mapped to the limit functional~$\lambda_\filter$ associated to $(d_{\tau_j})_j$. 
Notice finally that, if all the \ca{s} $A_j$ are unital and $\tau_j(1)=1$ for all~$j$, then both $\bar{\tau}_j(1)=1$ and $\overline{d_{\tau_j}}([1])=1$. 
Further, $\tau_\filter(1)=1$ and $\lambda_\filter([(1)_j])=1$.
In summary, we have the following theorem:
\end{pgr}

\begin{thm}
\label{QTFbijection}
The isomorphism between $\QT(\prod_j A_j)$ and $\FF(\prod_j \Cu_\mathrm{sc}(A_j))$, given by $\tau\mapsto d_\tau$, restricts to a natural bijection between the set $\LimQT(\prod_j A_j)$ of limit $2$-quasitraces and the set $\LimF (\prod_j \Cu_\mathrm{sc}(A_j))$ of limit functionals.

Similarly, for ultraproducts, $\tau\mapsto d_\tau$ yields a natural bijection
from  the set $\LimQT(\prod_\filter A_j)$  to  the set $\LimF (\prod_\filter \Cu_\mathrm{sc}(A_j))$.

Furthermore, if all the \ca{s} $A_j$ are unital, then $\tau\mapsto d_\tau$ also gives a bijection  between the set  
$\LimQT_1(\prod_\filter A_j)$ of limits of  normalized $2$-quasitraces and the set $\LimF_{[1]}(\prod_\filter (\Cu(A_j), [1]))$ of normalized limit functionals.
\end{thm}

In view of the previous theorem,  \autoref{pbm:3.1} is subsumed in the following more general problem: 

\begin{pbm}
\label{pbm:denseLimitFctl}
Retaining the setting from \autoref{ntn:limitFctl}, characterize when the set $\LimF(\prod_\filter S_j)$ is dense in $\FF(\prod_\filter S_j)$.
\end{pbm}

We address this problem in \autoref{sec:main}, together with similar questions for scaled ultrapowers and ultraproducts.

\section{Density of limit functionals}
\label{sec:main}

In this section we solve \autoref{pbm:denseLimitFctl} by characterizing the density of limit functionals in terms of a comparability condition;
see \autoref{prp:CharDensityUltraproduct}.
We study this condition more closely in \autoref{sec:LCBA}.

We start by characterizing when elements in an ultraproduct compare on all functionals in the closure of limit functionals.

\begin{prp}
\label{prp:CompClosureUltraproduct}
Let $(S_j,\Sigma_j)_{j\in J}$ be a collection of scaled \CuSgp{s} that satisfy \axiomO{5}.
Let $\filter$ be a free ultrafilter on $J$.
Let $\gamma\in\RR_+$, and let $x,y\in \prod_\filter (S_j,\Sigma_j)$. 
Suppose that $x=\pi_{\filter}(\tilde x)$ and $y=\pi_{\filter}(y)$, where
\[
\tilde x=[( (x_{t,j})_j )_{t\leq 0}], \andSep
\tilde y=[( (y_{t,j})_j )_{t\leq 0}] 
\]
are elements of the product $\prod_j (S_j,\Sigma_j)$.
The following are equivalent:
\begin{enumerate}[{\rm (i)}]
\item
We have $\widehat{x}(\lambda) \leq \gamma \widehat{y}(\lambda)$ for every $\lambda\in\overline{\LimF(\prod_\filter (S_j,\Sigma_j))}$.
\item
For every $s<0$ and $\gamma'>\gamma$, there exists $t<0$ such that
\[
\big\{ j\in J : \widehat{x_{s,j}} \leq \gamma' \widehat{y_{t,j}} \big\} \in \filter.
\]
\end{enumerate}
\end{prp}
\begin{proof}
In terms of the lifts $\tilde x$ and $\tilde y$ of $x$ and $y$ respectively, condition (i) can be restated as follows:
\begin{enumerate}
\item[(i')]
We have $\lambda(\tilde x)\leq \gamma \lambda(\tilde y)$ for every $\lambda\in\overline{\LimF(\prod_j (S_j,\Sigma_j))}$.
\end{enumerate}

Given $s<0$, we let $\tilde x_s$ denote the `cut-down' $\tilde x_s=[( (x_{s+t,j})_j )_{t\leq 0}]$, and similarly denote by $\tilde y_t$ the cut-downs of $\tilde y$.
Applying \autoref{prp:CompareClosureCone}, and using that $\LimF(\prod_j (S_j,\Sigma_j))$ is a subcone of $\FF(\prod_j (S_j,\Sigma_j))$, we see that~(i') is equivalent to:
\begin{enumerate}
\item[(ii')]
For every $s<0$ and $\gamma'>\gamma$, there exists $t<0$ such that
$\lambda(\tilde x_s)\leq \gamma' \lambda(\tilde y_t)$ for every $\lambda\in \LimF(\prod_j (S_j,\Sigma_j))$.
\end{enumerate}

It remains to verify that~(ii) and~(ii') are equivalent.

\smallskip

\emph{We show that~(ii') implies~(ii).}
To verify~(ii), let $s<0$ and $\gamma'>\gamma$.
Pick $s'\in(s,0)$ and $\gamma''\in(\gamma,\gamma')$.
By assumption, there exists $t<0$ such that $\lambda(\tilde x_{s'})\leq \gamma''\lambda(\tilde y_t)$ for every limit functional $\lambda$. 
Let us show that $t$ has the desired properties to verify~(ii).

Suppose that this is not the case.
Using that $\filter$ is an ultrafilter, this means that
\[
E := \big\{ j\in J : \widehat{x_{s,j}} \nleq \gamma' \widehat{y_{t,j}} \big\}
\]
belongs to $\filter$.
For each $j\in E$, choose $\lambda_j\in \FF(S_j)$ such that $\lambda_j(x_{s,j})>\gamma'\lambda_j(y_{t,j})$. 
By rescaling~$\lambda_j$ if necessary, we may assume that 
\[
\lambda_j(x_{s,j}) \geq 1 > \gamma'\lambda_j(y_{t,j})
\] 
for all $j\in E$. 
Set $\lambda_j=0$ for $j\in J\setminus E$, and let $\bar{\lambda}_\filter$ be the limit functional in $\FF(\prod_j(S_j,\Sigma_j))$ associated to $(\lambda_j)_j$.
Then, on the one hand
\begin{align*}
1 
\leq \lim_{j\to\filter}\lambda_j(x_{s,j})
\leq \sup_{s''<s'}\lim_{j\to\filter}\lambda_j(x_{s'',j})
= \bar{\lambda}_\filter(\tilde x_{s'}),
\end{align*}
while on the other hand
\begin{align*}
1 
\geq \gamma' \lim_{j\to\filter}\lambda_j(y_{t,j}) 
\geq \gamma' \bar{\lambda}_\filter(\tilde y_t).
\end{align*}
Thus, $\bar{\lambda}_\filter(\tilde x_{s'})>\gamma''\bar{\lambda}_\filter(\tilde y_t)$, which is the desired contradiction.

\smallskip

\emph{We show that~(ii) implies~(ii').}
Given $s<0$ and $\gamma'>\gamma$, apply the assumption to obtain $t'<0$ such that the set $\{ j\in J : \widehat{x_{s,j}} \leq \gamma' \widehat{y_{t',j}} \}$ belongs to $\filter$.
Then set $t=t'/2$.
To verify~(ii'), let $\lambda_j\in \FF(S_j)$ for each $j$, and let $\bar{\lambda}_\filter$ be the associated limit functional in $\FF(\prod_j (S_j,\Sigma_j))$.
Then
\[
\bar{\lambda}_\filter(x_s)
\leq \lim_{j\to\filter}\lambda_j(x_{s,j})
\leq \lim_{j\to\filter}\gamma'\lambda_j(y_{t',j})
\leq \gamma'\bar{\lambda}_\filter(y_t).
\]
This proves (ii').
\end{proof}

\begin{ntn}
\label{ntn:leqN}
Given elements $x$ and $y$ in a partially ordered semigroup, and given $N\in\NN$, we write $x\leq_N y$ to mean that $nx\leq ny$ for all $n\geq N$. 
\end{ntn}

If elements $x$ and $y$ in a partially ordered semigroup satisfy $(M+1)x\leq My$ for some $M\in\NN$, then for $N:=(M+1)M$ we have $(n+1)x\leq ny$ for all $n\geq N$, and in particular $x\leq_Ny$;
see the proof of \cite[Proposition~5.2.13]{AntPerThi18TensorProdCu}.

The next result describes the extent to which the order in a \CuSgp{} can be recovered by the order on functionals.

\begin{lma}
\label{leqss}
Let $x',x$ and $y$ be elements in a \CuSgp{}.
Assume that $x'\ll x$ and that $\widehat{x}\leq\gamma\widehat{y}$ for some $\gamma\in(0,1)$.
Then there exist $M,N\in\NN$ such that $(M+1)x'\leq My$ and $x'\leq_N y$.
\end{lma}
\begin{proof}
The statement for $M$ follows from \cite[Theorem~5.2.18]{AntPerThi18TensorProdCu} or \cite[Proposition~2.2.2]{Rob13Cone}.
As observed above, the statement for $N$ follows immediately.
\end{proof}

The key to the solution of \autoref{pbm:denseLimitFctl} will be to quantify $M$ and~$N$ in \autoref{leqss} depending on $\gamma$, but not the elements $x',x,y$.
In the context of scaled \CuSgp{s}, we also need to record the `size' of $x$ and $y$ as determined by the scale.
To formalize this, given a scaled \CuSgp{} $(S,\Sigma)$, and $d\in\NN$, we define the \emph{$d$-fold amplification} of $\Sigma$ as 
\[
\Sigma^{(d)} = \big\{ x\in S : \text{for each } x'\ll x \text{ there are } x_1,\ldots,x_d\in\Sigma \text{ with } x'\ll x_1+\ldots+x_d \big\},
\]
for $d\geq 1$, and as $\Sigma^{(0)}=\{0\}$.

Note that, for any $x\in S$, if there exists $\tilde{x}$ such that $x\ll\tilde{x}$, then $x\in\Sigma^{(d)}$ for some $d\in\NN$.

\smallskip

Recall that an ultrafilter $\filter$ is said to be \emph{countably incomplete} if there exists a sequence $(E_n)_n$ in $\filter$ with $\bigcap_n E_n=\varnothing$. 

\begin{thm}
\label{prp:CharDensityUltraproduct}
Let $(S_j,\Sigma_j)_{j\in J}$ be a collection of scaled \CuSgp{s} that satisfy \axiomO{5}.
Let $\filter$ be a countably incomplete ultrafilter on $J$.
The following are equivalent:
\begin{enumerate}[{\rm (i)}]
\item
The set of limit functionals $\LimF(\prod_j (S_j,\Sigma_j))$ is dense in $\FF(\prod_\filter (S_j,\Sigma_j))$.
\item
For every $\gamma\in(0,1)$ and $d\in\NN$ there exist $N=N(\gamma,d)\in\NN$ and $E=E(\gamma,d)\in\filter$ such that: 
\[
\widehat{x} \leq \gamma\widehat{y} \quad \text{ implies } \quad x\leq_Ny, \qquad
\text{for all }  j\in E \text{ and } x,y\in \Sigma^{(d)}_j.
\]
\end{enumerate}
\end{thm}
\begin{proof}
\emph{We show that~(i) implies~(ii).}
To reach a contradiction, assume that~(ii) does not hold.
Using that~$\filter$ is an ultrafilter, this means that there exist $\gamma\in(0,1)$ and $d\in\NN$ such that for every $N\in\NN$ the set
\[
E_N := \big\{ j\in J : \text{there exist } x,y\in\Sigma^{(d)}_j \text{ with } \widehat{x} \leq \gamma\widehat{y} \text{ but } x\nleq_Ny \big\}
\]
belongs to $\filter$. Using that $\filter$ is countably incomplete, we may choose a decreasing sequence $(E_N')_{N\in\NN}$ in $\filter$ such that $\bigcap_N E_N'=\varnothing$ and  $E_N'\subseteq E_N$ 
for each $N$.

We now pick suitable $x_j',x_j'',x_j,y_j',y_j\in S_j$ for each $j\in J$.
If $j\in J\backslash E_0'$, we simply set $x_j'=x_j''=x_j=0$ and $y_j'=y_j=0$.
If $j\in E_N'\backslash E_{N+1}'$ for $N\geq 0$, then we use that $E_N'\subseteq E_N$ to choose $x_j,y_j\in S_j$ such that
\[
x_j,y_j\in\Sigma^{(d)}_j, \quad
\widehat{x_j} \leq \gamma\widehat{y_j}, \andSep
x_j \nleq_{N}y_j.
\]
Next, choose $x_j',x_j'' \in S$ such that $x_j'\ll x_j''\ll x_j$ and $x_j'\nleq_N y_j$.
Pick $\gamma' \in (\gamma,1)$.
Then $\widehat{x_j''} \ll \gamma' \widehat{y_j}$, by \autoref{prp:wayBelowLFS}. 
This allows us to choose $y_j' \in S_j$ such that $y_j' \ll y_j$ and $\widehat{x_j''} \leq \gamma'\widehat{y_j'}$.

By \cite[Proposition~2.10]{AntPerThi20AbsBivariantCu}, for each $j\in J$ we can choose paths $(x_{t,j})_{t\leq 0}$ and $(y_{t,j})_{t\leq 0}$ in $S_j$ such that 
\[
x_{-2,j} = x_j', \quad
x_{0,j} = x_j'', \quad
y_{-1,j} = y_j', \andSep
y_{0,j} = y_j.
\]
	
Set $\vect{x}_t=(x_{t,j})_j$ and $\vect{y}_t=(y_{t,j})_j$ for $t\leq 0$.
Since $x_{t,j}$ and $y_{t,j}$ belong to~$\Sigma^{(d)}_j$ for each $j$, the elements $\tilde x:=[(\vect{x}_t)_{t\leq 0}]$ and $\tilde y:=[(\vect{y}_t)_{t\leq 0}]$ belong to the scaled product $\prod_j (S_j,\Sigma_j)$. 
Let $x=\pi_\filter(\tilde x)$ and $y=\pi_\filter(\tilde y)$.
We also consider $x_{s}$, $y_{s}$, images of the cut-downs $\tilde x_{s}=[(\vect{x}_{t+s})_{t\leq 0}]$ and $\tilde y_{s}=[(\vect{y}_{t+s})_{t\leq 0}]$ for $s<0$. 
	
Observe that the set of indices $j$ such that $\widehat{x_{0,j}}\leq \gamma'\widehat{y_{-1,j}}$ contains $E_0'$, and thus belongs to $\filter$. By \autoref{prp:CompClosureUltraproduct}, this implies that
$\widehat{x}(\lambda)\leq \gamma' \widehat{y}(\lambda)$ for every functional $\lambda$ in the closure of  $\LimF (\prod_\filter (S_j,\Sigma_j))$.
Since by assumption this set  is all of  $\FF(\prod_\filter (S_j,\Sigma_j))$, we conclude that $\widehat{x} \leq \gamma' \widehat{y}$.
Since $x_{-1}\ll x$,  by \autoref{leqss} there exists $M \in \NN$ such that
\[
(M + 1) x_{-1} \leq M y.
\]

Choose $z=[(\vect{z}_t)_{t\leq 0}]\in\cc_\filter$ such that $(M+1) \tilde x_{-1} \leq M\tilde y +z$.
We have $\tilde x_{-1}=[(\vect{x}_{t-1})_{t\leq 0}]$, and thus for $t=-1$ we obtain $s<0$ such that
\[
(M+1) \vect{x}_{-2} \llpw M\vect{y}_s + \vect{z}_s.
\]
Since $z\in\cc_\filter$, we have $\supp(\vect{z}_s)\notin\filter$.
Using that $E_{(M+1)M}'\in\filter$, we can choose $j\in J$ such that $j\notin\supp(\vect{z}_s)$ and $j\in E_{(M+1)M}'$.

Then
\[
(M+1) x_j'
= (M+1) x_{-2,j} 
\ll M y_{s,j} + z_{s,j}
= M y_{s,j}
\leq M y_j.
\]
As noted above \autoref{leqss}, this implies that $x_j' \leq_{(M+1)M} y_j$.
However, since $j\in E_{(M+1)M}'$, we have  $x_j' \not\leq_{(M+1)M} y_j$  by construction.
This is the desired contradiction.

\smallskip

\emph{We show that~(ii) implies~(i).}
By \autoref{prp:SubconeIsAll}, it suffices to show that for all $x,y\in\prod_\filter (S_j,\Sigma_j)$ with $\widehat x(\lambda)\leq \widehat y(\lambda)$ for all $\lambda$
in the closure of $\LimF(\prod_\filter  (S_j,\Sigma_j))$, we have $\widehat{x}\leq\widehat{y}$.
Thus, let  $x,y\in\prod_\filter (S_j,\Sigma_j)$ be such that $\widehat x(\lambda)\leq \widehat y(\lambda)$ for all $\lambda$
in the closure of $\LimF(\prod_\filter  (S_j,\Sigma_j))$.
Choose $\tilde x=[( (x_{t,j})_j )_{t\leq 0}]$ and $\tilde y=[( (y_{t,j})_j )_{t\leq 0}]$, lifts of $x$ and $y$ in $\prod_j (S_j,\Sigma_j)$.
Given $s<0$, we let $\tilde x_s$ denote the `cut-down' $\tilde x_s=[( (x_{s+t,j})_j )_{t\leq 0}]$, and similarly for $\tilde y_t$ for $t<0$.

Let $s<0$ and $\tfrac{k}{l}>\gamma'>1$ with $k,l\in\NN\setminus\{0\}$.
By \autoref{prp:CompClosureUltraproduct}, there exists $t<0$ such that
\[
E_0 := \big\{ j\in J : \widehat{x_{s,j}} \leq \gamma' \widehat{y_{t,j}} \big\}
\]
belongs to $\filter$.
Choose $d\in\NN$ such that $x_{s,j},y_{t,j}\in\Sigma_j^{(d)}$ for all $j$.
Applying the assumption for $\tfrac{l}{k}\gamma'$ and $d$, we obtain $N\in\NN$ and $E_1\in\filter$ such that 
\[
\widehat{v} \leq \tfrac{l}{k}\gamma' \widehat{w}
\quad \text{ implies } \quad 
v\leq_Nw, \quad \text{ for all }  j\in E_1 \text{ and } v,w\in \Sigma^{(d)}_j.
\]
For $j\in E_0\cap E_1$, we have
\[
\widehat{lx_{s,j}} \leq (\tfrac{l}{k}\gamma') \widehat{ky_{t,j}}
\]
and therefore
\[
lx_{s,j} \leq_N ky_{t,j}.
\]

This implies that $l\pi_\filter(\tilde x_s)\leq_N k \pi_\filter(\tilde y)=ky$.
Given $\lambda\in \FF(\prod_\filter (S_j,\Sigma_j))$, we obtain
\[
\lambda(\pi_\filter(\tilde x_s)) 
\leq \tfrac{k}{l} \lambda(y).
\]
Since this holds for every $s<0$ and for every $k,l$ with $\tfrac{k}{l}>1$, we obtain 
$\lambda(x)\leq\lambda(y)$, as desired.
\end{proof}

Let us now briefly comment on the version of the preceding theorem for functionals on products rather than ultraproducts.
Let $(S_j,\Sigma_j)_{j\in J}$ be a collection of scaled \CuSgp{s} that satisfy \axiomO{5}. Consider their scaled product $\prod_j (S_j,\Sigma_j)$.
For each $k\in J$, the projection map $\pi_k\colon \prod_j (S_j,\Sigma_j)\to S_k$ induces a cone morphism $\FF(\pi_k)\colon \FF(S_k)\to \FF(\prod_j (S_j,\Sigma_j))$, and we let $K_k$ denote the image of $\FF(\pi_k)$.

The next result is proven similarly to \autoref{prp:CompClosureUltraproduct}.
We omit the proof.

\begin{prp}
\label{prp:CompClosureProduct}
Let $\gamma\in\RR_+$, and let $x=[( (x_{t,j})_j )_{t\leq 0}]$ and $y=[( (y_{t,j})_j )_{t\leq 0}]$ in $\prod_j (S_j,\Sigma_j)$.
The following are equivalent:
\begin{enumerate}[{\rm (i)}]
\item
We have $\widehat{x}(\lambda) \leq \gamma\widehat{y}(\lambda)$ for every $\lambda$ in the closed subcone generated by $\bigcup_j K_j$. 
\item
For every $s<0$ and $\gamma'>\gamma$, there exists $t<0$ such that $\widehat{x_{s,j}} \leq \gamma' \widehat{y_{t,j}}$ for every~$j\in J$.
\end{enumerate}
\end{prp}

A proof similar to the proof of \autoref{prp:CharDensityUltraproduct}, using \autoref{prp:CompClosureProduct} instead of \autoref{prp:CompClosureUltraproduct}, leads to the next result. 
We omit the proof.

\begin{thm}
\label{prp:CharDensityProduct}
Let $(S_j,\Sigma_j)_{j\in J}$ be a collection of scaled \CuSgp{s} that satisfy \axiomO{5}.
The following are equivalent:
\begin{enumerate}[{\rm (i)}]
\item
The subcone generated by $\bigcup_j K_j$ is dense in $\FF(\prod_j (S_j,\Sigma_j))$.
\item
For every $\gamma\in(0,1)$ and $d\in\NN$ there exists $N=N(\gamma,d)\in\NN$ such that: 
\[
\widehat{x} \leq \gamma\widehat{y}  \quad \text{ implies } \quad x\leq_Ny, \quad
\text{ for all but finitely many }  j\in J \text{ and } x,y\in \Sigma^{(d)}_j.
\]
\end{enumerate}
\end{thm}

\section{Locally bounded comparison amplitude}
\label{sec:LCBA}

When specialized to powers and ultrapowers of a given \CuSgp{}, \autoref{prp:CharDensityUltraproduct}(ii) and \autoref{prp:CharDensityProduct}(ii) simplify to the same  comparison property, which we formalize in the following definition:

\begin{dfn}
\label{dfn:LBCA}
We say that a scaled \CuSgp{} $(S,\Sigma)$ has \emph{locally bounded comparison amplitude}, or (LBCA), if for every $\gamma\in(0,1)$ and $d\in\NN$ there exists $N=N(\gamma,d)\in\NN$ such that:
\[ 
\widehat{x} \leq \gamma \widehat{y} \quad \text{ implies } \quad x\leq_N y, \qquad  
\text{for all }  x,y\in \Sigma^{(d)}.
\]
\end{dfn}

Let $(S,\Sigma)$ be a scaled \CuSgp. 
Let $\filter$ be a free ultrafilter on some set.
Recall that we denote by $(S,\Sigma)_\filter$ the scaled \CuSgp{} ultrapower of $(S,\Sigma)$. 
Recall also that $\LimF((S,\Sigma)_\filter)$ denotes the set of limit functionals in $\FF((S,\Sigma)_\filter)$.

The next result follows from Theorems~\ref{prp:CharDensityUltraproduct} and~\ref{prp:CharDensityProduct}.

\begin{thm}
\label{prp:CharDensityUltrapower}
Let $(S,\Sigma)$ be a scaled \CuSgp{} that satisfies \axiomO{5}.
The following are equivalent:
\begin{enumerate}[{\rm (i)}]
\item
$(S,\Sigma)$ has (LBCA):
For every $\gamma\in(0,1)$ and $d\in\NN$ there exists $N=N(\gamma,d)\in\NN$ such that $\widehat{x}\leq\gamma\widehat{y}$ implies $x\leq_Ny$ for all $x,y\in\Sigma^{(d)}$.
\item
For some (equivalently, every) countably incomplete ultrafilter $\filter$, the set of limit functionals
$\LimF((S,\Sigma)_\filter)$ is dense in $\FF((S,\Sigma)_\filter)$.
\item
For some (equivalently, every) infinite set $J$, the subcone generated by $\bigcup_{j\in J} K_j$ is dense in $\FF(\prod_{j\in J} (S,\Sigma))$.
\end{enumerate}	
\end{thm}

\begin{pgr}[Comparison amplitude]
\label{par:ComparisonAmplitude}
Let $S$ be a \CuSgp.
We define the \emph{comparison amplitude} for $x,y\in S$ as
\[
\CompAmpl(x,y) = \min\{N\in\NN : x\leq_N y \big\},
\]
with the convention that $\CompAmpl(x,y)=\infty$ if there is no $N$ such that $x\leq_Ny$. Let $\Sigma$ be a scale on $S$. 
For $\gamma\in(0,1)$ and $d\in\NN$ consider the set
\[
C_{\gamma,d} = \big\{ (x,y) \in \Sigma^{(d)}\times \Sigma^{(d)} : \widehat{x} \leq \gamma\widehat{y} \big\}.
\]
Note then that $S$ has LCBA if and only if the comparison amplitude is bounded on each set $C_{\gamma,d}$. 
This explains the terminology in \autoref{dfn:LBCA}.

For elements $x$ and $y$ in a partially ordered semigroup, one writes $x<_s y$ if $(n+1)x\leq ny$ for some $n\in\NN$.
Given $x,y\in S$, we have $\widehat{x}<_s\widehat{y}$ if and only if $\widehat{x}\leq\gamma\widehat{y}$ for some $\gamma\in(0,1)$.
Thus, \autoref{leqss} shows that the comparison amplitude $\CompAmpl(x',y)$ is finite whenever $x',x,y$ satisfy $x'\ll x$ and $\widehat{x}<_s\widehat{y}$.
\end{pgr}

A partially ordered semigroup is said to be \emph{almost unperforated} if $x<_s y$ implies $x\leq y$ for all elements $x$ and $y$.
It follows that a \CuSgp{} $S$ is almost unperforated if and only if $\CompAmpl(x,y)=1$ for every $x,y$ with $x<_sy$.
In particular, an almost unperforated Cuntz semigroup has (LBCA) relative to any scale. The converse is not true in general. However,
we do have a converse under the additional assumption of almost divisibility. 
A \CuSgp{} $S$ is called almost divisible if for every $x',x\in S$ with $x'\ll x$ and $n\in \NN$ there exists $y\in S$ such that $ny\leq x$ and $x'\leq (n+1)y$;
see also \autoref{pgr:divisibility}.

\begin{prp}
\label{prp:almostDivCu}
Let $S$ be an almost divisible \CuSgp{} satisfying \axiomO{5}.
Then~$S$ has  (LBCA) for some (equivalently, every) scale on $S$ if and only if $S$ is almost unperforated.
\end{prp}
\begin{proof}
If $S$ is almost unperforated, then the comparison amplitude is globally bounded (by $1$), as noted in \ref{par:ComparisonAmplitude}.
In particular, $S$ has (LBCA) for every scale on $S$.

Suppose now that $S$ is almost divisible and let $\Sigma\subseteq S$ be a scale such that $(S,\Sigma)$ has (LBCA).
To verify that $S$ is almost unperforated, let $x,y\in S$ and $n\in \NN$ be such that $(n+1)x\leq ny$. 
Then $\widehat{x}\leq\gamma\widehat{y}$ with $\gamma=\tfrac{n}{n+1}<1$, and we have to show that $x\leq y$.
	
Choose $\gamma',\gamma''$ such that $\gamma<\gamma'<\gamma''<1$.
Let $x',x''$ be such that $x''\ll x'\ll x$. 
By \autoref{prp:wayBelowLFS} applied to $x'\ll x$ and $1<\tfrac{\gamma'}{\gamma}$, we have that $\widehat{x'}\ll \gamma'\widehat{y}$, which allows us to choose $y'',y'\in S$ such that $y'' \ll y'\ll y$ and $\widehat{x'}\leq \gamma'\widehat{y''}$. 
Choose $d\in \NN$ such that $x',y'\in \Sigma^{(d)}$.
By definition of (LBCA) applied to $\gamma''$ and $d$, there exists $N=N(\gamma'',d)\in\NN$ such that $\widehat{v}\leq \gamma'' \widehat{w}$, for $v,w\in \Sigma^{(d)}$, implies that $v \leq_Nw$. Let us increase $N$ if necessary so that we also have that $\gamma'\frac{N+1}{N-1}<\gamma''$. 
	
Applying the almost divisibility assumption to $x''\ll x'$ and $y''\ll y'$, we obtain elements $v$ and $w$ such that 
\[
(N-1)v\leq x', \quad 
x''\leq Nv, \quad 
Nw\leq y', \andSep
y''\leq (N+1)w.
\] 
Then 
\[
(N-1)\widehat{v}\leq \widehat{x'}\leq \gamma'\widehat{y''}\leq \gamma'(N+1)\widehat{w}.
\]
Hence, $\widehat{v}\leq \gamma''\widehat{w}$. 
Since we also have that $v,w\in \Sigma^{(d)}$, we obtain that $v\leq_N w$. 
Therefore,  $x''\leq Nv\leq Nw\leq y'\leq y$. 
Passing to the supremum over all $x''\ll x$, we get that $x\leq y$, as desired. 
\end{proof}

In particular, for a scaled \CuSgp{} $(S,\Sigma)$ that is almost divisible and satisfies \axiomO{5}, the set $\LimF((S,\Sigma)_\filter)$ is dense in $\FF((S,\Sigma)_\filter)$ if and only if $S$ is almost unperforated.

Given a unital \ca{} $A$, we will show in \autoref{BHdensity} that the set $\LDF(A)$ of lower-semicontinuous dimension functions is dense in the space $\DF(A)$ of dimension functions if and only if the comparison amplitude $\CompAmpl(x,y)$ is finite for all $x,y\in W(A)$ such that $y$ is full and $\widehat{x}<_s\widehat{y}$.
Blackadar and Handelman conjectured in \cite{BlaHan82DimFct} that $\LDF(A)$ is always dense in $\DF(A)$, and this has been confirmed for several classes of \ca{s};
see \autoref{pgr:BH}.

\section{A stronger density result and application to C*-algebras}
\label{sec:StrongerDensity}

In the previous section we obtained a characterization of the density of limit functionals on an ultraproduct of \CuSgp{s} satisfying \axiomO{5}. In this section we strengthen this result assuming that the \CuSgp{s} also satisfy \axiomO{6} and Edwards' condition (as defined in \autoref{pgr:Edwards}).

The Cuntz semigroups of \ca{s} always satisfy \axiomO{5}, \axiomO{6} and Edward's condition. Thus, it is this stronger result that we shall apply to the setting of \ca{s}. 
Furthermore, in the \ca{ic} setting the result can be reformulated as a density of limit quasitraces 
in ultraproducts of \ca{s}, by the identification between limit $2$-quasitraces and limit functionals given in  \autoref{QTFbijection}.

\begin{thm}
\label{thm:main2}
$(S_j,\Sigma_j)_{j\in J}$ be a collection of scaled \CuSgp{s} that satisfy  \axiomO{5}, \axiomO{6}, and Edwards' condition.
Let $\filter$ be a countably incomplete ultrafilter on~$J$.  
Then the following are equivalent:
\begin{enumerate}[{\rm (i)}]
\item	
The set of limit functionals $\LimF(\prod_\filter (S_j,\Sigma_j))$ is dense in $\FF(\prod_\filter (S_j,\Sigma_j))$.
\item
For every $\gamma\in(0,1)$ and $d\in\NN$ there exist $N=N(\gamma,d)\in\NN$ and $E=E(\gamma,d)\in\filter$ such that: 
\[
\widehat{x} \leq \gamma\widehat{y} \quad \text{ implies } \quad x\leq_Ny, \qquad
\text{for all }  j\in E \text{ and } x,y\in \Sigma^{(d)}_j.
\]
\item
There exists $M\in\NN$ such that for every $d\in\NN$ there exist  $N=N(d)\in\NN$ and $E=E(d)\in \filter$ such that 
\[
\widehat{x} \leq \widehat{y} 
\quad \text{ implies } Nx \leq  MNy, 
\qquad \text{for all }j\in E\text{ and } x,y\in\Sigma_j^{(d)}.
\] 
\end{enumerate}
\end{thm}
\begin{proof}
The equivalence of~(i) and~(ii) is \autoref{prp:CharDensityUltrapower}.

\emph{We show that~(ii) implies~(iii).}
We verify~(iii) with $M=2$. 
Let $d\in\NN$.
Applying~(ii) for $2d$ and $\gamma=\tfrac{1}{2}$, we obtain $N\in\NN$ and $E\in \filter$ such that $\widehat{x}\leq\tfrac{1}{2}\widehat{y}$ implies $x\leq_Ny$ for all $x,y\in\Sigma_j^{(2d)}$. To verify that $N$ and $E$ have the desired properties, let $j\in E$ and $x,y\in\Sigma_j^{(d)}$ satisfy $\widehat{x} \leq \widehat{y}$.
The elements~$x$ and $2y$ belong to $\Sigma^{(2d)}$ and satisfy $\widehat{x}\leq\tfrac{1}{2}\widehat{2y}$.
We thus deduce that $x\leq_N 2y$, and in particular $Nx\leq 2Ny$.

\smallskip

\emph{We show that~(iii) implies~(i).}
The argument is analogous to the proof of the implication `(ii)$\Rightarrow$(i)' in \autoref{prp:CharDensityUltraproduct}.
Let $M\in\NN$ as in~(iii).
By \autoref{prp:SubconeIsAll2}, it suffices to show that for all $x,y\in \prod_\filter (S_j,\Sigma_j)$ with $\lambda(x)\leq\lambda(y)$ for all 
$\lambda$ in the closure of $\LimF(\prod_\filter (S_j,\Sigma_j))$  we have $\widehat{x}\leq 2M\widehat{y}$.

Let $x,y\in\prod_\filter (S_j,\Sigma_j)$ be such that for all $\lambda$ in the closure of $\LimF(\prod_\filter (S_j,\Sigma_j))$, we have $\lambda(x)\leq\lambda(y)$.  
Let $\tilde x=[( (x_{t,j})_j )_{t\leq 0}]$ and $\tilde y=[( (y_{t,j})_j )_{t\leq 0}]$ in $\prod_j (S_j,\Sigma_j)$ be lifts of $x$ and $y$, respectively, that is, $x=\pi_\filter(\tilde x)$ and $y=\pi_\filter(\tilde y)$.
Given $s<0$, we let $\tilde x_s$ denote the `cut-down' $\tilde x_s=[( (x_{s+t,j})_j )_{t\leq 0}]$, and set $x_s=\pi_\filter(\tilde x_s)$. 
We define similarly $\tilde y_t$ and $y_t$ for $t<0$.

Let $s<0$. 
By \autoref{prp:CompClosureUltraproduct}, for $\gamma=1$ and $\gamma'=2$, there exists $t<0$ such that
\[
E_1 := \big\{ j\in J : \widehat{x_{s,j}} \leq 2 \widehat{y_{t,j}} \big\}\in \filter.
\]
Choose $d\in\NN$ such that $x_{s,j},y_{t,j}\in\Sigma_j^{(d)}$ for all $j\in E_1$.
Applying the assumption for $d$, we obtain $N\in \NN$ and $E_2\in \filter$ such that 
\[
\widehat{v} \leq \widehat{w} 
\quad \text{ implies } \quad Nv\leq MNw, 
\quad \text{for all } j\in E_2 \text{ and } v,w\in \Sigma_j^{(d)}.
\]
Let $E=E_1\cap E_2$.
For $j\in E$, we have $\widehat{x_{s,j}} \leq 2 \widehat{y_{t,j}} $, and so $Nx_{s,j} \leq 2MNy_{t,j}$.
This implies that 
\[
Nx_s= N\pi_\filter(\tilde x_s)\leq 2MN\pi_\filter(\tilde y)=2MNy.
\]
Evaluating on any functional $\lambda$ on $\prod_\filter (S_j,\Sigma_j)$, we  deduce that $\lambda(x_s) \leq 2M \lambda(y)$.
Since this holds for every $s<0$, we obtain 
$\lambda(x)\leq 2M\lambda(y)$, as desired.
\end{proof}

For the case of the ultrapower of a  trivially scaled \CuSgp{} (that is, $\Sigma=S$) the previous result adopts the following simpler form:

\begin{cor}
\label{prp:trivialScaleDenseFS}
Let $S$ be a \CuSgp{} satisfying \axiomO{5}, \axiomO{6}, and Edwards' condition. 
Let $\filter$ be a countably incomplete ultrafilter on a set $J$. 
The following are equivalent:
\begin{enumerate}[{\rm (i)}]
\item
The set $\LimF(S_\filter)$ is dense in $\FF(S_\filter)$.
\item
For every $\gamma\in(0,1)$ there exists $N\in\NN$ such that $\widehat{x}\leq\gamma\widehat{y}$ implies $x\leq_Ny$ for all $x,y\in S$.
\item
There exists $M\in \NN$ such that $\widehat x\leq \widehat y$ implies $x\leq My$ for all $x,y\in S$.
\end{enumerate}
\end{cor}

\begin{rmk}
We do not have a direct proof of the equivalence of (ii) and (iii) in \autoref{prp:trivialScaleDenseFS} that does not use ultrapowers and density of limit functionals.
\end{rmk}

Let $A$ be a \ca.
We remind the reader that the scale of $A$ is 
\[
\Sigma_A 
:= \left\{ x\in\Cu(A) : \, \parbox{7cm}{for every $x'\in\Cu(A)$ with $x' \ll x$ there exists $a\in A_+$ with $x \leq [a]$} \right\}.
\]
The $d$-fold amplification of $\Sigma_A$ is then defined as
\[
\Sigma_A^{(d)} 
:= \left\{ x\in\Cu(A) : \,
\parbox{7cm}{for each $x'\in\Cu(A)$ with $x' \ll x$ there are $x_1,\ldots,x_d\in\Sigma_A$  with $x' \ll x_1+\ldots+x_d$ }  \right\}.
\]

We noticed in \autoref{pgr:scales} that $\Sigma_A$ admits useful descriptions in terms of Cuntz classes of positive elements in $A$.
The next result shows that a similar result holds for~$\Sigma_A^{(d)}$ in terms of Cuntz classes of positive elements in $M_d(A)$.

\begin{prp}
\label{prp:AmplifiedScale}
Let $A$ be a \ca{} and let $d \in \NN$ with $d \geq 1$.
Then
\begin{align*}
\Sigma_A^{(d)} 
&= \left\{ x\in\Cu(A) : \,
\parbox{7cm}{there exists a sequence $(a_n)_n$ in~$M_d(A)_+$ such that $([a_n])_n$ is $\ll$-increasing with $x = \sup_n [a_n]$ } \right\} \\
&= \big\{ x\in\Cu(A) : \text{there exists } a \in M_d(A)_+ \text{ with }x\leq [a] \big\}.
\end{align*}
\end{prp}
\begin{proof}
After identifying $\Cu(A)$ with $\Cu(M_d(A))$, we view $\Sigma_{M_d(A)}$ as the subset 
\[
\big\{ x\in\Cu(A) : \text{for each $x'\in\Cu(A)$ with $x' \ll x$ there is $a\in M_d(A)_+$ with $x \leq [a]$} \big\}
\]
of $\Cu(A)$.
Then, by \cite[Lemma~3.3]{ThiVil22arX:Glimm}, $\Sigma_{M_d(A)}$ agrees with the two displayed sets of the statement.
It remains to verify that $\Sigma_A^{(d)} = \Sigma_{M_d(A)}$,

Let $x \in \Cu(A)$.
To show the inclusion `$\subseteq$', assume that $x \in \Sigma_A^{(d)}$.
Given $x' \in \Cu(A)$ with $x' \ll x$, there are $x_1,\ldots,x_d \in \Sigma_A$ such that $x' \leq x_1+\ldots+x_d$.
We obtain $a_1,\ldots,a_d \in A_+$ such that $x_j \leq [a_j]$ for $j=1,\ldots,d$.
Then the diagonal matrix $a := \textrm{diag}(a_1,\ldots,a_d)$ belongs to $M_d(A)_+$ and we have $x' \leq [a]$.
Since this holds for every $x'$ with $x' \ll x$, we obtain $x \in \Sigma_{M_{d}(A)}$.

To show that other inclusion, assume that $x \in \Sigma_{M_{d}(A)}$.
Pick $a \in M_d(A)_+$ such that $x \leq [a]$.
Given $x' \in \Cu(A)$ with $x' \ll x$, we find $\varepsilon>0$ such that $x' \leq [(a-\varepsilon)_+]$. Using an approximate unit $(u_\lambda)_\lambda$ in $A$, for sufficiently large $\lambda_0$ the diagonal matrix $u=\textrm{diag}(u_{\lambda_0},\dots,u_{\lambda_0})$ satisfies $\|a-uau\|\leq \varepsilon$. 
Then
\[
(a-\varepsilon)_+ 
\precsim uau
\precsim u^2
\sim u,
\]
and it follows that
\[
x' \leq [(a-\varepsilon)_+] \leq [u] = [u_{\lambda_0}]+\ldots+[u_{\lambda_0}],
\]
with $[u_{\lambda_0}] \in \Sigma_A$.
Since this holds for every $x'$ with $x' \ll x$, we get $x \in \Sigma_A^{(d)}$.
\end{proof}

As mentioned in \autoref{pgr:functionals}, given $\tau\in\QT(A)$,  we obtain a functional $d_\tau\in \FF(\Cu(A))$ defined as
$d_\tau([a])=\lim_n\tau(a^{\frac{1}{n}})$ for all $[a]\in \Cu(A)$. 
Moreover, the correspondence $\tau\mapsto d_\tau$ is an isomorphism of topological cones between $\QT(A)$ and $\FF(\Cu(A))$.
Through this identification, the function $\widehat{[a]}$ induced by a Cuntz semigroup element $[a]\in \Cu(A)$ on $\FF(\Cu(A))$ may be regarded as a function on~$\QT(A)$.  
In the sequel we make this identification and thus regard $\widehat{[a]}$ as having domain~$\QT(A)$, that is, $\widehat{[a]}(\tau)=d_\tau(a)$ for $\tau\in \QT(A)$.

\begin{thm}
\label{prp:MainThm}
Let $A$ be a \ca{} and let $\filter$ be a countably incomplete ultrafilter on a set $J$. 
The following are equivalent:
\begin{enumerate}[{\rm (i)}]
\item
The set of limit $2$-quasitraces $\LimQT(A_\filter)$ is dense in $\QT(A_\filter)$.
\item
For every $\gamma\in(0,1)$ and $d\in\NN$ there exists $N=N(\gamma,d)\in\NN$ such that 
\[
\widehat{[a]}\leq\gamma \widehat{[b]} 
\quad \text{ implies } \quad [a]\leq_{N}[b],
\qquad \text{for all } a,b\in M_d(A)_+.
\]
\item 
There exists $M\in\NN$ such that 	for every $d\in\NN$ there exists $N=N(d)\in\NN$ such that 
\[
\widehat{[a]}\leq \widehat{[b]}
\quad \text{ implies } \quad N[a]\leq MN[b],
\qquad \text{for all } a,b\in M_d(A)_+.
\]
\end{enumerate}
\end{thm}
\begin{proof}
By \autoref{QTFbijection}, statement~(i) is equivalent to proving the density of limit functionals in $\FF((\Cu(A),\Sigma_A)_\filter)$.
Then, using \autoref{thm:main2}, we see that~(i) implies~(ii), and a similar argument as in the proof of said theorem shows that~(ii) implies~(iii). 
Let us show that (iii) implies (i).

Assume (iii).
We will verify that condition (iii) in \autoref{thm:main2} is satisfied, which then implies~(i).
Let $d \in \NN$ and consider $M\in\NN$ and $N:=N(2d)$ as given from the assumption~(iii).
To verify (iii) in \autoref{thm:main2}, let $x,y\in \Sigma^{(d)}_A$ satisfy $\widehat{x} \leq \widehat{y}$.

Applying \autoref{prp:AmplifiedScale}, we can write $x$ and $y$ as suprema of rapidly increasing sequences $x=\sup_{n} [a_n]$ and $y=\sup_n [b_n]$, with $a_n,b_n\in M_{d}(A)_+$. 
Using \autoref{prp:wayBelowLFS} and reindexing conveniently we may assume that $\widehat{[a_n]}\leq 2\widehat{[b_n]}$ for all $n$. 
Since $[a_n]$ and $2[b_n]$ are Cuntz classes of positive elements in $M_{2d}(A)$, we obtain by the choice of $M$ and $N$ that $N[a_n]\leq MN [b_n]$.
Passing to the supremum over $n$, we get $Nx\leq MN y$. 
\end{proof}

A \ca{} $A$ is said to be \emph{stable} if $A\cong A\otimes\mathcal{K}$.

\begin{cor}
\label{prp:denseQTstable}
Let $A$ be a stable \ca, and let $\filter$ be a countably incomplete ultrafilter on a set $J$.
The following are equivalent:
\begin{enumerate}[{\rm (i)}]
\item
The set of limit $2$-quasitraces $\LimQT(A_\filter)$ is dense in $\QT(A_\filter)$.
\item
For every $\gamma\in(0,1)$ there exists $N\in\NN$ such that $\widehat{[a]}\leq\gamma\widehat{[b]}$ implies $[a]\leq_N[b]$ for all $a,b\in A_+$.
\item
There exists $M\in \NN$ such that $\widehat{[a]}\leq \widehat{[b]}$ implies $[a]\leq M[b]$ for all $a,b\in A_+$.
\end{enumerate}
\end{cor}
\begin{proof}
Since $A$ is stable, the scale $\Sigma_A$ in $\Cu(A)$ is all of $\Cu(A)$.
The result then follows from \autoref{prp:MainThm}.
Alternatively, we use the same argument as in the proof of \autoref{prp:MainThm} to deduce the result from \autoref{prp:trivialScaleDenseFS}.
\end{proof}

Using that traces form a closed subset among quasitraces, we obtain:

\begin{cor}
Let $A$ be a \ca{} such that every lower semicontinuous $2$-quasitrace on $A$ is a trace (for example, if $A$ is exact), and let $\filter$ be a countably incomplete ultrafilter on a set $J$.
Assume that $\Cu(A)$ satisfies the conditions of \autoref{prp:MainThm}. 
Then every lower semicontinuous $2$-quasitrace on $A_\filter$ is a trace.
\end{cor}

\begin{pgr}[Comparison]
\label{pgr:comparison}
Let $S$ be a \CuSgp.
Recall that the relation $<_s$ on $S$ is defined by setting $x<_sy$ if there is $k\in\NN$ such that $(k+1)x\leq ky$. 
Given $m\in \NN$, one says that $S$ has \emph{$m$-comparison} if, for all $x,y_0,\dots,y_m\in S$, the condition $x<_sy_j$ for $j=0,\dots,m$ implies $x\leq\sum_{i=0}^m y_i$;
see \cite[Definition~2.8]{OrtPerRor12CoronaStability}.
Note that $S$ is almost unperforated if and only if it has $0$-comparison. 
\end{pgr}

A \ca{} is said to be \emph{nowhere scattered} if it has no nonzero, elementary ideal-quotients;
see \cite{ThiVil21arX:NowhereScattered}.

The next theorem is essentially a consequence of \autoref{prp:denseQTstable} and of  \cite[Theorem~8.12]{AntPerRobThi22CuntzSR1}. 

\begin{thm}
\label{cor:nowscatt}
Let $A$ be a stable, nowhere scattered \ca{} of stable rank one, and let $\filter$ be a countably incomplete ultrafilter on a set $J$. 
The following are equivalent:
\begin{enumerate}[\rm (i)]
\item
$\LimQT(A_\filter)$ is dense in $\QT(A_\filter)$,
\item
$\Cu(A)$ is almost unperforated (equivalently, $A$ has strict comparison of positive elements).
\end{enumerate}
\end{thm}
\begin{proof}
By \autoref{prp:denseQTstable} the density of $\LimQT(A_\filter)$ in $\QT(A_\filter)$ is equivalent to the statement
\begin{equation}
\tag{$*$}
\text{There exists $M\in \NN$ such that $\widehat{x} \leq \widehat{y}$ implies $x\leq My$ for all $x,y\in \Cu(A)$.}
\end{equation} 
If $\Cu(A)$ is almost unperforated, then ($*$) holds for $M=2$, which shows that~(ii) implies~(i).

On the other hand, it is shown in \cite[Theorem~8.12]{AntPerRobThi22CuntzSR1} that ($*$) implies that $\Cu(A)$ is almost unperforated whenever $A$ is a separable, nowhere scattered \ca{} of stable rank one.
Separability, however, can be dropped, as we show in \autoref{prp:RemoveSeparable} below. 
The result thus follows.
\end{proof}

The next result removes the separability assumption from \cite[Theorem~8.12]{AntPerRobThi22CuntzSR1}. 
To this end we use the model theory of \ca{s}.

\begin{thm}
\label{prp:RemoveSeparable}
Let $A$ be a nowhere scattered \ca{} of stable rank one. 
The following are equivalent:
\begin{enumerate}[\rm (i)]
\item 
The Cuntz semigroup $\Cu(A)$ has $m$-comparison for some $m\geq 0$.
\item 
There exist $M \in \NN$ and $\gamma \in (0,1)$ such that $\widehat{x} \leq \gamma\widehat{y}$ implies that $x \leq M y$ for all $x,y$ in $\Cu(A)$.
\item 
The Cuntz semigroup $\Cu(A)$ is almost unperforated.
\end{enumerate}	
\end{thm}
\begin{proof}
It is shown in \cite[Theorem~8.12]{AntPerRobThi22CuntzSR1} that~(i) implies~(ii), and that~(iii) implies~(i).
It remains to prove that~(ii) implies~(iii).
We may assume that $A$ is stable.
Assume (ii) and suppose that $x,y\in\Cu(A)$ satisfy $x<_s y$.
Choose $a, b\in A_+$ with $x=[a]$ and $y=[b]$.
We need to show that $[a] \leq [b]$.
By assumption, there exist $M \in \NN$ and $\gamma \in (0,1)$ such that $\widehat{v} \leq \gamma\widehat{w}$ implies that $v \leq Mw$ for all $v,w\in\Cu(A)$.

Apply the downward L\"owenheim--Skolem theorem for \ca{s}, \cite[Theorem~2.6.2]{FarHarLupRobTikVigWin21ModelThy}, to obtain a separable sub-\ca{} $B\subseteq A$ that is an elementary submodel of $A$, and that contains $a$ and $b$.
By \cite[Lemma~3.8.2]{FarHarLupRobTikVigWin21ModelThy} and \cite[Proposition~4.11]{ThiVil21arX:NowhereScattered}, $B$ is nowhere scattered and has stable rank one. 
Further, by \cite[Theorem~8.1.3]{FarHarLupRobTikVigWin21ModelThy}, the induced map $\Cu(B)\to \Cu(A)$ is an order embedding.

Let us verify that $\Cu(B)$ satisfies~(ii) for the given $M$ and $\gamma$.
So let $v,w \in \Cu(B)$ satisfy $\lambda(v) \leq \gamma\lambda(w)$ for all $\lambda\in F(\Cu(B))$.
The proof of \cite[Lemma~9.2]{AntPerRobThi22CuntzSR1} is easily adapted to show that this implies that $\lambda(v) \leq \gamma\lambda(w)$ for all $\lambda\in F(\Cu(A))$.
By choice of $M$ and $\gamma$, we obtain that $v \leq Mw$ in $\Cu(A)$.
Since $\Cu(B)\to \Cu(A)$ is an order embedding, we get $v \leq Mw$ in $\Cu(B)$.

We can now apply \cite[Theorem~8.12]{AntPerRobThi22CuntzSR1} to $B$ to show that $\Cu(B)$ is almost unperforated. 
Since the induced map $\Cu(B)\to \Cu(A)$ is an order embedding by \cite[Theorem~8.1.3]{FarHarLupRobTikVigWin21ModelThy}, we obtain that $x <_s y$ in $\Cu(B)$ and hence $x\leq y$ in $\Cu(B)$, which in turn gives $x \leq y$ in $\Cu(A)$, as desired. 
\end{proof}

\section{Density of normalized limit quasitraces and a conjecture of Blackadar--Handelman}
\label{sec:BH}

We now turn to the question of density of limit functionals normalized at a full compact element, and similarly to the question about the density of normalized limit quasitraces for a unital \ca. 
Here we prove the first part of \autoref{thmtraces} from the introduction.
We focus on ultraproducts over a free ultrafilter, but similar results are valid for products.

Let $S$ be a \CuSgp{} satisfying \axiomO{5}. Let $u\in S$ be a compact, full element. Recall that we regard the pair $(S,u)$ as a scaled \CuSgp{}  endowed with the scale $\Sigma_u=\{x : x\leq u\}$. 

Let $((S_j,u_j))_{j\in J}$ be a family of pairs of a \CuSgp{} and a full compact element. 
Let $\filter$ be an ultrafilter on $J$. 
In \autoref{scaledultra} we have defined $\prod_\filter (S_j,u_j)$ as the pair $(S,v)$, where $(S,\Sigma_v)=\prod_\filter (S_j,\Sigma_{u_j})$ and  $v\in S$ is the element induced in the ultraproduct by the constant path $(\vect{u})_{t\leq 0}$, with  $\vect{u}=(u_j)_{j}$ in $\CatPomProd_j S_j$.  

\begin{thm}
\label{thm:charFuDense}
Let $((S_j,u_j))_{j\in J}$ be a collection of pairs of a \CuSgp{} satisfying \axiomO{5} and a full compact element. 
Let $\filter$ be a countably incomplete ultrafilter on $J$. 
Let $(S,v)=\prod_\filter (S_j,u_j)$.
The following are equivalent:
\begin{enumerate}[{\rm (i)}]
\item
The set $\LimF_v(\prod_\filter (S_j,u_j))$ is dense in $\FF_v(\prod_\filter S_j)$.
\item
For every $\gamma\in(0,1)$ and $d\in\NN$ there exist $N=N(\gamma,d)\in\NN$ and $E=E(\gamma,d)\in \filter$ such that 
\[
\widehat{x} \leq \gamma \widehat{y} 
\quad \text{ implies } \quad 
x\leq_N y,
\]
for all $j \in E$ and all $x,y \in S_j$ such that $x,y \leq du_j$ and $u_j \leq dy$.
\item
There exists $M\in \NN$ such that for every $d\in\NN$ there exist $N=N(d)\in\NN$ and $E=E(d)\in \filter$ such that 
\[
\widehat{x} \leq \widehat{y}
\quad \text{ implies } \quad 
Nx\leq MNy,
\]
for all $j\in E$ and  all $x,y \in S_j$ with $x,y \leq du_j$ and $u_j \leq dy$.
\end{enumerate}
\end{thm}
\begin{proof}
Set $\bar{v}=[(\vect{u})_{t\leq 0}]$ in $\prod_j S$, so that $v=\pi_\filter(\bar{v})$ in $\prod_\filter S_j$.

\emph{We show that~(i) implies~(ii).}
The proof proceeds as in the proof of `(i)$\Rightarrow$(ii)' of \autoref{prp:CharDensityUltraproduct}, with minor modifications.

To reach a contradiction, assume that (ii) does not hold. Using that $\filter$ is an ultrafilter, this means that there exist $\gamma\in (0,1)$ and $d\in \NN$ such that for every $N\in \NN$ the set
\[
E_N := \big\{ j\in J : \text{there are $x,y \in S_j$ with $x,y\leq du_j$, $u_j \leq dy$, $\widehat{x}\leq\gamma\widehat{y}$, and $x\nleq_N y$} \big\}
\]	
belongs to $\filter$. 
Using that $\filter$ is countably incomplete, we may choose a decreasing sequence $(E_N')_{N\in\NN}$ in $\filter$ such that $\bigcap_N E_N'=\varnothing$ and $E_N'\subseteq E_N$ for each $N$.

We now pick suitable $x_j',x_j'',x_j,y_j',y_j\in S_j$ for each $j\in J$.
If $j\in J\backslash E_0'$, we simply set $x_j'=x_j''=x_j=0$ and $y_j'=y_j=u_j$.
If $j\in E_N'\backslash E_{N+1}'$ for $N\geq 0$, then we use that $E_N'\subseteq E_N$ to choose $x_j,y_j \in S_j$ such that
\[
x_j,y_j \leq d u_j, \quad 
u_j \leq dy_j, \quad
\widehat{x_j} \leq \gamma\widehat{y_j}, \andSep
x_j \nleq_{N}y_j.
\]
Next, choose $x_j', x_j'' \in S_j$ such that $x_j'\ll x_j''\ll x_j$ and $x_j'\nleq_N y_j $. 
Set $\gamma'=(1+\gamma)/2$.
Then $\widehat{x_j''} \ll \gamma' \widehat{y_j}$, by \autoref{prp:wayBelowLFS}. 
This allows us to choose $y_j' \in S_j$ such that
\[
y_j' \ll y_j, \quad 
u_j \leq dy_j', \andSep
\widehat{x_j''} \leq \gamma'\widehat{y_j'}.
\]

By \cite[Proposition~2.10]{AntPerThi20AbsBivariantCu}, for each $j\in J$ we can choose paths $(x_{t,j})_{t\leq 0}$ and $(y_{t,j})_{t\leq 0}$ in $S_j$ such that 
\[
x_{-2,j} = x_j', \quad
x_{0,j} = x_j'', \quad
y_{-1,j} = y_j', \andSep
y_{0,j} = y_j.
\]

Since $x_{t,j},y_{t,j} \leq du_j$ and $u_j \leq d y_{-1,j}$ for each $j$, the elements $\bar{x} := [(\vect{x}_t)_{t\leq 0}]$ and $\bar{y} := [(\vect{y}_t)_{t\leq 0}]$ in $\prod_j(S_j,u_j)$ satisfy that $\bar{x}, \bar{y} \leq d\bar{v}$ and $\bar{v} \leq d\bar{y}$. 
Set $x := \pi_\filter(\bar{x})$ and $y := \pi_\filter(\bar{y})$. 
The set of indices $j$ such that $\widehat{x_{0,j}} \leq \gamma'\widehat{y_{-1,j}}$ contains $E_0'$, and thus belongs to $\filter$. 
By \autoref{prp:CompClosureUltraproduct}, this implies that $\widehat{x}(\lambda)\leq \gamma' \widehat{y}(\lambda)$ for every functional $\lambda$ in the closure of $\LimF(\prod_\filter (S_j,u_j))$. 
Since, by assumption, this set contains the set of all normalized functionals $\FF_v(\prod_\filter S_j)$, the inequality $\widehat{x}(\lambda)\leq \gamma' \widehat{y}(\lambda)$ holds for all $\lambda\in \FF(\prod_\filter S_j)$ such that $\lambda(v)<\infty$. 
On the other hand, since $v\leq dy$, the same inequality is trivially valid for all $\lambda$ such that  $\lambda(v)=\infty$. 
We thus conclude that $\widehat{x}\leq \gamma'\widehat{y}$.
The remainder of the proof follows verbatim the proof of `(i)$\Rightarrow$(ii)' in \autoref{prp:CharDensityUltraproduct}.

\smallskip

\emph{We show that~(ii) implies~(iii).}
Let us prove that (iii) is valid with $M=2$. 
Let $d\in \NN$. 
By (ii), applied with $2d$ and  $\gamma=\tfrac{1}{2}$, there exist $N\in\NN$ and $E\in \filter $ such that for each $j\in E$ if $x,y\in S_j$ are such that $x,y\in (2d)u_j$, $u_j\leq (2d)y$, and $\widehat{x} \leq \tfrac{1}{2} \widehat{y}$, then $x \leq_N y$.
Then, for the same $N$ and $E$, we clearly have that if $x,y \in S_j$ are such that $x,y\leq du_j$, $u_j\leq dy$, and $\widehat x\leq \widehat y$, then $Nx \leq 2Ny$. 

\smallskip

\emph{We show that~(iii) implies~(i).}
The proof proceeds as in the proof of `(iii)$\Rightarrow$(i)' of \autoref{thm:main2}, with minor modifications.

By \autoref{MdensityFu}, it suffices to show that for all $x,y\in\prod_\filter (S_j,u_j)$, with $y$ full,  such that $\widehat x(\lambda)\leq \widehat y(\lambda)$ for all $\lambda$ in the closure of $\LimF_v(\prod_\filter (S_j,u_j))$, we have $\widehat{x}\leq 2M\widehat{y}$.
Thus, let  $x,y\in\prod_\filter (S_j,u_j)$ be such that $y$ is full and $\widehat x(\lambda)\leq \widehat y(\lambda)$ for all~$\lambda$ in the closure of $\LimF_v(\prod_\filter (S_j,u_j))$.

By \eqref{normalizedlimits}, we have that $\widehat x(\lambda)\leq \widehat y(\lambda)$ for all $\lambda\in \LimF(\prod_\filter (S_j,u_j))$ such that $\lambda(v)=1$. 
This easily extends to all $\lambda$ in $\overline{\LimF(\prod_\filter (S_j,u_j))}$ such that $\lambda(v)=1$. 
The latter equality can be relaxed to $\lambda(v)<\infty$.
On the other hand, since $y$ is full, $v\leq dy$ for some $d$, and so $\widehat x(\lambda)\leq \widehat y(\lambda)=\infty$ is valid for all $\lambda$ such that $\lambda(v)=\infty$. 
In summary, we have shown that $\widehat x(\lambda)\leq \widehat y(\lambda)$ for all $\lambda$ in the closure of $\LimF(\prod_\filter (S_j,u_j))$.

Choose $\bar{x} = [( (x_{t,j})_j )_{t\leq 0}]$ and $\bar{y} = [( (y_{t,j})_j )_{t\leq 0}]$, lifts of $x$ and $y$ in $\prod_j (S_j,u_j)$.
Given $s<0$, we consider the `cut-down' $\bar{x}_s := [( (x_{s+t,j})_j )_{t\leq 0}]$ and set $x_s := \pi_\filter(\bar{x}_s)$.
Since $y$ is full, there exists $d_0\in \NN$ such that $v\leq d_0 y$. 
Hence, there exist $t_0<0$ such that
\[
E_0 := \big\{ j\in J : u_j\leq d_0y_{t_0,j} \big\} \in \filter.
\]

Let $s<0$. 
Applying \autoref{prp:CompClosureUltraproduct} for the given $s$, as well as $\gamma=1$ and $\gamma'=2$, we obtain $t<0$ such that
\[
E_1 := \big\{ j\in J : \widehat{x_{s,j}} \leq 2 \widehat{y_{t,j}} \big\} \in \filter.
\]
We may assume that $t_0 \leq t$.

Choose $d\geq d_0$ such that $x_{s,j},y_{t,j}\leq du_j$ for all $j\in E_1$.
Applying the assumption~(iii) for~$2d$, we obtain $N\in \NN$ and $E_2\in \filter$ such that 
\[
\widehat{v} \leq \widehat{w} 
\quad \text{ implies } \quad Nv\leq MNw, 
\]
for all $j \in E_2$ and all $v,w \in S_j$ with $v,w \leq (2d)u_j$ and $u_j \leq (2d)w$.

Set $E := E_0\cap E_1\cap E_2$.
Let $j \in E$.
We have $x_{s,j},y_{t,j} \leq du_j$ and consequently $x_{s,j}, 2y_{t,j} \leq (2d)u_j$.
Since $j \in E_0$, $t_0 \leq t$ and $d_0 \leq d$, we also have $u_j \leq dy_{t,j} \leq (2d)y_{t,j}$.
Since $j \in E_1$, we further have $\widehat{x_{s,j}} \leq \widehat{2y_{t,j}} $.
For $j \in E_2$, we get $Nx_{s,j} \leq MN2y_{t,j}$.

This implies that 
\[
Nx_s= N\pi_\filter(\tilde x_s)\leq 2MN\pi_\filter(\tilde y)=2MNy.
\]
Evaluating on any functional $\lambda$ on $\prod_\filter (S_j,\Sigma_j)$, we  deduce that $\lambda(x_s) \leq 2M \lambda(y)$.
Since this holds for every $s<0$, we obtain 
$\lambda(x)\leq 2M\lambda(y)$, as desired.
\end{proof}

The next result follows from \autoref{thm:charFuDense} by specializing to the case of ultrapowers.
We will refer to condition~(ii) in \autoref{prp:CharFuDenseUpower} by saying that $(S,u)$ has \emph{(LBCA) for uniformly full elements}.

\begin{thm}
\label{prp:CharFuDenseUpower}
Let $(S,u)$ be a \CuSgp{} satisfying \axiomO{5} together with a full compact element $u \in S$.
Let $\filter$ be a countably incomplete ultrafilter on a set $J$, and consider the ultrapower $(S,u)_\filter$ with its canonical full compact element $v$.
The following are equivalent:
\begin{enumerate}[{\rm (i)}]
\item
The set $\LimF_v((S,u)_\filter)$ is dense in $\FF_v((S,u)_\filter)$.
\item
For every $\gamma\in(0,1)$ and $d\in\NN$ there exists $N=N(\gamma,d)\in\NN$ such that 
\[
\widehat{x} \leq \gamma \widehat{y} 
\quad \text{ implies } \quad 
x\leq_N y, \quad
\text{for all $x,y \in S$ with $x,y \leq du$ and $u \leq dy$}.
\]
\item
There exists $M\in \NN$ such that for every $d\in\NN$ there exists $N=N(d)\in\NN$ such that 
\[
\widehat{x} \leq \widehat{y}
\quad \text{ implies } \quad 
Nx\leq MNy, \quad
\text{for all $x,y \in S$ with $x,y \leq du$ and $u \leq dy$}.
\]
\end{enumerate}
\end{thm}

Applying the above \autoref{prp:CharFuDenseUpower} in combination with \autoref{QTFbijection} for the Cuntz semigroup of a unital \ca{}, we deduce the equivalence of (i)--(iii) in \autoref{thmtraces}.

\begin{pgr}[A conjecture of Blackadar--Handelman]
\label{pgr:BH}
Let $A$ be a unital \ca, and let $M_{\infty}(A)=\bigcup_{n=1}^\infty M_n(A)$, where $M_n(A)$ is regarded as a subalgebra of $M_{n+1}(A)$ through the upper-left corner embedding.
Following \cite[Section 3]{Cun78DimFct}, let us call a map $d\colon M_\infty(A)_+\to [0,\infty)$ a normalized \emph{dimension function} if $d(a\oplus b)=d(a)+d(b)$ for all $a,b\in M_\infty(A)_+$, $d(a)\leq d(b)$ if $a\precsim b$, and $d(1_A)=1$.
Let us endow the set $\DF(A)$ of normalized dimension functions with the topology of pointwise convergence.

Let $W(A)$ denote the classical (non-complete) Cuntz semigroup of $A$.
This is the subsemigroup of $\Cu(A)$ consisting of those elements that admit a representative in $M_\infty(A)$ (regarded as a subalgebra of $A\otimes\mathcal K$). 
Note that $\DF(A)$ is the set of normalized states on the partially ordered semigroup $W(A)$.   

Let $\LDF(A)$ denote the subset of $\DF(A)$ of lower semicontinuous (normalized) dimension functions. 
Blackadar and Handelman conjectured in \cite{BlaHan82DimFct} that $\LDF(A)$ is always dense in $\DF(A)$,  and verified this in the commutative case; 
see \cite[Theorem~I.2.4]{BlaHan82DimFct}.
The conjecture was also verified for simple, exact, $\mathcal{Z}$-stable \ca{s} in \cite[Theorem~B]{BroPerTom08CuElliottConj}, and this was further generalized in \cite[Theorem~5.2.5]{Sil16Thesis} to include (not necessarily simple) \ca{s} with finite radius of comparison.

In the result below we offer a characterization of when $\LDF(A)$ is dense in $\DF(A)$ in terms of finiteness of the comparison amplitude.
Another characterization was obtained in \cite[Theorem~5.1.1]{Sil16Thesis}. 
Some parts of our argument follow a similar approach, which we include for completeness.
\end{pgr}

\begin{thm}
\label{BHdensity}
Let $A$ be a unital \ca. 
The following are equivalent:
\begin{enumerate}[{\rm (i)}]
\item 
The set $\LDF(A)$ is dense in $\DF(A)$.
\item
For any $x,y\in W(A)$ with $y$ full, $x<_sy$ if and only if $\widehat{x}<_s\widehat{y}$.
\end{enumerate}
\end{thm}
\begin{proof}
In (ii), we only need to prove the backwards implication.

\emph{We show that~(i) implies~(ii).}
Let $x,y\in W(A)$ be such that $y$ is full and $\widehat{x}<_s\widehat{y}$.
Then there is $k\in\NN$ such that $(k+1)\widehat{x}\leq k\widehat{y}$ and thus $(k+1)d(x)\leq kd(y)$ for every $d\in\LDF(A)$.
Since $\LDF(A)$ is dense in $\DF(A)$, this implies that $(k+1)d(x)\leq kd(y)$ for any $d\in\DF(A)$.
Since $y$ is full, this implies $x<_sy$;
see \cite[Proposition~5.2.13]{AntPerThi18TensorProdCu}.

\smallskip

\emph{We show that~(ii) implies~(i).}
Let $K=\overline{\LDF(A)}$ in $\DF(A)$. 
Using \cite[Lemma~2.9]{BlaRor92ExtendStates}, we need to show that, for $x,y\in W(A)$, if $d(x)<d(y)$ for every $d\in K$, then $d(x)<d(y)$ for every $d\in \DF(A)$.

Thus, let $x,y\in W(A)$ such that $d(x)<d(y)$ for all $d\in K$. 
The function $K\to \RR$ given by $d\mapsto d(y)-d(x)$ is strictly positive and continuous, hence there is $\delta>0$ such that $d(y)-d(x)>\delta$. 
Choose $n\in\NN$ such that $n\delta >1$ and we get
\[
nd(x)+1<nd(y)\text{ for all }d\in K.
\]
This implies that
\[
n\lambda(x) + \lambda([1]) \leq n\lambda(y)
\]
for every $\lambda\in \FF_{[1]}(\Cu(A))$, and consequently for every $\lambda\in \FF(\Cu(A))$ such that $\lambda([1])<\infty$.  
Adding $\lambda([1])$ on both sides extends the inequality to all functionals $\lambda\in \FF(\Cu(A))$, since both sides are then $\infty$ whenever $\lambda([1])=\infty$. 
Hence,
\[
n\widehat{x}+2\widehat{[1]} \leq n\widehat {y} + \widehat{[1]}.
\]
Given any $k\in\NN$, we deduce that $k(n\widehat{x}+2\widehat{[1]}) <_s (k+1)(n\widehat {y} + \widehat{[1]})$, and using the assumption at the second step, we obtain
\[
knx + 2k[1] 
= k(nx+2[1]) 
<_s (k+1)(ny+[1])
= kny + ny + (k+1)[1].
\]

Since $y\in W(A)$, there exists $k$ such that $ny\leq (k-2)[1]$, and therefore 
\[
ny + (k+1)[1] \leq (2k-1)[1].
\] 
With this choice of $k$ we get
\[
knx + 2k[1] 
<_s kny + ny + (k+1)[1]
\leq kny + (2k-1)[1].
\]
Evaluating at any $d\in\DF(A)$, we see that $d(x)<d(y)$, as desired.
\end{proof}

\begin{thm}
\label{BHdensitycor}
Let $A$ be a unital \ca{} such that $(\Cu(A),[1])$ has (LBCA) for uniformly full elements.
Then $\LDF(A)$ is dense in $\DF(A)$. 
\end{thm}
\begin{proof}
It suffices to verify~(ii) of \autoref{BHdensity}.
Let $x,y\in W(A)$ such that $y$ is full and $\widehat{x}<_s\widehat{y}$.
We need to prove that $x <_s y$.

Using that $\widehat{x}<_s\widehat{y}$, we can choose $m\in\NN$ such that $(m+2)\widehat{x}\leq m\widehat{y}$.
Set $\gamma := \tfrac{m}{m+1}$ and notice that then $(m+2)\widehat{x}\leq\gamma(m+1)\widehat{y}$.
Since every element $z$ in $W(A)$ satisfies $z\ll\infty[1]$, and since $y$ is full, we can choose $d\in\NN$ such that $(m+2)x,(m+1)y\leq d[1]$ and $[1]\leq d(m+1)y$.

Applying that $(\Cu(A),[1])$ has (LBCA) for uniformly full elements for $\gamma$ and $d$, we obtain $N=N(\gamma,d)$, which we can apply to $(m+2)x$ and $(m+1)y$ to obtain $(m+2)x\leq_N(m+1)y$, whence $x<_s y$.
\end{proof}

\begin{pgr}[Radius of comparison]
Let $S$ be a \CuSgp, and let $u\in S$ be a compact, full element.
Following \cite[Deﬁnition~3.2.2]{BlaRobTikTomWin12AlgRC}, the \emph{radius of comparison} of~$(S,u)$, denoted by $\rc(S,u)$, is defined as the infimum over all $r\in[0,\infty)$ such that the following holds:
If $x,y\in S$ satisfy $\widehat{x}+r\widehat{u}\leq\widehat{y}$, then $x\leq y$.

The radius of comparison of a unital \ca{} $A$ is $\rc(A)=\rc(\Cu(A),[1])$.
\end{pgr}

\begin{lma}
\label{prp:FiniteRC-LBCAfull}
Let $S$ be a \CuSgp{} satisfying \axiomO{5}, let $u\in S$ be a compact, full element, and assume that $(S,u)$ has finite radius of comparison.
Then $(S,u)$ has (LBCA) for uniformly full elements.
\end{lma}
\begin{proof}
Choose $R\in\NN$ with $\rc(S,u)<R$.
Let $\gamma\in(0,1)$ and $d\in\NN$.
Choose $n=n(\gamma,d)\in\NN$ large enough such that
\[
\gamma < \frac{n}{n+d+1}.
\]
Then set $N'=R(n+d)$ and $N=N'(N'+1)$.

To see that $N$ has the desired properties, let $x,y\in S$ such that $x,y\leq du$ and $u\leq dy$ and $\widehat{x}\leq \gamma\widehat{y}$. 
Then $\widehat{x}\leq \frac{n}{n+d+1} \widehat{y}$, and we get
\[ 
(n+d+1)\widehat{x} \leq n\widehat{y}.
\]

Adding $\widehat{u}\leq d\widehat{y}$ and multiplying everything by $R$ we have 
\[ 
R(n+d+1)\widehat{x}+R\widehat{u} \leq R(n+d)\widehat{y}.
\]
Using $\rc(S,u)<R$, we get $R(n+d+1)x\leq R(n+d)y$, which implies $(N'+1)x\leq N'y$, and consequently $x\leq_N y$. 
\end{proof}

We recover \cite[Theorem 5.2.5 (1)]{Sil16Thesis} and \cite[Theorem 3.19]{ArcRobTik17Dixmier} for exact \ca{s}.

\begin{thm}
Let $A$ be a unital \ca{} with finite radius of comparison.
Then the set of limit $2$-quasitracial states $\LimQT_1(A_\filter)$ is dense in the set $\QT_1(A_\filter)$ of $2$-quasitracial states.
Further, $\LDF(A)$ is dense in $\DF(A)$.
\end{thm}
\begin{proof}
By \autoref{prp:FiniteRC-LBCAfull}, $(\Cu(A),[1])$ has (LBCA) for uniformly full elements, that is, condition~(ii) of \autoref{prp:CharFuDenseUpower} is satisfied.
Using \autoref{QTFbijection}, it follows from \autoref{prp:CharFuDenseUpower}(i) that $\LimQT_1(A_\filter)$ is dense in $\QT_1(A_\filter)$.
Further, it follows from \autoref{BHdensitycor} that $\LDF(A)$ is dense in $\DF(A)$.
\end{proof}

\section{Applications to simple, pure C*-algebras}
\label{sec:pure}

In this section, we show that every simple \ca{} that is $(m,n)$-pure in the sense of Winter is already pure;
see \autoref{pureReduction}.
An important ingredient in the proof is that $m$-comparison implies (LBCA);
see \autoref{prp:ComparisonLBCA}.

\begin{pgr}[Divisibility]
\label{pgr:divisibility}
Let $S$ be a \CuSgp.
Given $n\in\NN$, an element $x\in S$ is \emph{$n$-almost divisible} if for every $k\in\NN$ and every $x' \in S$ with $x' \ll x$, there exists $z\in S$ such that $kz\leq x$ and $x'\leq (k+1)(n+1)z$. 
If all elements in $S$ are $n$-almost divisible, then $S$ is said to be $n$-almost divisible.
One says that~$S$ is \emph{almost divisible} if it is $0$-almost divisible.

This notion of ($n$-)almost divisibility differs slightly from other notions considered in the literature, but it has been considered, for example, in \cite[Section~2.3]{RobTik17NucDimNonSimple} and \cite[Definition~7.3.4]{AntPerThi18TensorProdCu}. 
It is a more convenient notion as it behaves well with respect to natural constructions such as ultraproducts and direct limits.
\end{pgr}

\begin{rmk}
\label{rmk:mcomparison}
We remark that a \CuSgp{} $S$ has $m$-comparison if, and only if, for $x,y_0,\dots,y_m\in S$, the condition $\widehat{x} \leq \gamma\widehat{y_j}$ for some $\gamma<1$ and for $j=0,\ldots,m$ implies $x\leq\sum_{j=0}^m y_j$. 
This was observed in \cite[Lemma~2.1]{Rob11NuclDimComp}, and we offer a short sketch of the argument for completeness:
The backward implication is an immediate application of the definition. 
For the forward direction, if $S$ has $m$-comparison and $x,y_0,\dots,y_m\in S$ are such that $\widehat{x} \leq \gamma\widehat{y_i}$ for some $\gamma<1$ and all~$j$, let $x' \in S$ with $x' \ll x$ and apply \autoref{leqss} to conclude that $x' <_s y_j$ for each $j$. 
It then follows that $x' \leq \sum_{j=0}^m y_j$, and the desired inequality follows by passing to the supremum over all $x'$ with $x' \ll x$. 

In particular, if $S$ has $m$-comparison and $\widehat{x}\leq\widehat{y}$, we have $x\leq 2(m+1)y$. 
This follows from the previous argument applied to $\widehat{x}\leq \tfrac{1}{2}\widehat{2y}$.
\end{rmk}

\begin{prp}
\label{prp:ComparisonLBCA}
Let $S$ be a \CuSgp{} satisfying \axiomO{5}, \axiomO{6}, and Edwards' condition, and assume that $S$ has $m$-comparison for some $m$.
Then $S$ has (LBCA).
\end{prp}
\begin{proof}
This follows from \autoref{rmk:mcomparison} and \autoref{prp:trivialScaleDenseFS}.
\end{proof}

\begin{pgr}[Pure \ca{s}]
A \ca{} $A$ is said to be \emph{$(m,n)$-pure} provided $\Cu(A)$ has $m$-com\-par\-i\-son and is $n$-almost divisible. 
This notion was considered by Winter in \cite[Section~3]{Win12NuclDimZstable} in the context of the non-complete Cuntz semigroup $W(A)$, and replacing the condition $x'\leq (k+1)(n+1)z$ as above with the stronger inequality $x\leq (k+1)(n+1)z$. 
As defined here, this concept was introduced in \cite[Paragraph~2.3]{RobTik17NucDimNonSimple}. 
Note that, in this terminology, $(0,0)$-pure means that $\Cu(A)$ is almost unperforated and almost divisible. 
As in \cite{Win12NuclDimZstable}, a $(0,0)$-pure \ca{} will be called \emph{pure}.
\end{pgr}

Winter proved in \cite[Corollary~7.2]{Win12NuclDimZstable} that if $A$ is a unital, simple, separable \ca{} with locally finite nuclear dimension and which is $(m,n)$-pure for some $m,n\in\NN$, then $A$ is $\mathcal{Z}$-stable.
Using results of R{\o}rdam from \cite{Ror04StableRealRankZ}, this in turn implies that $A$ is pure;
see also \cite[Proposition 3.7]{Win12NuclDimZstable}
In \cite{Tik14NucDimZstable}, Tikuisis showed that the existence of a unit can be dropped.

We show here that $(m,n)$-pureness still implies pureness after dropping the assumptions of separability and of locally finite nuclear dimension.
Note that pureness is the \CuSgp{} analogue of $\mathcal{Z}$-stability, in the sense that it characterizes the \CuSgp{s} that tensorially absorb $\Cu(\mathcal{Z})$; 
see \cite[Theorem~7.3.11]{AntPerThi18TensorProdCu}.

\begin{thm}
\label{pureReduction}
A simple $(m,n)$-pure \ca{} is pure.
\end{thm}
\begin{proof}
Let $A$ be a simple \ca{} that is $(m,n)$-pure for some $m,n \in \NN$.
Assume, without loss of generality, that $A$ is stable.

Let us consider first the case that $A$ has no nontrivial lower semicontinuous $2$-quasitraces.
By the isomorphism between $\QT(A)$ and $\FF(\Cu(A))$ (\autoref{QTFbijection}), this means that $\Cu(A)$ only has the zero and the $\infty$ functionals.
Let us show that this implies that~$A$ is purely infinite (hence, pure).
Indeed, let $x,y\in \Cu(A)$ be nonzero elements.  
Using $n$-almost divisibility, find a nonzero $z\in \Cu(A)$ such that $2(m+1)z\leq y$.
Since $\widehat{x}\leq \widehat{z}$, we have by $m$-comparison that $x\leq 2(m+1)z$;
see \autoref{rmk:mcomparison}.
Thus, $x\leq y$.  
Since $x,y$ are arbitrary, we obtain $\Cu(A)=\{0,\infty\}$.

Let us now assume that $\FF(\Cu(A))$ has at least one element other than the $0$ and $\infty$ functionals. 
Since $\Cu(A)$ has $m$-comparison, it has $(LBCA)$ by \autoref{prp:ComparisonLBCA}.
It suffices now to show that $\Cu(A)$ is almost divisible, since then almost unperforation follows from \autoref{prp:almostDivCu}.

Fix a free ultrafilter $\filter$ on $\NN$.
Let $a\in A_+$ be a positive contraction. 
Let $C=\{a\}'\cap A_\filter$ and $I=\{a\}^\perp \cap A_{\filter}$ denote the commutant and annihilator of $\{a\}$ in~$A_\filter$, respectively. 
By \cite[Corollary~7.6]{RobTik17NucDimNonSimple}, there exists a unital embedding of the Jiang--Su algebra in $C/I$. 
(In the notation of \cite{RobTik17NucDimNonSimple}, $C/I$ is $\FF(C^*(a),A)$, which is a generalized central sequence algebra of the type studied by Kirchberg.) 
In particular, since $[1]\in \Cu(\mathcal{Z})$ is almost divisible (by \cite[Lemma~4.2]{Ror04StableRealRankZ}), so is the case for $[1]\in\Cu(C/I)$. 

Thus, for given $k\in\NN$ there exists $e\in C/I$ such that $k[e]\leq [1]\leq (k+1)[e]$ in $\Cu(C/I)$. 
Choose any positive lift $\bar e\in C\subseteq A_\filter$, and consider the element $b=a\bar e\in A_{\filter}$. 
Since $\Cu(C/I)\cong\Cu(C)/\Cu(I)$, induced by the quotient map $C\to C/I$, the inequality $k[e]\leq [1]$ means that $k[\bar{e}]\leq [1]+[z]$ in $\Cu(C)$, for some $[z]\in\Cu(I)$; 
see the comments prior to \autoref{pgr:Cuultraproducts}. 
Using that $b=a\bar{e}=\bar{e}a$ and that $az=za=0$, we obtain $k[b]\leq [a]$.
Likewise, it follows from $[1]\leq (k+1)[e]$ that $[a]\leq (k+1)[b]$ in $\Cu(C)$. 
Then $k[b]\leq [a]\leq (k+1)[b]$ in $\Cu(A_\filter)$, and thus $[a]$ is almost divisible in $\Cu(A_\filter)$.

We now show that $[a]$ is almost divisible as an element of $\Cu(A)$. 
Let $\varepsilon>0$. 
Choose $\delta>0$ such that $[(a-\varepsilon)_+] \leq (k+1)[(b-\delta)_+]$ in $\Cu(A_\filter)$, where $[b]$ is as in the previous paragraph. 
Then there exist $x,y\in M_{k+1}(A_\filter)$ such that 
\[
\|b\otimes 1_k - xax^*\| < \delta, \andSep
\|(a-\varepsilon)_+-y((b-\delta)_+\otimes 1_{k+1})y^*\| < \varepsilon.
\]
Let $(b_n)_n \in (\prod_n A)_+$, and $(x_n)_n,(y_n)_n \in \prod_n M_{k+1}(A)$ be lifts of $b$, $x$, and $y$. 
Then, with $b'=b_n$, $x'=x_n$, and $y'=y_n$ for sufficiently large $n$, we have that 
\[
\|b'\otimes 1_k - x'ax'^*\| < \delta, \andSep
\|(a-\varepsilon)_+-y'((b'-\delta)_+\otimes 1_{k+1})y'^*\| < \varepsilon.
\]
Now, working in $\Cu(A)$, we deduce from the first inequality that $k[(b'-\delta)_+]\leq [a]$ and from the second one that $[(a-2\varepsilon)_+]\leq (k+1)[(b'-\delta)_+]$. 
This shows that $[a]$ is almost divisible in $\Cu(A)$, as desired.
\end{proof}

\appendix

\section{Separation of functionals}
\label{sec:separateFS}

For a \CuSgp{} $S$ satisfying \axiomO{5}, we prove in this appendix a version of the Hahn--Banach separation theorem for $\FF(S)$;
see \autoref{prp:Separation}.
We deduce a version of the bipolar theorem, characterizing when a functional in $\FF(S)$ belongs to the closed cone generated by a subset of $\FF(S)$;
see \autoref{prp:Bipolar}.

Throughout this appendix we make the blanket assumption that $S$ is a \CuSgp{} satisfying \axiomO{5}.

Let us start with some preliminary definitions and lemmas. 
By a \emph{subcone} of a cone~$C$ we understand a subset $D\subseteq C$ that is closed under addition and multiplication by strictly positive scalars and that is a monoid.
Note that a subcone is not necessarily a submonoid since its origin may be different from the origin of the containing cone.

We say that a cone $C$ is \emph{cancellative} if $x+z=y+z$ implies $x=y$, for all $x,y,z\in C$.
Every $\RR$-vector space is a cancellative cone.
More generally, every subcone of an $\RR$-vector space is cancellative.
Using the Grothendieck completion, one sees that the converse also holds:
A cone is cancellative if and only if it is a subcone of an $\RR$-vector space.

We will use the following version of the Hahn--Banach Separation Theorem.

\begin{prp}
\label{prp:HB-Cones}
Let $F$ and $P$ be cancellative cones, let $\langle\freeVar,\freeVar\rangle\colon F\times P\to\RR$ be a map that is additive and $(0,\infty)$-homogeneous in each variable.
Let $D\subseteq F$ be a subcone that is closed in the $\sigma(F,P)$ topology associated to the pairing $\langle\freeVar,\freeVar\rangle$,  and that contains the origin of $F$.  Let $\mu\in F\setminus D$. 
Then there exist $f_1,f_2\in P$ such that
\[
\langle\lambda, f_1\rangle \leq \langle\lambda, f_2\rangle \hbox{ for all $\lambda\in D$},\andSep
\langle\mu, f_1\rangle > \langle\mu, f_2\rangle.
\]
\end{prp}	
\begin{proof}
Let $V$ denote the Grothendieck completion of $P$.
Then $V$ is an $\RR$-vector space.
Since $P$ is cancellative, the canonical map $P\stackrel{i}{\to} V$ is injective. Let us  use this map to identify $P$ with a subset of $V$.
Let $V^*$ denote the algebraic dual of all $\RR$-linear maps $V\to\RR$.
Define $\kappa\colon F\to V^*$ by
\[
\kappa(\lambda)(g_1-g_2) := \langle \lambda, g_1\rangle - \langle \lambda, g_2\rangle 
\]
for $\lambda\in F$ and $g_1,g_2\in P$.
One verifies that $\kappa$ is well defined, additive and $(0,\infty)$-homogeneous.
Let us denote by $\langle\freeVar,\freeVar\rangle_{V^*,V}\colon V^*\times V\to\RR$ the natural pairing given by evaluation. From our definitions, it is clear that the diagram
\[
\xymatrix@+1pc{
F\times P\ar[dr]^{\langle\freeVar,\freeVar\rangle}\ar[d]_{\kappa\times i} & \\
V^*\times V \ar[r]_{\quad\langle\freeVar,\freeVar\rangle_{V^*,V}}& \RR
}
\]
is commutative, that is, $\langle \lambda,g\rangle=\langle\kappa(\lambda),g \rangle_{V^*,V}$ whenever $\lambda\in F$ and $g\in P$.

Note that $\kappa(D)\subseteq V^*$ is a subcone containing the origin of $V^*$.
Let $\overline{\kappa(D)}\subseteq V^*$ be the closure of $\kappa(D)$ in the weak*-topology $\sigma(V^*,V)$.
Then $\overline{\kappa(D)}$ is a subcone that is closed in the $\sigma(V^*,V)$-topology and that contains the origin of $V^*$.

Let us verify that $\kappa(\mu)\notin\overline{\kappa(D)}$.
By assumption, $D\subseteq F$ is closed in the $\sigma(F,P)$ topology and $\mu\in F\setminus D$. 
Thus, there exist $g_1,\ldots,g_m\in P$ and $t_1,\ldots,t_m\in(0,\infty)$ such that the set
\[
U = \big\{ \lambda\in F : |\langle \lambda, g_j \rangle - \langle \mu, g_j \rangle |<t_j \text{ for } j=1,\ldots,m \big\}
\]
is disjoint from $D$. (Note that the sets of the form as $U$ above form a neighborhood basis of $\mu$, for different choices of $g_1,\ldots,g_m$ in $P$ and $t_1,\ldots,t_m$ in $(0,\infty)$.)
Now, the set
\[
U'=\big\{ \Lambda\in V^* : |\langle \Lambda, g_j \rangle_{V^*,V} - \langle \kappa(\mu), g_j \rangle_{V^*,V} |<t_j \text{ for } j=1,\ldots,m \big\}
\]
is a subset of $V^*$ that is open for $\sigma(V^*,V)$ and contains $\kappa(\mu)$. 
Since, as observed above, $\langle \lambda,g\rangle=\langle\kappa(\lambda),g \rangle_{V^*,V}$ for any $\lambda\in F$, $g\in P$, and $U$ is disjoint from $D$, we conclude that $U'$ is disjoint from $\kappa(D)$, as desired.

As a consequence of the Bipolar Theorem (see  \cite[Theorem~5, p.~62]{Gro73TVS}) applied to  the pair $V^*, V$, there exists $f\in V$ such that $\langle\Lambda,f\rangle_{V^*,V}\geq -1$ for all $\Lambda\in \overline{\kappa(D)}$ and $\langle\kappa(\mu),f\rangle<-1$.
(Note that a subcone of a $\RR$-vector space is called a convex cone in \cite{Gro73TVS}.)

Since $\overline{\kappa(D)}$ is a cone, we get from the first inequality that $\langle\Lambda,f\rangle\geq 0$ for all $\Lambda\in \overline{\kappa(D)}$. 
(Indeed, if $\langle\Lambda,f\rangle<0$ for some $\Lambda\in \overline{\kappa(D)}$, then $-1\leq \langle t\Lambda,f\rangle=t\langle \Lambda,f\rangle<0$ for all $t>0$, which is impossible.) 
Since $P$ spans $V$, we may write $f=f_2-f_1$, with $f_1,f_2\in P$. 
Then
\[
\langle\Lambda, f_1\rangle_{V^*,V} \leq \langle\Lambda, f_2\rangle_{V^*,V}
\]
for all $\Lambda\in\overline{\kappa(D)}$, and
\[
\langle\kappa(\mu), f_1\rangle_{V^*,V} > \langle\kappa(\mu), f_2\rangle_{V^*,V}.
\]
Now $f_1$ and $f_2$ have the desired properties.
\end{proof}

\begin{pgr}
\label{pgr:LC}
An \emph{algebraically ordered, compact cone} is a cone $C$ such that the algebraic pre-order is antisymmetric (if $\lambda+\lambda'=\mu$ and $\mu+\mu'=\lambda$, then $\lambda=\mu$) endowed with a compact, Hausdorff topology such that addition and scalar multiplication become jointly continuous;
see \cite[Section~3.1]{AntPerRobThi21Edwards}.
We use $\Lsc(C)$ to denote the set of maps $C\to [0,\infty]$ that are lower semicontinuous, zero-preserving, additive, and $(0,\infty)$-homogeneous.

For $f,g\in\Lsc(C)$, we write $f\lhd g$ provided there is $\varepsilon>0$ such that $f\leq (1-\varepsilon)g$ and $f$ is continuous at $\lambda\in C$ whenever $g(\lambda)<\infty$.
We use $\LL(C)$ to denote the set of functions in $\Lsc(C)$ that are suprema of $\lhd$-increasing sequences in $\Lsc(C)$.

Let $S$ be a \CuSgp{} satisfying \axiomO{5}.
Then $\FF(S)$ is an algebraically ordered, compact cone (see \cite[Proposition 2.2.3]{Rob13Cone} and \cite[Section 4]{EllRobSan11Cone}).
Given $x\in S$, recall that we denote by $\widehat{x}\colon \FF(S)\to [0,\infty]$ the function such that $\widehat{x}(\lambda)=\lambda(x)$ for all $\lambda\in \FF(S)$.
Then $\widehat{x}\in \LL(\FF(S))$ for all $x\in S$ (\cite{Rob13Cone}).
By \cite[Theorem~3.2.1]{Rob13Cone}, $\LL(\FF(S))$ is also the smallest subset of $\Lsc(\FF(S))$ containing $\widehat{x}$ for all $x\in S$ and closed under multiplication by scalars in $(0,\infty)$ and by suprema of increasing sequences.
Moreover, for each $f\in \LL(\FF(S))$ we have $f=\sup \frac{\widehat{x_n}}{k_n}$ for suitable $x_n\in S$ and $k_n\in\NN$ such that the sequence $(\frac{\widehat{x_n}}{k_n})_n$ is $\ll$-increasing.
It follows from \cite[Proposition~3.1.1, Theorem~3.2.1]{Rob13Cone} that $\LL(\FF(S))$ is a \CuSgp.

Given $u\in S$, recall from \autoref{pgr:functionals} that $\FF_u(S)$ denotes the convex set of functionals normalized at $u$. If $\widehat u$ is a continuous function on $\FF(S)$, then  $\FF_u(S)$ is a closed (hence compact) subset of $\FF(S)$. In particular, if $u$ is a compact element of $S$, then $\widehat u$ is continuous and  $\FF_u(S)$ is a compact subset of $\FF(S)$. 
\end{pgr}

We will make use of the following lemmas, which we state here for convenience. We remind the reader that we assume throughout the appendix that $S$ is a \CuSgp{} satisfying \axiomO{5}.

\begin{lma}[{\cite[Lemma~2.2.5]{Rob13Cone}}]
\label{prp:wayBelowLFS}
Let $x\ll y$ in $S$ and let $\alpha<\beta$ in $(0,\infty)$. 
Then $\alpha\widehat{x}\ll\beta\widehat{y}$ in $\Lsc(\FF(S))$ (and consequently also in $\LL(\FF(S))$).
\end{lma}

Following \cite[Section~5.1]{EllRobSan11Cone}, we define $\set(f)=\{\lambda\in C : f(\lambda)>1\}$ for $f\in\Lsc(C)$.
Large parts of the next result are shown in \cite[Proposition~5.1]{EllRobSan11Cone}.
We include a complete proof for the convenience of the reader.

\begin{lma}
\label{prp:CharWayBelowLscC}
Let $C$ be an algebraically ordered, compact cone.
(For example, $C=F(S)$ for a \CuSgp{} $S$ satisfying \axiomO{5}.)
Let $f,g\in \Lsc(C)$.
Consider the following statements:
\begin{enumerate}[{\rm (i)}]
\item
There exists $g'\in\Lsc(C)$ such that $f\leq g'\lhd g$.
\item
We have $\overline{\set(f)}\subseteq\set(g)$.
\item
The function $f$ is non-sequentially way-below $g$ in $\Lsc(C)$, that is, whenever an increasing net $(h_j)_j$ in $\Lsc(C)$ satisfies $g\leq\sup_jh_j$ then there exists $j'$ such that $f\leq h_{j'}$.
\item
We have $f\ll g$, that is, $f$ is sequentially way-below $g$ in $\Lsc(C)$.
\end{enumerate}
Then the implications `(i)$\Rightarrow$(ii)$\Rightarrow$(iii)$\Rightarrow$(iv)' hold.
If $g$ belongs to $\LL(C)$, then~(iv) implies~(i) and then all statements are equivalent.
\end{lma}
\begin{proof}
To verify that~(i) implies~(ii), let $(\lambda_j)_j$ be a net in $\set(f)$ converging to $\lambda\in C$.
We need to show $\lambda\in\set(g)$, that is, $g(\lambda)>1$.
This is clear if $g(\lambda)=\infty$.
On the other hand, if $g(\lambda)<\infty$, then $g'$ is continuous at $\lambda$ and therefore
\[
g'(\lambda)=\lim_j g'(\lambda_j)\geq \liminf_j f(\lambda_j) \geq 1.
\]
Since $g'\lhd g$, there is $\varepsilon>0$ such that $g' \leq (1-\varepsilon)g$, and so $g(\lambda) \geq \tfrac{1}{1-\varepsilon} > 1$.

To verify that~(ii) implies~(iii), let $(h_j)_j$ be an increasing net in $\Lsc(C)$ satisfying $g\leq\sup_jh_j$.
Then $(\set(h_j))_j$ is an increasing net of open subsets of $C$ satisfying $\set(g)\subseteq\bigcup_j\set(h_j)$.
Using that $\overline{\set(f)}$ is compact, we get $j'$ such that $\overline{\set(f)}\subseteq\set(h_{j'})$, which implies that $f\leq h_{j'}$.

It is clear that~(iii) implies~(iv).
Lastly,  assuming that $g$ belongs to $\LL(C)$, let us show that~(iv) implies~(i).
By definition of $\LL(C)$, there exists a $\lhd$-increasing sequence $(g_n)_n$ in $\Lsc(C)$ with supremum $g$.
Since $f\ll g$, we obtain $m$ such that $f\leq g_m$.
Then $f\leq g_m\lhd g$, as desired.
\end{proof}

Let $K\subseteq \FF(S)$ be a closed subcone.
Then $K$ is an algebraically ordered, compact cone.
Further, for each $f\in\Lsc(\FF(S))$ the restriction $f|_K$ belongs to $\Lsc(K)$.

\begin{lma}
\label{prp:restrictWayBelowLFS}
Let $K\subseteq \FF(S)$ be a closed subcone, and let $f,g\in \LL(\FF(S))$ satisfy $f\ll g$.
Then $f|_K$ is non-sequentially way-below $g|_K$ in $\Lsc(K)$ (and hence also $f|_K\ll g|_K$ in $\Lsc(K)$).
\end{lma}
\begin{proof}
Using that $g$ is the supremum of a $\lhd$-increasing sequence in $\Lsc(\FF(S))$, we obtain $g' \in \Lsc(\FF(S))$ such that $f\leq g'\lhd g$.
Then $f|_K\leq g'|_K\lhd g|_K$.
By \autoref{prp:CharWayBelowLscC}, we get that $f|_K$ is non-sequentially way-below $g|_K$.
\end{proof}

\begin{pgr}
\label{pgr:FI-PI}
Let $I\subseteq S$ be an ideal.
Let $\lambda_I\in \FF(S)$ denote the functional that is~$0$ on~$I$ and~$\infty$ otherwise. (Note that, with this notation, $\lambda_S$ is the zero functional.)
Define
\begin{equation}
\label{coneFI}
\FF_I(S) = \lambda_I + \big\{ \lambda\in \FF(S) : \lambda(x')<\infty\hbox{ whenever } x'\ll x \text{ for some } x\in I  \big\}.
\end{equation}

Then $\FF_I(S)$ is a subcone of $\FF(S)$ with origin $\lambda_I$. As noted in \cite[Proposition~3.2.3]{Rob13Cone}, $\FF_I(S)$ is cancellative.

For each $\lambda\in \FF(S)$ there exists a unique ideal $I\subseteq S$ such that $\lambda\in \FF_I(S)$; namely, the ideal generated by the set $\{x\in S:\lambda(x)<\infty\}$. 
This ideal is termed the \emph{support ideal} of $\lambda$; 
see \cite{AntPerRobThi21Edwards}.
In this way, the cone $\FF(S)$ is decomposed into the disjoint union of the cancellative subcones $\FF_I(S)$, where $I$ ranges through the ideals of $S$. 
\end{pgr}

We need a few more lemmas for the proof of \autoref{prp:SubconeIsAll}.

\begin{lma}
\label{prp:CharSupportIdeal}
Let $\mu\in \FF(S)$ with support ideal $I$, and let $x\in S$.
Then $\widehat{x}(\lambda_I)=0$ if and only if $\mu(x')<\infty$ for every $x' \in S$ satisfying $x' \ll x$.
\end{lma}
\begin{proof}
Since $I$ is the support ideal of $\mu$, we have $\mu\in \FF_I(S)$, and thus $\mu=\lambda_I+\mu_0$, where $\mu_0\in \FF(S)$ satisfies $\mu_0(x')<\infty$ whenever $x'\ll x$ and $x\in I$. 
Since $\lambda_I$ is idempotent, we also have $\mu=\lambda_I+\mu$.

Now assume that $\lambda_I(x)=0$ and let $x' \in S$ satisfy $x' \ll x$. 
Then $x' \leq x \in I$ and therefore $\lambda_I(x')=0$.
Since also $\mu_0(x')<\infty$, we have $\mu(x')=\lambda_I(x')+\mu_0(x')<\infty$.

Conversely, assume that $\mu(x')<\infty$ for every $x' \in S$ with $x' \ll x$.
Then from $\mu=\lambda_I+\mu$ we deduce that $\lambda_I(x')=0$ for every $x'$ way-below $x$.
Passing to the supremum over all such $x'$, we obtain $\lambda_I(x)=0$.
\end{proof}

\begin{lma}
\label{prp:technical1} 
Let $K$ be a closed subcone of $\FF(S)$ with $0\in K$.
Let $I$ be an ideal of~$S$. 
Suppose that for all $x,y\in S$ with $\widehat{x}|_{K}\leq\widehat{y}|_{K}$, we have $\widehat{x}(\lambda_I)\leq\widehat{y}(\lambda_I)$. 
Then $\lambda_I\in K$.
\end{lma}
\begin{proof}
Set
\[
C = K\cap \big\{ \lambda\in \FF(S) : \lambda\leq\lambda_I \big\}.
\]
Observe that $0\in C$, since $0\in K$ and $0\leq \lambda_I$. 
Further, $C$ is closed under sums and a closed subset of $\FF(S)$, as it is the intersection of two subsets with these properties. 
In particular, $C$ is upward directed. 
Set $\lambda=\sup C$, which is the limit of a net of elements in $C$, and thus belongs to $C$. 

Since $2\lambda\in C$, we have $2\lambda=\lambda$, which in turn implies that $\lambda=\lambda_J$ for some ideal~$J$ of~$S$; namely, $J=\{x\in S\colon \lambda(x)=0\}$. 
Further, since $\lambda_J\leq \lambda_I$, we have that $I\subseteq J$. 
We will show that $I=J$, and thus $\lambda_I\in K$. 
To reach a contradiction, suppose that $I\neq J$, and take $y\in J\setminus I$.
Choose $y'\in S$ such that $y'\ll y$ and $y'\notin I$.

If $x\in I$, then $\widehat{x}(\lambda_I)=0$, while  $\widehat{y'}(\lambda_I)=\infty$.
Thus, $\widehat{x}|_{K}\ngeq\widehat{y'}|_{K}$ (since otherwise $\widehat{x}(\lambda_I)\geq\widehat{y'}(\lambda_I)$ by assumption).
Choose $\lambda_x\in K$ such that $\widehat{x}(\lambda_x)<\widehat{y'}(\lambda_x)$.
Scaling the functional~$\lambda_x$ if necessary, we may assume that $1\leq \widehat{y'}(\lambda_x)$. 
Denote by $E(I)=\{x\in I\colon x=2x\}$, the set of idempotent elements in $I$. 
If now $x\in E(I)$, we have $\widehat{x}$ is also idempotent and thus $\widehat{x}(\lambda_x)=0$.

Note that $E(I)$ is an upward directed set. 
Since $\FF(S)$ is compact, there exists a convergent subnet $(\lambda_{x(j)})_j$ of $(\lambda_x)_{x\in E(I)}$. 
Let $\bar{\lambda}$ be its limit. 
As $K$ is closed, $\bar{\lambda}\in K$. 

Fix $x\in E(I)$. Then, for every $y\in E(I)$ with $x\leq y$, we have $\widehat{x}(\lambda_y)\leq\widehat{y}(\lambda_y)=0$.
Using that $\widehat{x}$ is lower semicontinuous, it follows that $\widehat{x}(\bar{\lambda})=0$.
We deduce that $\bar{\lambda}$ vanishes on $I$, and thus $\bar{\lambda}\leq\lambda_I$.
By definition, we get $\bar{\lambda}\in C$, and so $\bar{\lambda}\leq\lambda_J$.

On the other hand, using that $y'\ll y$ and that $\bar{\lambda}=\lim_j \lambda_{x(j)}$, we have
\[
1\leq  \limsup_j \lambda_{x(j)}(y') 
\leq \bar{\lambda}(y). 
\]
Since $y$ belongs to $J$, we have $\bar{\lambda}(y)\leq\lambda_J(y)=0$, a contradiction.
Thus, $I=J$.
\end{proof}

For \autoref{prp:SubconeIsAll} below, we shall only need the case $M=1$ of the next two results.
The general versions will be used later in the proof of \autoref{prp:SubconeIsAll2}.

\begin{lma}
\label{prp:technical2} 
Let $K$ be a closed subcone of $\FF(S)$ with $0\in K$. 
Let $\mu\in \FF(S)$ with support ideal $I$.
Suppose that there is $M\in (0,\infty)$ such that for all $x,y\in S$ with $\widehat{x}|_{K}\leq\widehat{y}|_{K}$ we have $\widehat{x}(\mu)\leq M\widehat{y}(\mu)$.
Then $\lambda_I\in K$.
\end{lma}
\begin{proof}
We will show that for all $x,y\in S$ with $\widehat{x}|_{K}\leq\widehat{y}|_{K}$, we have $\widehat{x}(\lambda_I)\leq\widehat{y}(\lambda_I)$.
It then follows from \autoref{prp:technical1} that $\lambda_I\in K$.

So let $x,y\in S$ satisfy $\widehat{x}|_{K}\leq\widehat{y}|_{K}$.
If $\widehat{y}(\lambda_I)=\infty$, then clearly $\widehat{x}(\lambda_I)\leq\widehat{y}(\lambda_I)$.
Thus, we may assume that $\widehat{y}(\lambda_I)=0$.
Choose a $\ll$-increasing sequence $(y_n)_n$ in~$S$ with supremum $y$.
By \autoref{prp:CharSupportIdeal}, we have $\mu(y_n)<\infty$ for every $n\in\NN$.
Let $x'\in S$ satisfy $x'\ll x$.
By \autoref{prp:wayBelowLFS}, we have $\widehat{x'}\ll 2\widehat{y}$ in $\Lsc(\FF(S))$.
Applying \autoref{prp:restrictWayBelowLFS}, we obtain $\widehat{x'}|_K\ll 2\widehat{y}|_K$ in $\Lsc(K)$, and we get $m\in\NN$ such that $\widehat{x'}|_K\leq 2\widehat{y_m}|_K$.
Using the assumption at the first step, we have
\[
\widehat{x'}(\mu)\leq 2M\widehat{y_m}(\mu) < \infty.
\]
Using \autoref{prp:CharSupportIdeal} again, it follows that $\widehat{x}(\lambda_I)=0$, and so $\widehat{x}(\lambda_I)\leq\widehat{y}(\lambda_I)$.
\end{proof}

\begin{lma}
\label{prp:separationLFStoS}
Let $K\subseteq \FF(S)$ be a closed subset, let $\mu\in \FF(S)$, and let $M\in (0,\infty)$. 
Suppose that $\widehat{x}|_K\leq \widehat{y}|_K$ implies $\widehat{x}(\mu)\leq M\widehat{y}(\mu)$ for all $x,y\in S$. 
Then $f|_K\leq g|_K$ implies $f(\mu)\leq Mg(\mu)$ for all  $f,g\in \LL(\FF(S))$.
\end{lma}	
\begin{proof}
Let $f,g\in \LL(\FF(S))$ satisfy $f|_K\leq g|_K$. 
Choose sequences $(x_n)_n$ and $(y_n)_n$ in~$S$ and natural numbers $(k_n)_n$ and $(l_n)_n$ such that $(\tfrac{\widehat{x_n}}{k_n})_n$ and $(\tfrac{\widehat{y_n}}{l_n})_n$  are $\ll$-increasing sequences in $\LL(\FF(S))$ with suprema $f$ and $g$, respectively;
see \autoref{pgr:LC}.
Applying \autoref{prp:restrictWayBelowLFS}, it follows that the sequences $(\tfrac{\widehat{x_n}}{k_n}|_{K})_n$ and $(\tfrac{\widehat{y_n}}{l_n}|_{K})_n$ are $\ll$-increasing in $\Lsc(K)$, with suprema $f|_{K}$ and $g|_{K}$, respectively.

Fix $m\in\NN$. 
Since $f|_K\leq g|_K$, there exists $n$ such that $\tfrac{\widehat{x_{m}}}{k_{m}}|_K \leq \tfrac{\widehat{y_n}}{l_n}|_K$, that is, $\widehat{l_nx_m}|_K\leq \widehat{k_my_n}|_K$. 
Hence, by assumption, $\widehat{l_nx_m}(\mu)\leq M \widehat{k_my_n}(\mu)$. 
Therefore 
\[
\frac{\widehat{x_{m}}}{k_{m}}(\mu) \leq M\frac{\widehat{y_n}}{l_n}(\mu)\leq Mg(\mu).
\] 
Passing to the supremum over all $m\in \NN$ we get $f(\mu)\leq Mg(\mu)$, as desired.
\end{proof}

For an ideal $I$ of $S$ define
\begin{equation}
\label{spacePI}
\PP_I(S)
:= \big\{ f'\in \LL(\FF(S)) : f'\lhd f\ll \widehat{x}\text{ for some } f\in \LL(\FF(S))\text{ and } x\in I \big\}
\end{equation}
and
\[
\widetilde \PP_I(S) 
:= \big\{ f|_{\FF_I(S)} : f\in \PP_I(S) \big\}.
\]
As established in the proof of \cite[Proposition~3.2.3]{Rob13Cone}, the functions in $\PP_I(S)$ are finite on $\FF_I(S)$.
It follows that $\widetilde \PP_I(S)$ is a subcone of the vector space of maps $\FF_I(S)\to\RR$.
In particular, $\widetilde \PP_I(S)$ is a cancellative cone.
We define a pairing $\langle\freeVar,\freeVar\rangle\colon \FF_I(S)\times \widetilde \PP_I(S)\to\RR$ by setting $\langle \lambda,f\rangle=f(\lambda)$.
This map is additive and $(0,\infty)$-homogeneous in each variable.

The restriction of the topology of $\FF(S)$ to $\FF_I(S)$ agrees with the $\sigma(\FF_I(S),\widetilde \PP_I(S))$ topology.
In other words, if a net $(\lambda_j)_j$ and a functional $\lambda$ are in $\FF_I(S)$, then $\lambda_j\to \lambda$ if and only if $f(\lambda_j)\to f(\lambda)$ for all $f\in \PP_I(S)$.
Indeed,  the forward implication follows since one can check that every function in $\PP_I(S)$ is continuous on $\FF_I(S)$.
The other implication is proven in \cite[Proposition~3.2.3]{Rob13Cone}.

An ideal $I$ of $S$ is called \emph{countably generated} if it is the smallest ideal containing a countable set $\{x_1,x_2,\ldots\}$.
In this case, $I$ is also singly generated by the element $x=\sum_{j=1}^\infty x_j$, and further $\infty\cdot x$ is the largest element in $I$.

We are now ready to prove the first separation result for subcones of $\FF(S)$.

\begin{thm}
\label{prp:Separation}
Let $K$ be a closed subcone of $\FF(S)$ with $0\in K$, and let $\mu\in \FF(S)\backslash K$. 
Then there exist $x,y\in S$ such that $\widehat{x}|_K\leq \widehat{y}|_K$ and $\widehat{x}(\mu)>\widehat{y}(\mu)$.
\end{thm}
\begin{proof}
To reach a contradiction, we assume that for all $x,y\in S$ with $\widehat{x}|_{K}\leq\widehat{y}|_{K}$ we have $\widehat{x}(\mu)\leq\widehat{y}(\mu)$. 
It then follows from \autoref{prp:separationLFStoS} (with $M=1$) that for all $f,g\in \LL(\FF(S))$ with $f|_K\leq g|_K$ we have $f(\mu)\leq g(\mu)$. Our goal is to reach a contradiction.
Let $I$ be the support ideal of $\mu$, so that $\mu\in \FF_I(S)$.
Applying \autoref{prp:technical2} (with $M=1$), we have $\lambda_I\in K$.

\textit{Claim: }\emph{There is a countably generated ideal $J\subseteq I$ such that $\lambda_J+\mu \notin \lambda_J+K$.
}

To prove the claim, let $\mathcal{S}$ denote the family of countably generated ideals contained in~$I$.
Ordered by inclusion, $\mathcal{S}$ is upward directed with $I=\bigcup\mathcal{S}$. Hence, $\lim_{J\in\mathcal{S}}\lambda_J=\lambda_I$.

To reach a contradiction, assume that for every $J\in\mathcal{S}$ we have $\lambda_J+\mu \in \lambda_J+K$, that is, there exists $\nu_J\in K$ such that $\lambda_J+\mu=\lambda_J+\nu_J$.
Since $K$ is compact, there exists a convergent subnet $(\nu_{J_\alpha})_\alpha$. Denote its limit by $\nu\in K$.
We have $\lim_{\alpha}\lambda_{J_\alpha}=\lambda_I$.
Using at the first step that $\mu\in \FF_I(S)$, and using at the last step that $\lambda_I\in K$, we get
\[
\mu
= \lambda_I+\mu
= \lim_{\alpha}( \lambda_{J_\alpha}+\mu )
= \lim_{\alpha}( \lambda_{J_\alpha}+\nu_{J_\alpha} )
=\lambda_I+\nu
\in K.
\]
This is the desired contradiction that proves the claim.

Fix $J$ as in the claim and set $D := (\lambda_J+K)\cap \FF_J(S)$. Then $D$  is a subcone of $\FF_J(S)$ closed in the $\sigma(\FF_J(S),\widetilde \PP_J(S))$ topology and containing the origin $\lambda_J$.  Since $\lambda_J+\mu\in \FF_J(S)\backslash D$, we can apply \autoref{prp:HB-Cones}  to the pairing between $\FF_J(S)$ and $\widetilde \PP_J(S)$ to obtain  $\tilde f_1,\tilde f_2\in \widetilde \PP_J(S)$ such that $\tilde f_1|_D\leq \tilde f_2|_D$ and $\tilde f_1(\lambda_J+\mu)>\tilde f_2(\lambda_J+\mu)$.  Choose $f_1,f_2\in \PP_J(S)$ such that $f_1|_{\FF_J(S)}=\tilde f_1$ and $f_2|_{\FF_J(S)}=\tilde f_2$

Using that $J$ is countably based, choose a $\ll$-increasing sequence $(z_n)_n$ in $J$ whose supremum is the largest element of $J$.
Note that $J$ is the support ideal of $\lambda_J+\mu$.
Given $n\in\NN$, we have $\widehat{z_n}(\lambda_J)=\lambda_J(z_n)=0$, and thus
\[
\widehat{z_n}(\mu)
=\widehat{z_n}(\lambda_J+\mu)
<\infty.
\]

Define 
\begin{equation*}
h=\sum_{n=0}^\infty \beta_n\widehat{z_n}\in \LL(\FF(S)),
\end{equation*}
where the scalars $(\beta_n)_n$ are strictly positive and chosen so that $h(\mu)\leq 1$. 
Now set
\[
g_1 = f_1 + h \andSep
g_2 = f_2 + h.
\]

Since $f_1,f_2\in \PP_J(S)$ (see \eqref{spacePI}), we have $f_1(\lambda_J)=f_2(\lambda_J)=0$.
Using that $h(\mu)<\infty$, we deduce that 
\[
g_1(\mu)
= f_1(\mu)+h(\mu)
= f_1(\lambda_J+\mu)+h(\mu)
> f_2(\lambda_J+\mu)+h(\mu)
= g_2(\mu).
\]

Let us show that $g_1|_K \leq g_2|_K$, which will yield the desired contradiction. 
Let $\lambda\in K$.
Assume first that $\lambda_J+\lambda\in \FF_J(S)$. 
Then $\lambda_J+\lambda\in D$.
Hence, $f_1(\lambda_J+\lambda)\leq f_2(\lambda_J+\lambda)$. 
Using that $h(\lambda_J)=0$, which is clear from the definition of $h$, we get 
\begin{align*}
g_1(\lambda)
&= f_1(\lambda)+h(\lambda)\\
&= f_1(\lambda+\lambda_J) + h(\lambda+\lambda_J)\\
&\leq f_2(\lambda+\lambda_J) + h(\lambda+\lambda_J)
= g_2(\lambda).
\end{align*}

Assume now that $\lambda_J+\lambda\notin \FF_J(S)$. 
From the definition of $\FF_J(S)$ (see \eqref{coneFI}) we deduce that
\[
\lambda \notin \big\{ \lambda'\in \FF(S) : \lambda'(x')<\infty\text{ whenever }x'\ll x \text{ for some }x\in J \big\}.
\]
Recall that $(z_n)_{n}$ is an increasing sequence with supremum the largest element of~$J$. 
Hence, we must have that $\lambda(z_n)=\infty$ for some $n$, and thus $h(\lambda)=\infty$. 
Then, $g_1(\lambda)=\infty=g_2(\lambda)$.
\end{proof}

\begin{cor}
\label{prp:SubconeIsAll}
Let $K\subseteq \FF(S)$ be a closed subcone containing $0$.
Assume that $\widehat{x}|_K\leq \widehat{y}|_K$ implies $\widehat{x}\leq\widehat{y}$, for all $x,y\in S$.
Then $K=\FF(S)$.
\end{cor}

\begin{exa}
Let $S=\{0,\infty\}$. 
Then $\FF(S)$ contains only two elements: the zero functional and the functional $\lambda_\infty$ that satisfies $\lambda_\infty(\infty)=\infty$. 
Set $K=\{\lambda_\infty\}$. 
Then~$K$ is a proper closed subcone of $\FF(S)$, such that for all $x,y\in S$ with $\widehat{x}_{|K}\leq\widehat{y}_{|K}$ we have $\widehat{x}\leq\widehat{y}$. 
Thus, the assumption that $K$ contains $0$ cannot be removed from \autoref{prp:SubconeIsAll}.
\end{exa}

We derive a kind of bipolar theorem for subsets of $\FF(S)$.

\begin{thm}
\label{prp:Bipolar} 
Let $K$ be a subset of $\FF(S)$, and let $\mu\in \FF(S)$.
The following are equivalent:
\begin{enumerate}[{\rm (i)}]
\item
The element $\mu$ belongs to the closed cone generated by $K\cup\{0\}$.
\item 
For all $x,y',y\in S$ with $\widehat{x}|_{K}\leq\widehat{y'}|_{K}$ and $y'\ll y$, we have $\widehat{x}(\mu)\leq\widehat{y}(\mu)$.
\item
For all $x',x,y',y\in S$ and $\gamma\in(0,1)$ satisfying $x'\ll x$, $\widehat x|_K\leq \gamma \widehat{y'}|_K$ and $y'\ll y$, we have $\widehat{x'}(\mu)\leq\widehat{y}(\mu)$.
\end{enumerate}
\end{thm}
\begin{proof}
To show that~(i) implies~(ii), let $C$ be the cone generated by $K\cup\{0\}$, that is,
\[
C = \big\{ t_1\lambda_1+\ldots+t_n\lambda_n : t_j\in(0,\infty),\lambda_j\in K\cup\{0\} \big\}.
\]
By assumption, $\mu\in\overline{C}$.

Let $x,y',y\in S$ satisfy $\widehat{x}|_{K}\leq\widehat{y'}|_{K}$ and $y'\ll y$.
We need to verify $\widehat{x}(\mu)\leq\widehat{y}(\mu)$.
Using that $\widehat{x}$ and $\widehat{y'}$ are linear and $(0,\infty)$-homogeneous, it follows that $\widehat{x}(\lambda)\leq\widehat{y'}(\lambda)$ for every $\lambda\in C$.
Let $(\lambda_j)_j$ be a net in $C$ that converges to $\mu$.
Then
\[
\widehat{x}(\mu)
= \mu(x)
\leq \liminf_j \lambda_j(x)
\leq \limsup_j \lambda_j(y') 
\leq \mu(y) 
= \widehat{y}(\mu).
\]

It is clear that~(ii) implies~(iii).
To show that~(iii) implies~(i), let $L$ be the closed cone generated by $K\cup\{0\}$.
To reach a contradiction, assume $\mu\notin L$.
Applying \autoref{prp:Separation}, we obtain $v,w\in S$ such that
\[
\widehat{v}|_L\leq \widehat{w}|_L, \andSep 
\widehat{v}(\mu)>\widehat{w}(\mu).
\]
Using that $\mu$ preserves suprema of increasing sequences, we can choose $v'$ such that
\[
v'\ll v, \andSep
\widehat{v'}(\mu)>\widehat{w}(\mu).
\]
Choose $v''\in S$ and $m\in\NN$ such that
\[
v'\ll v''\ll v, \andSep
m\widehat{v'}(\mu)>(m+2)\widehat{w}(\mu).
\]
By \autoref{prp:wayBelowLFS}, we have $m\widehat{v''}\ll(m+1)\widehat{v}$ in $\LL(\FF(S))$.
Let $(w_n)_n$ be a $\ll$-increasing sequence in $S$ with supremum $w$.
Applying \autoref{prp:restrictWayBelowLFS} at the first step, we get
\[
m\widehat{v''}|_L 
\ll (m+1)\widehat{v}|_L
\leq (m+1)\widehat{w}|_L
= \sup_n (m+1)\widehat{w_n}|_L
\]
in $\Lsc(L)$, which allows us to choose $l\in\NN$ such that $m\widehat{v''}|_L\leq (m+1)\widehat{w_l}|_L$.
Then $x'=mv'$, $x=mv''$, $y'=(m+2)w_l$, $y=(m+2)w$, and $\gamma=\tfrac{m+1}{m+2}$ satisfy
\[
x'\ll x, \quad
\widehat{x}|_K\leq \gamma\widehat{y'}|_K, \quad
y'\ll y, \andSep
\widehat{x'}(\mu)>\widehat{y}(\mu),
\]
which is the desired contradiction.
\end{proof}

\begin{prp}
\label{prp:CompareClosureCone}
Let $K$ be a subcone of $\FF(S)$ with closure $\overline K$.
Let $x,y\in S$ and $\gamma \in \RR_+$.
The following are equivalent:
\begin{enumerate}[{\rm (i)}]
\item
We have $\widehat{x}|_{\overline{K}} \leq \gamma \widehat{y}|_{\overline{K}}$.
\item
For every $x' \in S$ with $x' \ll x$ and every $\gamma'>\gamma$ there exists $y' \in S$ such that $y' \ll y$ and $\widehat{x'}|_K \leq \gamma'\widehat{y'}|_K$.
\end{enumerate}
\end{prp}
\begin{proof}
\emph{We show that~(ii) implies~(i).}
To verify~(i), let $\lambda\in \overline{K}$. 
Choose a net $(\lambda_j)_j$ in $K$ that converges to $\lambda$.
Let $x' \in S$ satisfy $x' \ll x$, and let $\gamma' > \gamma$.
By assumption, we obtain $y' \in S$ such that $y' \ll y$ and $\gamma\widehat{x'}|_K \leq \widehat{y'}|_K$. 
Then 
\[
\lambda(x') 
\leq \liminf_j \lambda_j(x')
\leq \limsup_j \gamma'\lambda_j(y')
\leq \gamma'\lambda(y). 
\]

Passing to the supremum over all $x'$ way-below $x$ on the left hand side, and to the infimum over all $\gamma'>\gamma$ on the right hand side, we get that $\lambda(x)\leq \lambda(y)$.

\smallskip

\emph{We show that~(i) implies~(ii).}
Suppose that $\widehat{x}|_{\overline{K}} \leq \gamma\widehat{y}|_{\overline{K}}$. 
Let $x' \in S$ satisfy $x' \ll x$ and let $\gamma'>\gamma$.
Then $\widehat{x'} \ll \tfrac{\gamma'}{\gamma}\widehat{x}$ in $\Lsc(\FF(S))$ (and hence in $\LL(\FF(S))$ by \autoref{prp:wayBelowLFS}.
Using \autoref{prp:restrictWayBelowLFS} at the first step, it follows that 
\[
\widehat{x'}|_{\overline{K}} 
\ll \frac{\gamma'}{\gamma}\widehat{x}|_{\overline{K}}
\leq \frac{\gamma'}{\gamma}\gamma\widehat{y}|_{\overline{K}} 
= \gamma' \widehat{y}|_{\overline{K}}
\]
in $\Lsc(\overline{K})$.
Choose a $\ll$-increasing sequence $(y_n)_n$ in $S$ with supremum $y$. 
Then $\widehat{x}|_{\overline{K}} \ll \sup_n \gamma'\widehat{y_n}|_{\overline{K}}$, and we obtain $n$ such that $\widehat{x'}|_{\overline{K}} \leq \gamma' \widehat{y_n}|_{\overline{K}}$.
Then $y' := y_n$ has the desired properties.
\end{proof}


\section{A stronger separation theorem}
\label{sec:strongerSeparateFS}

Our goal in this appendix is to obtain an improved version of \autoref{prp:SubconeIsAll}  imposing further properties on $S$. 
This is achieved in \autoref{prp:SubconeIsAll2}.

\begin{pgr}
\label{pgr:Edwards}
We say that a \CuSgp{} $S$ satisfies \emph{Edwards' condition} if, for any $\lambda\in \FF(S)$ and $x,y\in S$, one has
\[
\inf \big\{ \lambda_1(x)+\lambda_2(y) : \lambda_1+\lambda_2=\lambda \big\}
= \sup \big\{ \lambda(z) : z \leq x,y \big\};
\]
see \cite[Definition~4.1]{AntPerRobThi21Edwards}, \cite[Section~4]{Thi20RksOps} and \cite[6.3]{AntPerRobThi22CuntzSR1}. 
We remark that the expression on the left hand side of the above equality is equal to the infimum of the functions $\widehat{x}$ and $\widehat{y}$, taken in $\Lsc(\FF(S))$, evaluated at $\lambda$;
see \cite[Lemma~3.4]{AntPerRobThi21Edwards}.
The Cuntz semigroup of a \ca{} satisfies Edwards' condition;
see \cite[Theorem~5.3]{AntPerRobThi21Edwards}.
\end{pgr}

\begin{lma}
\label{prp:SREC}
Let $S$ be a \CuSgp{} satisfying \axiomO{5}, \axiomO{6} and Edwards' condition. 
Then this is also the case for $\LL(\FF(S))$.
\end{lma}
\begin{proof}
Set $T=\LL(\FF(S))$.
By \cite[Proposition~3.1.1, Theorem~3.2.1]{Rob13Cone}, $T$ is a \CuSgp{} satisfying \axiomO{5}.
By \cite[Lemma~4.0.1]{Rob13Cone}, $T$ satisfies \axiomO{6}. 
It remains to prove Edwards' condition for $T$.

For each $\Lambda\in \FF(T)$ there exists a unique $\lambda\in \FF(S)$ such that $\Lambda(h)=h(\lambda)$ for all $h\in T$, and this assignment is moreover additive; see the last paragraph of the proof of \cite[Proposition~3.1.1]{Rob13Cone}. 
That is, the functionals on $T$ arise as point evaluations on~$\FF(S)$. 
We use this below.

Given $\Lambda\in \FF(T)$, let $\lambda\in \FF(S)$ such that $\Lambda(h)=h(\lambda)$ for all $h\in T$.
To prove Edwards' condition for $\Lambda$, we must show that
\begin{equation}
\label{mustproveedwards}
\inf \big\{ f(\lambda_1)+g(\lambda_2) : \lambda_1+\lambda_2=\lambda \big\}
= \sup \big\{ h(\lambda) : h\leq f,g \big\},
\end{equation}
for all $f,g\in T$.
It is straightforward to show that the right hand side is dominated by the left hand side. 
Let us prove the opposite inequality.

By \cite[Theorem~3.5]{AntPerRobThi21Edwards}, the left hand side of \eqref{mustproveedwards}
is equal to $(f\wedge g)(\lambda)$, where $f\wedge g$ is the infimum of $f$ and $g$ in $\Lsc(\FF(S))$.
Choose sequences $(x_n)_n$ and~$(y_n)_n$ in $S$, and sequence $(k_n)_n$ and $(l_n)_n$ in $\NN\setminus\{0\}$, such that $(\tfrac{\widehat{x_n}}{k_n})_n$ and $(\tfrac{\widehat{y_n}}{l_n})_n$ are $\ll$-increasing sequences in $\Lsc(\FF(S))$ with suprema $f$ and   $g$, respectively; see \autoref{pgr:LC}. 
By \cite[Theorem~3.5]{AntPerRobThi21Edwards}, 
we have
\[
(f\wedge g)(\lambda) = \sup_n \Big(\tfrac{\widehat{x_n}}{k_n}\wedge \tfrac{\widehat{y_n}}{l_n}\Big)(\lambda),
\]
where the infima on both sides are taken in $\Lsc(\FF(S))$. 

This makes it clear that it is enough to prove $\leq$ in \eqref{mustproveedwards} for the case $f=\tfrac{\widehat{x}}{k}$ and~$g=\tfrac{\widehat{y}}{l}$ for $x,y \in S$ and $k,l \in\NN\setminus\{0\}$.
So assume that $f$ and $g$ are of this form.
Then, applying Edward's condition to $mx,ny\in S$ at the last equality, we have
\begin{align*}
(f\wedge g)(\lambda)  
&= \inf \big\{ f(\lambda_1)+g(\lambda_2) : \lambda=\lambda_1+\lambda_2 \big\} \\
&= \inf \left\{ \frac{\lambda_1(x)}{k}+ \frac{\lambda_2(y)}{l} : \lambda=\lambda_1+\lambda_2 \right\} \\
&= \frac{1}{kl} \inf \big\{ \lambda_1(lx)+\lambda_2(ky) \colon \lambda=\lambda_1+\lambda_2 \big\} \\
&= \frac{1}{kl} \sup \big\{ \lambda(z) : z\leq lx,ky \big\} \\
&\leq \sup \left\{ \frac{\widehat{z}}{kl}(\lambda) : \frac{\widehat{z}}{kl} \leq \frac{\widehat{x}}{k},\frac{\widehat{y}}{l} \right\}
\leq \sup \big\{ h(\lambda) : h\leq f,g \big\}
\end{align*}
as desired. 
The result thus follows.
\end{proof}

\begin{pgr}
\label{pgr:ray}
A \emph{ray} in a cancellative cone $C$ is a subset of the form $\RR_+\lambda$, for a non-zero element $\lambda\in C$. 
A ray $R$ is said to be \emph{extreme} if for all $\mu\in R$, whenever $\mu=\mu_1+\mu_2$ for some $\mu_1,\mu_2\in C$ we have $\mu_1,\mu_2\in R\cup\{0\}$;
see, for example, \cite[p. 79]{Phe01LNMChoquet}.

Let $S$ be a \CuSgp{} and let $I$ be an ideal of $S$. 
Let $\mu\in \FF_I(S)\setminus\{\lambda_I\}$ be a functional generating an extreme ray of $\FF_I(S)$. 
Define $\sigma_\mu\colon \FF(S)\to [0,\infty]$ as
\[
\sigma_\mu(\lambda)=
\begin{cases}
0 & \text{ if }\lambda\leq \lambda_I,\\
t & \text{ if }\lambda+\lambda_I=t\mu,\hbox{ where }t\in (0,\infty),\\
\infty & \text{ otherwise.}
\end{cases}
\]
\end{pgr}

The result below is proved for the Cuntz semigroup of a \ca{} in \cite[Proposition~7.4]{AntPerRobThi22CuntzSR1}. 
We follow here a similar argument in the context of \CuSgp{s}.

\begin{lma}
\label{prp:chiselLFS}
Let $S$ be a \CuSgp{} satisfying \axiomO{5}, \axiomO{6}, and Edwards' condition. 
Let $I$ be an ideal of $S$ and let $\mu\in \FF_I(S)\setminus\{\lambda_I\}$ be a functional generating an extreme ray of $\FF_I(S)$ 
Then $\sigma_\mu$ defined as above is the supremum of an increasing net of functions in $\LL(\FF(S))$.
\end{lma}
\begin{proof}
Consider the set 
\[
X = \big\{ f\in \LL(\FF(S)) \colon f(\mu)<1 \big\}.
\]

\textit{Claim~1: }\emph{Let $f_1,f_2\in X$ satisfy $f_1(\mu)\leq f_2(\mu)$. 
Then}
\[
f_1(\mu)
= \inf \big\{ f_1(\lambda_1)+f_2(\lambda_2) \colon \lambda_1+\lambda_2=\mu \big\}.
\]

The inequality `$\geq$' follows using $\lambda_1=\mu$ and $\lambda_2=0$.
To show the converse inequality, let $\lambda_1,\lambda_2\in \FF(S)$ satisfy $\lambda_1+\lambda_2=\mu$.
Then $\lambda_1+\lambda_I$ and $\lambda_2+\lambda_I$ belong to $\FF_I(S)$.
Since
\[
\tfrac{1}{2}\mu
= \tfrac{1}{2}(\lambda_1+\lambda_I) + \tfrac{1}{2}(\lambda_2+\lambda_I)
\]
and since $\mu$ generates an extreme ray of $\FF_I(S)$, we see that $\lambda_1+\lambda_I$ and $\lambda_2+\lambda_I$ are scalar multiples of $\mu$.
Say $\lambda_1+\lambda_I=t_1 \mu$ and $\lambda_2+\lambda_I=t_2\mu$. 
Now $(t_1+t_2)\mu=(\lambda_1+\lambda_I)+(\lambda_2+\lambda_I)=\mu$, and since $\mu\neq\lambda_I$, we have $t_1+t_2=1$. 
Thus, $\lambda_1+\lambda_I=t \mu$ and $\lambda_2+\lambda_I=(1-t)\mu$ for some $t\in[0,1]$ (where we use the convention that $0\cdot \mu=\lambda_I$, the neutral element of $\FF_I(S)$).
Using that $f_1(\mu),f_2(\mu)<\infty$ and $\lambda_I+\mu=\mu$, it follows that $f_1(\lambda_I)=f_2(\lambda_I)=0$.
Then
\begin{align*}
f_1(\lambda_1)+f_2(\lambda_2)
&= f_1(\lambda_1+\lambda_I)+f_2(\lambda_2+\lambda_I) \\
&= tf_1(\mu)+(1-t)f_2(\mu) \\
&\geq tf_1(\mu)+(1-t)f_1(\mu)=f_1(\mu). 
\end{align*}
This proves the claim. 

\textit{Claim~2: }\emph{$X$ is upward directed.}
To prove the claim, let $f_1,f_2\in X$.
Without loss of generality, we may assume that $f_1(\mu)\leq f_2(\mu)$.
By \autoref{prp:SREC}, $\LL(\FF(S))$ satisfies Edwards' condition.
Using this at the second step (see \eqref{mustproveedwards}), and using Claim~1 at the first step, we get
\[
f_1(\mu) = 
\inf \big\{ f_1(\lambda_1)+f_2(\lambda_2) \colon \lambda_1+\lambda_2=\mu \big\}
= \sup \big\{ g(\mu) \colon g\leq f_1,f_2,\, g\in \LL(\FF(S)) \big\}.
\]

Choose $\varepsilon>0$ such that $f_2(\mu)+\varepsilon<1$. 
Then choose $g',g\in \LL(\FF(S))$ such that
\[
g'\lhd g\ll f_1, f_2, \andSep 
g'(\mu) > f_1(\mu)-\varepsilon.
\]
Applying \cite[Lemma~3.3.2]{Rob13Cone} to $g'\lhd g\ll f_1+f_2$, we obtain $h\in \LL(\FF(S))$ and $C\in(0,\infty)$ such that
\[
g'+h=f_1+f_2, \andSep g'\leq Ch.
\]
We have
\[
f_1+h\geq g'+h=f_1+f_2.
\]
If $\lambda\in \FF(S)$ satisfies $h(\lambda)<\infty$, then $g'(\lambda)<\infty$, whence $f_1(\lambda)<\infty$.
This allows us to cancel $f_1(\lambda)$ to conclude that $h(\lambda)\geq f_2(\lambda)$.
If on the other  hand $h(\lambda)=\infty$, then again
$h(\lambda)\geq f_2(\lambda)$. Hence,  $h\geq f_2$,  and symmetrically   $h\geq f_1$. 
On the other hand, 
\[
f_1(\mu)-\varepsilon+h(\mu)
\leq g'(\mu)+h(\mu)
=f_1(\mu)+f_2(\mu),
\]
from which we deduce that $h(\mu)\leq f_2(\mu)+\varepsilon<1$. 
Thus, $h$ is an upper bound for~$f_1$ and~$f_2$ in~$X$.
This proves the claim.

Let us show that $\sup_{f\in X} f(\lambda)=\sigma_\mu(\lambda)$ for all $\lambda\in \FF(S)$, from which the lemma readily follows by the claim that we have just established.
We distinguish the following three cases:

\textit{Case~1: }\emph{Let $\lambda\in \FF(S)$ satisfy $\lambda\leq \lambda_I$.}
Given $f\in X$, using that $f(\mu)<1$ and $\mu+\lambda_I=\mu$, we have $f(\lambda_I)=0$, and so $f(\lambda)=0$.
This implies that $\sup_{f\in X} f(\lambda)=0=\sigma_\mu(\lambda)$.

\textit{Case~2: }\emph{Let $\lambda\in \FF(S)$ satisfy $\lambda+\lambda_I=t\mu$ for some $t\in (0,\infty)$.} 
Given $f\in X$, we saw in Case~1 that $f(\lambda_I)=0$, whence
\[
f(\lambda)=f(\lambda+\lambda_I)=f(t\mu)=tf(\mu)<t=\sigma_\mu(\lambda).
\]
This shows that $\sup_{f\in X} f(\lambda)\leq\sigma_\mu(\lambda)$.
To show the converse, note that there exists $g\in X$ with $g(\mu)>0$.
(Otherwise, $\mu$ would only take values in $\{0,\infty\}$ on $\LL(\FF(S))$, which would imply $\mu=2\mu$, a contradiction.)
Then, $f_n=\tfrac{n}{(n+1)g(\mu)}g$ belongs to $X$ and satisfies $f_n(\lambda)=\tfrac{n}{n+1}t$.
Therefore, $\sup_{f\in X} f(\lambda)\geq\sup_nf_n(\lambda)=t=\sigma_\mu(\lambda)$.

\textit{Case~3: }\emph{Suppose that we are in neither one of the two cases above.}
Then $\sigma_\mu(\lambda)=\infty$, and we need to show that $\sup_{f\in X} f(\lambda)=\infty$. 
Let $C\in(0,\infty)$.
It will suffice to argue that there exists $f\in X$ such that $f(\lambda)>C$. 
Since $\lambda+\lambda_I$ is not a scalar multiple of $\mu$ and the latter generates an extreme ray, we have $\lambda\not \leq 2C\mu$. 
Let $y\in S$ be such that $2C\mu(y)<\lambda(y)$. 
If $\mu(y)=0$ and $\lambda(y)=\infty$, then $f=\widehat{y}$ is as desired, and if $\mu(y)=0$ and $0<\lambda(y)<\infty$, then $f=\tfrac{2C}{\lambda(y)}\widehat{y}$ is as desired. 
Finally,~if~$\mu(y)>0$, then $f=\tfrac{1}{2\mu(y)}\widehat{y}$ satisfies
\[
f(\mu)=\tfrac{1}{2}<1, \andSep
f(\lambda)= \tfrac{\lambda(y)}{2\mu(y)}>C.
\]
Hence, $f$ is as desired.
\end{proof}

Let $C$ be a cone embedded in a locally convex topological $\RR$-vector space. 
A subset~$K$ of~$C$ is called a \emph{cap} if $K$ is compact, convex, and $C\backslash K$ is also convex. 
The cone~$C$ is said to be \emph{well capped} if it is the union of its caps; 
see, for example, \cite[p.~80]{Phe01LNMChoquet}. 
It was proved in \cite[Proposition~3.11]{AntPerRobThi21Edwards} that if $I$ is a countably generated ideal of a \CuSgp{} satisfying \axiomO{5}, then the cone $\FF_I(S)$ is well-capped.

The next result is an improved version of \autoref{prp:SubconeIsAll}.

\begin{thm}
\label{prp:SubconeIsAll2}
Let $S$ be a \CuSgp{} satisfying \axiomO{5}, \axiomO{6}, and Edwards' condition, and let $K$ be a closed subcone of $\FF(S)$ with $0\in K$.
Let $M\in (0,\infty)$. 
Assume that $\widehat{x}|_K\leq \widehat{y}|_K$ implies $\widehat{x}\leq M \widehat{y}$, for all $x,y\in S$. 
Then $K=\FF(S)$.
\end{thm}
\begin{proof}
By \autoref{prp:separationLFStoS}, $f|_K\leq g|_K$ implies $f\leq Mg$, for all $f,g\in \LL(\FF(S))$. 

Let $I$ be a countably generated ideal of $S$. 
By \autoref{prp:technical2}, we have $\lambda_I\in K$.
We claim that $K$ contains every extreme ray of the cone $\FF_I(S)$ (see \autoref{pgr:ray}). 
To this end, let $\mu\in \FF_I(S)\setminus \{\lambda_I\}$ be a functional generating an extreme ray of $\FF_I(S)$ and assume, for the sake of contradiction, that $\mu\notin K$.

Let $\sigma_\mu$ be as defined in \autoref{pgr:ray}.
If $\lambda\in K$ and $\lambda+\lambda_I=t\mu$ for some $t>0$, then this implies that $\mu\in K$, contrary to our assumption. 
Hence, by the definition of $\sigma_\mu$, we have $\sigma_\mu(\lambda)\in \{0,\infty\}$ for all $\lambda\in K$.
Put differently, $(M+1)\sigma_\mu|_K=\sigma_\mu|_K$. 

By \autoref{prp:chiselLFS}, there is an increasing net $(f_j)_j$ in $\LL(\FF(S))$ with supremum~$\sigma_\mu$ in $\Lsc(\FF(S))$.
Fix an index $j_0$, and let $h\in \LL(\FF(S))$ be such that $h\ll f_{j_0}$.
Then $(M+1)h\ll(M+1)f_{j_0}$.
Let us use $\lll$ to denote the non-sequential way-below relation.
Applying \autoref{prp:restrictWayBelowLFS} at the first step, we get
\[
(M+1)h|_K 
\lll (M+1)f_{j_0}|_K 
\leq (M+1)\sigma_\mu|_K 
= \sigma_\mu|_K 
= \sup_j f_j|_K
\]
in $\Lsc(K)$.
Hence, $(M+1)h|_K\leq f_{j}|_K$ for some $j$. 
It follows from our assumption on $K$  that $(M+1)h\leq Mf_j\leq M\sigma_\mu$. 
Evaluating both sides at $\mu$, and using that $\sigma_\mu(\mu)=1$, we get $(M+1)h(\mu)\leq M$. 
Since $\LL(\FF(S))$ is a \CuSgp{}, $f_{j_0}$ is the supremum of all $h\in \LL(\FF(S))$ satisfying $h\ll f_{j_0}$.
Passing to the supremum over all $h$ way-below $f_{j_0}$, we get $(M+1)f_{j_0}(\mu)\leq M$. 
Now passing to the supremum over all $j_0$ and using again that $\sigma_\mu(\mu)=1$ we get $M+1\leq M$. 
This is the desired contradiction. 

We have thus shown that $K$ contains every extreme ray of $\FF_I(S)$. 
Since $I$ is countably generated, we have by \cite[Proposition~3.11]{AntPerRobThi21Edwards} that $\FF_I(S)$ is well capped. 
Therefore, $K$ contains all of $\FF_I(S)$ by \cite[p.~81]{Phe01LNMChoquet}. 

As at the end of the proof of \autoref{prp:SubconeIsAll}, it now follows that $K=F(S)$.
\end{proof}


\begin{cor}
\label{prp:SubconeDense2}
Let $S$ be a \CuSgp{} satisfying \axiomO{5}, \axiomO{6}, and Edwards' condition. 
Let $K$ be a subcone of $\FF(S)$ with $0\in K$.
Let $M\in (0,\infty)$.
Suppose that for all $x,y',y\in S$ with $\widehat{x}|_K\leq \widehat{y'}|_K$ and $y'\ll y$, we have $\widehat{x}\leq M\widehat{y}$.
Then $K$ is dense in $\FF(S)$.
\end{cor}
\begin{proof}
By \autoref{prp:SubconeIsAll2}, it suffices to show that $\widehat{x}|_{\overline{K}}\leq \widehat{y}|_{\overline{K}}$ implies $\widehat{x}\leq 2M\widehat{y}$, for all $x,y\in S$. 
So let $x,y\in S$ satisfy $\widehat{x}|_{\overline{K}}\leq \widehat{y}|_{\overline{K}}$. 
Let $x'\in S$ satisfy $x'\ll x$.
By \autoref{prp:wayBelowLFS}, we get $\widehat{x'}\ll 2\widehat{x}$.
Applying \autoref{prp:restrictWayBelowLFS}, we obtain $\widehat{x'}|_{\overline{K}}\ll 2\widehat{y}|_{\overline{K}}$ in $\Lsc(\overline{K})$.
This allows us to choose $y'\in S$ such that
\[
\widehat{x'}|_{\overline{K}}\leq 2\widehat{y'}|_{\overline{K}}, \andSep
y'\ll y.
\]
By assumption, we get $\widehat{x'}\leq 2M\widehat{y}$. 
Passing to the supremum over all $x'$ such that $x'\ll x$, we obtain that $\widehat{x}\leq 2M\widehat{y}$, as desired.
\end{proof}


\section{Separation of normalized functionals}\label{sec:sepnormalized}

In this section we obtain a result on the separation of functionals similar to \autoref{prp:SubconeIsAll2}, but in the context of normalized functionals. 
This time we rely on standard tools from the theory of compact convex sets.

Recall that an element $x$ in a \CuSgp{} $S$ is called full if it generates $S$ as an ideal.

\begin{thm}
\label{MdensityFu}
Let $S$ be a \CuSgp{} satisfying \axiomO{5}. 
Let $u\in S$ be a full compact element, and let $K\subseteq \FF_u(S)$ be a closed convex subset.
Let $M\in (0,\infty)$. 
Suppose that for all $x,y\in S$ with $y$ full and with $\widehat{x}|_K \leq \widehat{y}|_K$, we have $\widehat{x}\leq M\widehat{y}$. 
Then $K = \FF_u(S)$.
\end{thm}
\begin{proof}
The proof of \autoref{prp:separationLFStoS} is easily adapted to show that, under the present hypotheses,  $f|_K\leq g|_K$ implies $f\leq Mg$ for all $f,g\in \LL(\FF(S))$ with $g$ full in $\LL(\FF(S))$.

We will show that $K$ contains every extreme point of $\FF_u(S)$. 
Then, by the Krein--Milman Theorem, it will follow that $K = \FF_u(S)$.  
Let $\mu\in \FF_u(S)$ be an extreme point, and define $\sigma_\mu$ as in \autoref{pgr:ray}.
Then $\sigma_\mu|_{\FF_u(S)}$ is a strictly positive, lower semicontinuous, affine function.
Applying \cite[Corollary~I.1.4]{Alf71CpctCvxSets}, we find a net of continuous, affine functions $(f_j)_j$ defined on~$\FF_u(S)$ and with supremum $\sigma_\mu|_{\FF_u(S)}$.
We can also arrange for the functions $f_j$ to be strictly positive.

By \cite[Proposition~6.9]{AntPerRobThi22CuntzSR1}, each function $f_j$ can be extended to a full function $\tilde{f}_j\in \LL(\FF(S))$. 
More explicitly, as shown in the proof \cite[Proposition~6.9]{AntPerRobThi22CuntzSR1}, we have
\[
\tilde{f}_j(\lambda)=\begin{cases}
\infty & \text{ if }\lambda(u)=\infty\\
\lambda(u)f_j(\frac{\lambda}{\lambda(u)}) & \text{ if }0<\lambda(u)<\infty\\
0 & \text{ if }\lambda(u)=0.
\end{cases}
\]

Since the functions $\tilde{f}_j$ are full, they are infinite on all $\lambda\in \FF(S)$ such that $\lambda(u)=\infty$.
It readily follows that  $(\tilde f_j)_j$ is an increasing net of functions in $\LL(\FF(S))$ with supremum $\sigma_\mu$. 

The rest of the argument is very similar to the proof of \autoref{prp:SubconeIsAll2}. 
We sketch it here: 
To reach a contradiction, assume that $\mu\notin K$. 
Fix an index $j_0$ and let $h \in \LL(\FF(S))$ satisfy $h \ll \tilde f_{j_0}$. 
Then use \autoref{prp:restrictWayBelowLFS} to find $j>j_0$ such that $(M+1)h|_K\leq \tilde f_j|_K$, and hence $(M+1)h\leq M\tilde f_j\leq M\sigma_\mu$. 
Passing to the supremum over all $h$ way-below $\tilde f_{j_0}$ and then over all $j_0$, we get $(M+1)\sigma_\mu\leq M\sigma_\mu$, which implies $M+1\leq M$ after evaluating at $\mu$, an absurdity.
\end{proof}


\begin{thebibliography}{BRT{\etalchar{+}}12}

\bibitem[Alf71]{Alf71CpctCvxSets}
\bgroup\scshape{}E.~M. Alfsen\egroup{}, \emph{Compact convex sets and boundary
  integrals}, Springer-Verlag, New York-Heidelberg, 1971, Ergebnisse der
  Mathematik und ihrer Grenzgebiete, Band 57.

\bibitem[APP18]{AntPerPet18PerfConditionsCu}
\bgroup\scshape{}R.~Antoine\egroup{}, \bgroup\scshape{}F.~Perera\egroup{}, and
  \bgroup\scshape{}H.~Petzka\egroup{}, Perforation conditions and almost
  algebraic order in {C}untz semigroups,  \emph{Proc. Roy. Soc. Edinburgh Sect.
  A} \textbf{148} (2018), 669--702.

\bibitem[APRT21]{AntPerRobThi21Edwards}
\bgroup\scshape{}R.~Antoine\egroup{}, \bgroup\scshape{}F.~Perera\egroup{},
  \bgroup\scshape{}L.~Robert\egroup{}, and \bgroup\scshape{}H.~Thiel\egroup{},
  Edwards' condition for quasitraces on \ca{s},  \emph{Proc. Roy. Soc.
  Edinburgh Sect. A} \textbf{151} (2021), 525--547.

\bibitem[APRT22]{AntPerRobThi22CuntzSR1}
\bgroup\scshape{}R.~Antoine\egroup{}, \bgroup\scshape{}F.~Perera\egroup{},
  \bgroup\scshape{}L.~Robert\egroup{}, and \bgroup\scshape{}H.~Thiel\egroup{},
  \ca{s} of stable rank one and their {C}untz semigroups,  \emph{Duke Math. J.}
  \textbf{171} (2022), 33--99.

\bibitem[APT18]{AntPerThi18TensorProdCu}
\bgroup\scshape{}R.~Antoine\egroup{}, \bgroup\scshape{}F.~Perera\egroup{}, and
  \bgroup\scshape{}H.~Thiel\egroup{}, Tensor products and regularity properties
  of {C}untz semigroups,  \emph{Mem. Amer. Math. Soc.} \textbf{251} (2018),
  viii+191.

\bibitem[APT20a]{AntPerThi20AbsBivariantCu}
\bgroup\scshape{}R.~Antoine\egroup{}, \bgroup\scshape{}F.~Perera\egroup{}, and
  \bgroup\scshape{}H.~Thiel\egroup{}, Abstract bivariant {C}untz semigroups,
  \emph{Int. Math. Res. Not. IMRN} (2020), 5342--5386.

\bibitem[APT20b]{AntPerThi20AbsBivarII}
\bgroup\scshape{}R.~Antoine\egroup{}, \bgroup\scshape{}F.~Perera\egroup{}, and
  \bgroup\scshape{}H.~Thiel\egroup{}, Abstract bivariant {C}untz semigroups
  {II},  \emph{Forum Math.} \textbf{32} (2020), 45--62.

\bibitem[APT20c]{AntPerThi20CuntzUltraproducts}
\bgroup\scshape{}R.~Antoine\egroup{}, \bgroup\scshape{}F.~Perera\egroup{}, and
  \bgroup\scshape{}H.~Thiel\egroup{}, Cuntz semigroups of ultraproduct \ca{s},
  \emph{J. Lond. Math. Soc. (2)} \textbf{102} (2020), 994--1029.

\bibitem[ART17]{ArcRobTik17Dixmier}
\bgroup\scshape{}R.~Archbold\egroup{}, \bgroup\scshape{}L.~Robert\egroup{}, and
  \bgroup\scshape{}A.~Tikuisis\egroup{}, The {D}ixmier property and tracial
  states for \ca{s},  \emph{J. Funct. Anal.} \textbf{273} (2017), 2655--2718.

\bibitem[BF15]{BicFar15Trace}
\bgroup\scshape{}T.~Bice\egroup{} and \bgroup\scshape{}I.~Farah\egroup{},
  Traces, ultrapowers and the {P}edersen-{P}etersen \ca{s},  \emph{Houston J.
  Math.} \textbf{41} (2015), 1175--1190.

\bibitem[BH82]{BlaHan82DimFct}
\bgroup\scshape{}B.~Blackadar\egroup{} and
  \bgroup\scshape{}D.~Handelman\egroup{}, Dimension functions and traces on
  \ca{s},  \emph{J. Funct. Anal.} \textbf{45} (1982), 297--340.

\bibitem[BRT{\etalchar{+}}12]{BlaRobTikTomWin12AlgRC}
\bgroup\scshape{}B.~Blackadar\egroup{}, \bgroup\scshape{}L.~Robert\egroup{},
  \bgroup\scshape{}A.~P. Tikuisis\egroup{}, \bgroup\scshape{}A.~S.
  Toms\egroup{}, and \bgroup\scshape{}W.~Winter\egroup{}, An algebraic approach
  to the radius of comparison,  \emph{Trans. Amer. Math. Soc.} \textbf{364}
  (2012), 3657--3674.

\bibitem[BR92]{BlaRor92ExtendStates}
\bgroup\scshape{}B.~Blackadar\egroup{} and
  \bgroup\scshape{}M.~R{\o}rdam\egroup{}, Extending states on preordered
  semigroups and the existence of quasitraces on \ca{s},  \emph{J. Algebra}
  \textbf{152} (1992), 240--247.

\bibitem[BBS{\etalchar{+}}19]{BBSTWW19}
\bgroup\scshape{}J.~Bosa\egroup{}, \bgroup\scshape{}N.~P. Brown\egroup{},
  \bgroup\scshape{}Y.~Sato\egroup{}, \bgroup\scshape{}A.~Tikuisis\egroup{},
  \bgroup\scshape{}S.~White\egroup{}, and \bgroup\scshape{}W.~Winter\egroup{},
  Covering dimension of \ca{s} and 2-coloured classification,  \emph{Mem. Amer.
  Math. Soc.} \textbf{257} (2019), vii+97.

\bibitem[BC09]{BroCiu09IsoHilbModSF}
\bgroup\scshape{}N.~P. Brown\egroup{} and
  \bgroup\scshape{}A.~Ciuperca\egroup{}, Isomorphism of {H}ilbert modules over
  stably finite \ca{s},  \emph{J. Funct. Anal.} \textbf{257} (2009), 332--339.

\bibitem[BPT08]{BroPerTom08CuElliottConj}
\bgroup\scshape{}N.~P. Brown\egroup{}, \bgroup\scshape{}F.~Perera\egroup{}, and
  \bgroup\scshape{}A.~S. Toms\egroup{}, The {C}untz semigroup, the {E}lliott
  conjecture, and dimension functions on \ca{s},  \emph{J. Reine Angew. Math.}
  \textbf{621} (2008), 191--211.

\bibitem[CRS10]{CiuRobSan10CuIdealsQuot}
\bgroup\scshape{}A.~Ciuperca\egroup{}, \bgroup\scshape{}L.~Robert\egroup{}, and
  \bgroup\scshape{}L.~Santiago\egroup{}, The {C}untz semigroup of ideals and
  quotients and a generalized {K}asparov stabilization theorem,  \emph{J.
  Operator Theory} \textbf{64} (2010), 155--169.

\bibitem[Con76]{Con76ClassifInjFactors}
\bgroup\scshape{}A.~Connes\egroup{}, Classification of injective factors.
  {C}ases {$II\sb{1},$} {$II\sb{\infty },$} {$III\sb{\lambda },$} {$\lambda
  \not=1$},  \emph{Ann. of Math. (2)} \textbf{104} (1976), 73--115.

\bibitem[CEI08]{CowEllIva08CuInv}
\bgroup\scshape{}K.~T. Coward\egroup{}, \bgroup\scshape{}G.~A.
  Elliott\egroup{}, and \bgroup\scshape{}C.~Ivanescu\egroup{}, The {C}untz
  semigroup as an invariant for \ca{s},  \emph{J.\ Reine Angew.\ Math.}
  \textbf{623} (2008), 161--193.

\bibitem[Cun78]{Cun78DimFct}
\bgroup\scshape{}J.~Cuntz\egroup{}, Dimension functions on simple \ca{s},
  \emph{Math. Ann.} \textbf{233} (1978), 145--153.

\bibitem[CP79]{CunPed79EquivTraces}
\bgroup\scshape{}J.~Cuntz\egroup{} and \bgroup\scshape{}G.~K.
  Pedersen\egroup{}, Equivalence and traces on \ca{s},  \emph{J. Funct. Anal.}
  \textbf{33} (1979), 135--164.

\bibitem[ERS11]{EllRobSan11Cone}
\bgroup\scshape{}G.~A. Elliott\egroup{}, \bgroup\scshape{}L.~Robert\egroup{},
  and \bgroup\scshape{}L.~Santiago\egroup{}, The cone of lower semicontinuous
  traces on a \ca{},  \emph{Amer. J. Math.} \textbf{133} (2011), 969--1005.

\bibitem[FHL{\etalchar{+}}21]{FarHarLupRobTikVigWin21ModelThy}
\bgroup\scshape{}I.~Farah\egroup{}, \bgroup\scshape{}B.~Hart\egroup{},
  \bgroup\scshape{}M.~Lupini\egroup{}, \bgroup\scshape{}L.~Robert\egroup{},
  \bgroup\scshape{}A.~Tikuisis\egroup{}, \bgroup\scshape{}A.~Vignati\egroup{},
  and \bgroup\scshape{}W.~Winter\egroup{}, Model theory of \ca{s},  \emph{Mem.
  Amer. Math. Soc.} \textbf{271} (2021), viii+127.

\bibitem[GKL19]{GarKalLup19ModelThyRokhlinDim}
\bgroup\scshape{}E.~Gardella\egroup{}, \bgroup\scshape{}M.~Kalantar\egroup{},
  and \bgroup\scshape{}M.~Lupini\egroup{}, Model theory and {R}okhlin dimension
  for compact quantum group actions,  \emph{J. Noncommut. Geom.} \textbf{13}
  (2019), 711--767.

\bibitem[GL18]{GarLup18ApplModelThyCDynamics}
\bgroup\scshape{}E.~Gardella\egroup{} and \bgroup\scshape{}M.~Lupini\egroup{},
  Applications of model theory to {${\rm C}^*$}-dynamics,  \emph{J. Funct.
  Anal.} \textbf{275} (2018), 1889--1942.

\bibitem[GP23]{GarPer23arX:ModernCu}
\bgroup\scshape{}E.~Gardella\egroup{} and \bgroup\scshape{}F.~Perera\egroup{},
  The modern theory of {C}untz semigroups of \ca{s}, preprint (arXiv:2212.02290
  [math.OA]), 2023.

\bibitem[GT23]{GarThi23arX:GenByCommutators}
\bgroup\scshape{}E.~Gardella\egroup{} and \bgroup\scshape{}H.~Thiel\egroup{},
  Rings and \ca{s} generated by commutators, preprint (arXiv:2301.05958
  [math.RA]), 2023.

\bibitem[Gro73]{Gro73TVS}
\bgroup\scshape{}A.~Grothendieck\egroup{}, \emph{Topological vector spaces},
  Gordon and Breach Science Publishers, New York-London-Paris, 1973, Translated
  from the French by Orlando Chaljub, Notes on Mathematics and its
  Applications.

\bibitem[Haa14]{Haa14Quasitraces}
\bgroup\scshape{}U.~Haagerup\egroup{}, Quasitraces on exact \ca{s} are traces,
  \emph{C. R. Math. Acad. Sci. Soc. R. Can.} \textbf{36} (2014), 67--92.

\bibitem[{Kei}17]{Kei17CuSgpDomainThy}
\bgroup\scshape{}K.~{Keimel}\egroup{}, The {C}untz semigroup and domain theory,
   \emph{Soft Comput.} \textbf{21} (2017), 2485--2502.

\bibitem[KR14]{KirRor14CentralSeq}
\bgroup\scshape{}E.~Kirchberg\egroup{} and
  \bgroup\scshape{}M.~R{\o}rdam\egroup{}, Central sequence \ca{s} and tensorial
  absorption of the {J}iang-{S}u algebra,  \emph{J. Reine Angew. Math.}
  \textbf{695} (2014), 175--214.

\bibitem[McD69]{McD69UncountablyManyII1}
\bgroup\scshape{}D.~McDuff\egroup{}, Uncountably many {$\mathrm{II}_1$}
  factors,  \emph{Ann. of Math. (2)} \textbf{90} (1969), 372--377.

\bibitem[NR16]{NgRob16CommutatorsPureCa}
\bgroup\scshape{}P.~W. Ng\egroup{} and \bgroup\scshape{}L.~Robert\egroup{},
  Sums of commutators in pure \ca{s},  \emph{M\"{u}nster J. Math.} \textbf{9}
  (2016), 121--154.

\bibitem[OPR12]{OrtPerRor12CoronaStability}
\bgroup\scshape{}E.~Ortega\egroup{}, \bgroup\scshape{}F.~Perera\egroup{}, and
  \bgroup\scshape{}M.~R{\o}rdam\egroup{}, The corona factorization property,
  stability, and the {C}untz semigroup of a \ca{},  \emph{Int. Math. Res. Not.
  IMRN} (2012), 34--66.

\bibitem[Oza13]{Oza13DixmApproxSymmAmen}
\bgroup\scshape{}N.~Ozawa\egroup{}, Dixmier approximation and symmetric
  amenability for \ca{s},  \emph{J. Math. Sci. Univ. Tokyo} \textbf{20} (2013),
  349--374.

\bibitem[Ped79]{Ped79CAlgsAutGp}
\bgroup\scshape{}G.~K. Pedersen\egroup{}, \emph{\ca{s} and their automorphism
  groups}, \emph{London Mathematical Society Monographs} \textbf{14}, Academic
  Press, Inc. [Harcourt Brace Jovanovich, Publishers], London-New York, 1979.

\bibitem[Phe01]{Phe01LNMChoquet}
\bgroup\scshape{}R.~R. Phelps\egroup{}, \emph{Lectures on {C}hoquet's theorem},
  second ed., \emph{Lecture Notes in Mathematics} \textbf{1757},
  Springer-Verlag, Berlin, 2001.

\bibitem[Rob11]{Rob11NuclDimComp}
\bgroup\scshape{}L.~Robert\egroup{}, Nuclear dimension and {$n$}-comparison,
  \emph{M\"{u}nster J. Math.} \textbf{4} (2011), 65--71.

\bibitem[Rob13]{Rob13Cone}
\bgroup\scshape{}L.~Robert\egroup{}, The cone of functionals on the {C}untz
  semigroup,  \emph{Math. Scand.} \textbf{113} (2013), 161--186.

\bibitem[Rob15]{Rob15NucDimSumsCommutators}
\bgroup\scshape{}L.~Robert\egroup{}, Nuclear dimension and sums of commutators,
   \emph{Indiana Univ. Math. J.} \textbf{64} (2015), 559--576.

\bibitem[Rob16]{Rob16LieIdeals}
\bgroup\scshape{}L.~Robert\egroup{}, On the {L}ie ideals of \ca{s},  \emph{J.
  Operator Theory} \textbf{75} (2016), 387--408.

\bibitem[RT17]{RobTik17NucDimNonSimple}
\bgroup\scshape{}L.~Robert\egroup{} and \bgroup\scshape{}A.~Tikuisis\egroup{},
  Nuclear dimension and {$\mathcal{Z}$}-stability of non-simple \ca{s},
  \emph{Trans. Amer. Math. Soc.} \textbf{369} (2017), 4631--4670.

\bibitem[R{\o}r04]{Ror04StableRealRankZ}
\bgroup\scshape{}M.~R{\o}rdam\egroup{}, The stable and the real rank of
  {$\mathcal{Z}$}-absorbing \ca{s},  \emph{Internat. J. Math.} \textbf{15}
  (2004), 1065--1084.

\bibitem[RW10]{RorWin10ZRevisited}
\bgroup\scshape{}M.~R{\o}rdam\egroup{} and \bgroup\scshape{}W.~Winter\egroup{},
  The {J}iang-{S}u algebra revisited,  \emph{J. Reine Angew. Math.}
  \textbf{642} (2010), 129--155.

\bibitem[dS16]{Sil16Thesis}
\bgroup\scshape{}K.~de~Silva\egroup{}, \emph{Rank constrained homotopies of
  matrices and the {B}lackadar-{H}andelman conjectures on {C}*-algebras},
  ProQuest LLC, Ann Arbor, MI, 2016, Thesis (Ph.D.)--Purdue University.

\bibitem[Thi20]{Thi20RksOps}
\bgroup\scshape{}H.~Thiel\egroup{}, Ranks of operators in simple \ca{s} with
  stable rank one,  \emph{Comm. Math. Phys.} \textbf{377} (2020), 37--76.

\bibitem[TV21a]{ThiVil21DimCu2}
\bgroup\scshape{}H.~Thiel\egroup{} and \bgroup\scshape{}E.~Vilalta\egroup{},
  Covering dimension of {C}untz semigroups {II},  \emph{Internat. J. Math.}
  \textbf{32} (2021), 27~p., Paper No. 2150100.

\bibitem[TV21b]{ThiVil21arX:NowhereScattered}
\bgroup\scshape{}H.~Thiel\egroup{} and \bgroup\scshape{}E.~Vilalta\egroup{},
  Nowhere scattered \ca{s}, J. Noncommut. Geom. (to appear), preprint
  (arXiv:2112.09877 [math.OA]), 2021.

\bibitem[TV22a]{ThiVil22DimCu}
\bgroup\scshape{}H.~Thiel\egroup{} and \bgroup\scshape{}E.~Vilalta\egroup{},
  Covering dimension of {C}untz semigroups,  \emph{Adv. Math.} \textbf{394}
  (2022), 44~p., Article No.~108016.

\bibitem[TV22b]{ThiVil22arX:Glimm}
\bgroup\scshape{}H.~Thiel\egroup{} and \bgroup\scshape{}E.~Vilalta\egroup{},
  The {G}lobal {G}limm {P}roperty, Trans. Amer. Math. Soc. (to appear),
  preprint (arXiv:2204.13059 [math.OA]), 2022.

\bibitem[Tho95]{Tho95TracesCrProdZ}
\bgroup\scshape{}K.~Thomsen\egroup{}, Traces, unitary characters and crossed
  products by {${\bf Z}$},  \emph{Publ. Res. Inst. Math. Sci.} \textbf{31}
  (1995), 1011--1029.

\bibitem[Tik14]{Tik14NucDimZstable}
\bgroup\scshape{}A.~Tikuisis\egroup{}, Nuclear dimension,
  {$\mathcal{Z}$}-stability, and algebraic simplicity for stably projectionless
  \ca{s},  \emph{Math. Ann.} \textbf{358} (2014), 729--778.

\bibitem[TWW17]{TikWhiWin17QDNuclear}
\bgroup\scshape{}A.~Tikuisis\egroup{}, \bgroup\scshape{}S.~White\egroup{}, and
  \bgroup\scshape{}W.~Winter\egroup{}, Quasidiagonality of nuclear \ca{s},
  \emph{Ann. of Math. (2)} \textbf{185} (2017), 229--284.

\bibitem[Vac23]{Vac23arX:UltraprodWBundles}
\bgroup\scshape{}A.~Vaccaro\egroup{}, Ultraproducts of factorial
  {$W^*$}-bundles, preprint (arXiv:2303.01942 [math.OA]), 2023.

\bibitem[Win12]{Win12NuclDimZstable}
\bgroup\scshape{}W.~Winter\egroup{}, Nuclear dimension and
  {$\mathcal{Z}$}-stability of pure \ca{s},  \emph{Invent. Math.} \textbf{187}
  (2012), 259--342.

\end{thebibliography}

\providecommand{\etalchar}[1]{$^{#1}$}
\providecommand{\bysame}{\leavevmode\hbox to3em{\hrulefill}\thinspace}
\providecommand{\noopsort}[1]{}
\providecommand{\mr}[1]{\href{http://www.ams.org/mathscinet-getitem?mr=#1}{MR~#1}}
\providecommand{\zbl}[1]{\href{http://www.zentralblatt-math.org/zmath/en/search/?q=an:#1}{Zbl~#1}}
\providecommand{\jfm}[1]{\href{http://www.emis.de/cgi-bin/JFM-item?#1}{JFM~#1}}
\providecommand{\arxiv}[1]{\href{http://www.arxiv.org/abs/#1}{arXiv~#1}}
\providecommand{\doi}[1]{\url{http://dx.doi.org/#1}}
\providecommand{\MR}{\relax\ifhmode\unskip\space\fi MR }
\providecommand{\MRhref}[2]{%
  \href{http://www.ams.org/mathscinet-getitem?mr=#1}{#2}
}
\providecommand{\href}[2]{#2}

\end{document}